\documentclass[12pt]{article}

\usepackage[english]{babel}
\usepackage[utf8]{inputenc}
\usepackage{amsfonts,mathrsfs, stmaryrd}
\usepackage{graphicx}
\usepackage{amsmath, mathtools}
\usepackage{amsthm}
\usepackage{amssymb,fourier}
\usepackage{xparse}
\usepackage{color}
\usepackage[colorlinks=true,citecolor=blue]{hyperref}
\usepackage[
    backend=biber,
    style=alphabetic,
    sortlocale=auto,
    natbib=true,
    url=false, 
    doi=true,
    eprint=true,
    safeinputenc=true,
    maxbibnames=10
]{biblatex}
\usepackage{enumitem}
\usepackage[]{todonotes} %colorinlistoftodos %disable

\usepackage{pgf}
\usepackage{tikz}
\usetikzlibrary{arrows,automata}

    \newtheorem{theorem}{Theorem}
    \newtheorem{lemma}[theorem]{Lemma}
    \newtheorem{proposition}[theorem]{Proposition}
    \newtheorem{corollary}[theorem]{Corollary}

    \newtheorem{assumption}[theorem]{Assumption}

\theoremstyle{definition} % For roman text in the body
    \newtheorem{definition}[theorem]{Definition}
    \newtheorem{remark}[theorem]{Remark}

\newcommand{\var}{\operatorname{Var}}
\renewcommand{\P}[1]{\mathds{P}\left[#1\right]}
\newcommand{\E}[1]{\mathds{E}\left[#1\right]}

\renewcommand{\i}{\mathbf{i}}

\newcommand{\1}{\mathbf{1}}
\newcommand{\A}{\mathcal{A}}
\newcommand{\B}{\mathcal{B}}
\renewcommand{\d}{\mathrm{d}}
\newcommand{\C}{\mathbb{C}}

\renewcommand{\E}{\mathbb{E}}
\newcommand{\D}{\mathbb{D}}
\newcommand{\F}{\mathscr{F}}

\newcommand{\N}{\mathbb{N}}
\renewcommand{\O}{\mathcal{O}}
\renewcommand{\P}{\mathbb{P}}
\newcommand{\R}{\mathbb{R}}

\newcommand{\T}{\mathds{T}}
\newcommand{\U}{\mathscr{U}}
\renewcommand{\u}{\mathbf{U}}
\newcommand{\x}{\boldsymbol{\xi}}
\newcommand{\X}{\mathbf{X}}
\newcommand{\Z}{\mathbb{Z}}

\newcommand{\Id}{\operatorname{Id}}
\newcommand{\diag}{\operatorname{diag}}

\numberwithin{equation}{section}
\numberwithin{theorem}{section}

\newcommand{\Exp}{\mathbb{E}}
\newcommand{\prob}{\mathbb{P}}
\newcommand{\filt}{\mathscr{F}}

\renewcommand{\Pr}{\prob}
\DeclareDocumentCommand \one { o }
{%
\IfNoValueTF {#1}
{\mathbf{1}  }
{\mathbf{1}\left\{ {#1} \right\} }%
}

\newcommand{\Var}{\operatorname{Var}}

\newcommand{\Beta}{\operatorname{Beta}}

\newcommand{\Unif}{\operatorname{Unif}}
\newcommand{\As}{\operatorname{a.s.}}

\DeclareDocumentCommand{\Prto} {o} {
\IfNoValueTF {#1}
 {\overset{\Pr}{\longrightarrow}}
 { \xrightarrow[ #1 \to \infty]{\Pr }}
}
\DeclareDocumentCommand{\Asto} {o} {
\IfNoValueTF {#1}
 {\overset{\operatorname{a.s.}}{\longrightarrow}}
 {
 \xrightarrow[ #1 \to \infty]{\operatorname{a.s.} }
 }
}
\DeclareDocumentCommand{\Mgfto} {o} {
\IfNoValueTF {#1}
{\overset{\operatorname{mgf}}{\longrightarrow}}
{ \xrightarrow[ #1 \to \infty]{\operatorname{mgf} }}
}

\DeclareDocumentCommand{\Wkto} {o} {
\IfNoValueTF {#1}
 {\overset{(d)}{\longrightarrow}}
 { \xrightarrow[ #1 \to \infty]{(d) }}
}

\DeclareDocumentCommand \LPto { O{1} }
{\overset{\operatorname{\LP^{#1}}}{\longrightarrow}}

\title{Strong approximation of Gaussian $\beta$ ensemble characteristic polynomials: the hyperbolic regime}
\date{\today}
\author{Gaultier Lambert\footnote{
University of Zurich, Winterthurerstrasse 190, 8057 Z\"urich, Switzerland. 
\newline G.L. research is supported  by the SNSF Ambizione grant S-71114-05-01.
\newline Email: \href{mailto:gaultier.lambert@math.uzh.ch}{\nolinkurl{gaultier.lambert@math.uzh.ch}}} 
 and Elliot Paquette\footnote{
 McGill University, 805 Sherbrooke Street West, Montreal, Quebec, Canada.
 \newline Supported by Simons Foundation travel grant 638152.
 \newline Email: \href{mailto:elliot.paquette@mcgill.ca}{\nolinkurl{elliot.paquette@mcgill.ca}}
 }}

\begin{document}

\maketitle

\begin{abstract}
We investigate the characteristic polynomials $\varphi_N$ of the Gaussian $\beta$-ensemble for general $\beta>0$ through its transfer matrix recurrence. 
Our motivation is to obtain a (probabilistic) approximation for  $\varphi_N$ in terms of a Gaussian log--correlated field. We distinguish between different types of transfer matrices and analyze completely the  \emph{hyperbolic part} of the recurrence.
As a result, we obtain a new coupling between $\varphi_N$ and a Gaussian analytic function with an error which is uniform away from the support of the semicircle law. 
We use this as input to give the almost sure scaling limit of the characteristic polynomial at the edge in \cite{LambertPaquette03}.
This is also required to obtain analogous strong approximations inside of the bulk of the semicircle law. 
Our analysis relies on moderate deviation estimates for the product of transfer matrices and this approach might also be useful in different contexts.
\end{abstract}
{
  \hypersetup{colorlinks=true,linkcolor=black}  
\tableofcontents
}

%\paragraph{Keywords:} Gaussian $\beta$-ensembles, product of random matrices, Gaussian Multiplicative Chaos, 

\section{Introduction}

In this article, we will develop new properties of $N$--dimensional  \emph{Gaussian $\beta$-ensembles}, or G$\beta$E, and extend its known connections to \emph{log--correlated fields}.  For $\beta>0$, the G$\beta$E is the $N$--point process on $\R$ with joint law
\begin{equation}\label{eq:GbE}
d\mu_{G\beta E}(\lambda_1, \lambda_2, \dots,\lambda_N)
=
\frac{1}{\mathcal{Z}_{N,\beta}} e^{- \sum_{i=1}^N \beta N \lambda_i^2} \prod_{i > j} \left|\lambda_i - \lambda_j\right|^\beta , 
\end{equation}
where $\mathcal{Z}_{N,\beta}>0$ is a normalizing constant.
In this scaling, the limiting spectral distribution as $N\to\infty$ is a semicircle on $[-1,1].$
In terms of these points, we define the characteristic polynomial 
\[
  \varphi_N(z) = {\textstyle \prod_{i=1}^N} ( z - \lambda_i)  , \qquad z\in\C. 
\]
It is well--known that $\log|\varphi_N(z)|-\Exp \log|\varphi_N(z)|$ converges weakly to a Gaussian field $X(z)$ at fixed $z\in \C\setminus[-1,1]$ as it follows e.g.~from Johansson classical central limit Theorem~\ref{thm:clt}. 
In particular, as a consequence of Theorem~\ref{thm:planar} below
this can be extended to process convergence in the sense of locally uniform convergence on $\C \setminus [-1,1]$ (in the notation of Theorem~\ref{thm:planar},  $X= \sqrt{2/\beta} \Re \mathrm{W}$).

\medskip

An important feature of the characteristic polynomial is that the boundary values of $X$ on $[-1,1]$ form a log--correlated Gaussian field 
\[
  \Exp\big[ X(x)X(y) \big] = -\frac{\log|2(x-y)|^{-1}}{\beta},
  \qquad
  \text{for all } 
  x,y \in [-1,1]
\]
cf~\eqref{logcov}.
Hence, this field is not pointwise defined on $[-1,1]$ and its boundary values must be understood in an appropriate functional sense.  By an approximation argument (\cite{Kahane,Berestycki,Shamov}), it is however possible to define the exponential of $X(z)$ as a family of Gaussian multiplicative chaos measures $\big\{M_\gamma(dx) : \gamma \in [0, \sqrt{2\beta})\big\}.$ A natural approximation scheme is to take a limit from the upper half-plane, i.e.\ to consider the in-probability weak limit of measures
  \[
    M_\gamma(dx) = \lim_{\epsilon \to 0} e^{ \gamma X(x+i\epsilon) - \frac{\gamma^2}{2} \E X^2(x+i\epsilon)}dx.
  \]
%, see \cite{Berestycki, Shamov} for further references.
%Note that since $X$ is not defined pointwise on $[-1,1]$, the construction of these measures relies on introducing a smooth Gaussian approximation $X_n$ and  let $M_\gamma(dx) = \lim e^{ \gamma X_n(x) - \frac{\gamma^2}{2} \E X_n^2(x) }dx$ as $n\to+\infty$ in an appropriate sense.
Many other methods of approximation can be shown to yield the same limit, suggesting $M_\gamma$ as an unambiguous representation of the exponential of $X(z)$ on $[-1,1]$ for $\gamma \in [0, \sqrt{2\beta}).$

%These limits turn out to be independent of the approximation procedure and one can for instance define  $X_n$  
%by convolving $X$ with an approximate $\delta$--function or using the Karhunen--Lo\`eve expansion (see Kahane \cite{Kahane}). 
  The value $\gamma = \sqrt{2\beta}$ is critical, and a further renormalization is needed to produce a nondegenerate limit.  Here too there are senses in which the limit is independent of the method of approximation \cite{JS17,Junnila18}.   For the supercritical cases $\gamma > \sqrt{2\beta},$ it is no longer possible to consider in-probability weak limits, but there are various senses of weak limits in law that can be considered (see e.g.\ \cite{MRV16}).

\medskip

This naturally motivates the analytic question if $|\varphi_N(z)|^\gamma/\Exp|\varphi_N(z)|^\gamma$ converges to a Gaussian multiplicative chaos in a suitable sense. 
That is, can we use the characteristic polynomial of  the Gaussian $\beta$-ensemble as a smooth, and in some sense finite, approximation of a family of Gaussian multiplicative chaos measures. 
 In the special case of $\beta=2,$ relying on the determinantal structure and the associated othogonal polynomials Riemann--Hilbert problems, \cite{BWW} obtain such a weak convergence for $0 \leq \gamma < \sqrt{2},$ the so-called $L^2$--phase, and \cite{CFLW} in the whole subcritical--phase $\gamma <2$. 
Analogous results also exist for the characteristic polynomial of random unitary matrices (also known as  circular unitary ensemble) \cite{WebbL2, Nikula}.  
For circular $\beta$-ensembles for general $\beta>0$, convergence to the GMC has been shown by the first author in the subcritical phase for a regularized version of the characteristic polynomial (performed by looking on a circle inscribed in the unit disk by $\Theta( (\log N)^6/N)$)  \cite{Lambert}.  Moreover the spectral mesures of the \emph{CMV representation} of the infinite circular $\beta$-ensemble are (up to normalization) also given exactly by Gaussian multiplicative chaos measures \cite{CN19, AN19} when $\beta \ge 2$. 

\medskip

One can also ask if other log--correlated field predictions hold for $\varphi_N(z),$ such as the behavior of its extreme values. For the Gaussian unitary ensemble ($\beta=2),$ \cite{FyodorovSimm} conjecture an exact convergence in law for these extreme values.  In addition, the law of the position of the maximum is characterized by \cite{FyodorovDoussal}.  For $\beta=2,$ the authors show convergence of the leading term of the maximum of the modulus of the recentered log--characteristic polynomial \cite{LambertPaquette01}.  Again for $\beta=2,$ similar theorems are proven for the behavior of the maximum of the recentered argument of the characteristic polynomial \cite{CFLW}.
For the circular ensembles, the state of the art is substantially better developed \cite{ArguinBeliusBourgade,PaquetteZeitouni02,CMN}. % including work on general $\beta$ \cite{ChhaibiMadauleNajnudel}.
Indeed, \cite{CMN} obtains the precise behavior of the maximum for general $\beta$ based on the analysis of \emph{Pr\"ufer phases}, which give an effective Markov process description for the log--characteristic polynomial of a circular $\beta$-ensemble.  %Inspired by that work, we look to develop on a related Markovian structure for the Gaussian $\beta$-ensemble {\red (which has been very influential in the study of local limits of $1$--dimensional $\beta$-ensembles \cite{ValkoVirag, RiderVirag, RamirezRider,ValkoViragOp}).}

\subsection{Transfer-matrix recurrence} \label{sect:TM}

Recall that for any $\alpha>0$, a $\chi_{\alpha}$ random variable has density proportional to
$x^{\alpha -1} e^{- x^2 /2} \1_{x>0}$ and we have $\chi_{\alpha}^2 \sim \Gamma(\frac\alpha2, 2)$ where $\Gamma(\frac\alpha2, 2)$ denotes a Gamma distribution with shape $\frac\alpha2$ and rate $\frac 12$.  
In terms of these variables, we define the semi-infinite tridiagonal matrix 
\begin{equation} \label{def:trimatrix}
\mathbf{A} =
\left[ \begin{array}{cccc} 
b_1 & a_1 & &\\
a_1 & b_2 & a_2 & \\
& a_2 & b_3 & \ddots  \\
&& \ddots & \ddots
\end{array} \right],
\end{equation}
where $b_i \sim \mathcal{N}(0,2)$ and  $a_i \sim \chi_{\beta i}$ are independent random variables. 
By \cite{DumitriuEdelman}, the eigenvalues of the principal $N \times N$ minor of  the random matrix 
$ \mathbf{A}/\sqrt{4N\beta}$ 
have the same law as the Gaussian $\beta$-ensemble, \eqref{eq:GbE}, and so in particular $\varphi_N(z) = \det([z-({4}{N\beta})^{-1/2}\mathbf{A}]_{N,N}).$
%which is scaled so that its limiting density of states is given by the semicircle law on $[-1,1]$. 

\medskip

We let $\Phi_n(z)= \det([z-({4}{N\beta})^{-1/2}\mathbf{A}]_{n,n})$ for any $n\in\N.$
%(i.e. inside the spectrum) and $z$ in the upper--half plane. 
%Our analysis relies on the fact that $\Phi_n$ satisfies the following recurrence 
By cofactor expanding the $n$--th column of this determinant, we are led to the following recurrence for any integer $n \geq 2$, 
\begin{equation*}
 \left[ \begin{array}{c} 
\Phi_{n}(z) \\ \Phi_{n-1}(z) \end{array}\right]
= \left[ \begin{array}{cc}z-\frac{b_n}{{2}\sqrt{N\beta}} & -\frac{ a_{n-1}^2}{{4}N\beta}  \\
 1 & 0 \end{array} \right]
\left[ \begin{array}{c}
\Phi_{n-1}(z)\\
\Phi_{n-2}(z)
\end{array}\right] 
\eqqcolon T_n(z) \left[ \begin{array}{c}
\Phi_{n-1}(z)\\
\Phi_{n-2}(z)
\end{array}\right] ,
\end{equation*}
where by convention $\Phi_0 = 1$ and  
$ \Phi_1(z) = z- \frac{b_1}{ 2\sqrt{N\beta}}$. 
This shows that for any $n\ge 1$,
\begin{equation} \label{eq:recurrence}
\begin{pmatrix} \Phi_{n}(z) \\ \Phi_{n-1}(z) \end{pmatrix}
= T_n(z)  \cdots T_2(z)  \begin{pmatrix}   z- \frac{b_1}{ 2\sqrt{N\beta}}   \\ 1\end{pmatrix} . 
\end{equation}

A similar matrix recurrence, the Szeg\H{o} recurrence, can be posed for the circular $\beta$-ensemble (\cite[ Equation (2.1)]{CMN}):
\begin{equation}
  \left[ \begin{array}{c} 
  \Theta_{n+1}(z) \\ \Theta_{n+1}(z)^* \end{array}\right]
  = \left[ 
    \begin{array}{cc}
      z & -\overline{\alpha_n}  \\
      -\alpha_n z & 1
    \end{array} 
    \right]
    \left[ \begin{array}{c}
      \Theta_{n}(z)\\
      \Theta_{n}(z)^*
  \end{array}\right],
  \quad
  \begin{aligned}
    &|\alpha_n|^2 \sim \Beta(1,\beta(n+1)/2), \\
    &\arg(\alpha_n) \sim \Unif( [0,2\pi]),
  \end{aligned}
  \label{eq:Szego}
\end{equation}
and all random variables are independent.
%and \eqref{eq:recurrence} bears some similarities to it. 
We expect that \eqref{eq:recurrence} has the potential to give the same type of precise information on the statistics of the characteristic polynomial as (implicitly) \eqref{eq:Szego} gave for \cite{CMN}.  We further expect \eqref{eq:recurrence} to be useful in giving the needed estimates for showing GMC convergence as well as the asymptotics of the maximum of the characteristic polynomial.  It should be noted it is also possible to define Pr\"ufer phases for the Gaussian $\beta$-ensemble \cite[ Equation (1.170)]{Forrester}), in much the same way as it is possible to define Pr\"ufer phases for \eqref{eq:Szego}, but we do not believe analyzing the Pr\"ufer phase recurrence for the Gaussian $\beta$-ensemble  would be appreciably easier than analyzing the transfer matrix recurrence.

\subsection{Hermite recurrence}
\label{sec:Hermite}

There is however a \emph{major} phenomenological difference between \eqref{eq:recurrence} and the Szeg\H{o} recurrence for the circular $\beta$-ensemble.  The Gaussian $\beta$-ensemble recurrence can have three $z$-dependent regimes of $n,$ each of which has a substantially different dynamical behavior.  To illustrate this, we consider the Hermite recurrence.  Define for $n \ge 2$, 
\begin{equation}\label{eq:hermite}
%\label{eq:Ttilde}
\widetilde T_{n}= \E T_n = \left[\begin{array}{cc} z &- \frac{n-1}{{ 4}N} \\  1 & 0  \end{array}\right] 
%\end{align}
%Importantly, if we define for $n\ge 1$
\quad\text{and}\quad
  \begin{pmatrix}
    \pi_n(z) \\
    \pi_{n-1}(z) 
  \end{pmatrix}
  =
  \widetilde T_{n}
  \widetilde T_{n-1}
  \cdots
  \widetilde T_{2}
  \begin{pmatrix}
  z \\ 1
\end{pmatrix}.
\end{equation}
Then $\{\pi_n\}$ are exactly the monic Hermite polynomials scaled to be orthogonal with respect to the weight $e^{-2N x^2}$ on $\R$.
In particular, it follows from \eqref{eq:hermite} that with our conventions:
$\E \Phi_{n} = \pi_n$ for any $n\in\N$.
Let us record how to diagonalize the matrices $\widetilde T_k.$ 
\begin{lemma}
  \label{lem:diagonalization}
  For any $z \in \C$ and $t >0$ with $z^2 \neq t,$
  \[
    \begin{aligned}
    &\begin{bmatrix}
      z & - \frac{t}{ 4} \\
      1 & 0 \\
    \end{bmatrix}
    =
    \begin{bmatrix}
      \lambda_+ & \lambda_{-} \\
      1 & 1 \\
    \end{bmatrix}
    \begin{bmatrix}
      \lambda_+ & 0 \\
      0 & \lambda_- \\
    \end{bmatrix}
    \begin{bmatrix}
      \lambda_+ & \lambda_{-}  \\
      1 & 1 \\
    \end{bmatrix}^{-1},
    \quad\text{where}\quad
    &\begin{bmatrix}
      \lambda_+ & \lambda_{-} \\
      1 & 1 \\
    \end{bmatrix}^{-1}
   =
    \begin{bmatrix}
      1 &-\lambda_{-} \\
      -1  & \lambda_{+}
    \end{bmatrix}\frac{1}{\lambda_+ - \lambda_{-}},
  \end{aligned}
  \]
  with
  \begin{equation}  \label{def:root}
  \lambda_{\pm}(t) = \frac{z \pm \sqrt{z^2-t}}{2},
  \end{equation}
  and where we take the convention here that the branch behaves like $z$ at $\infty$ so that $|\lambda_+| \geq |\lambda_{-}|.$
  %and throughout that $\sqrt{\cdot}$ is the principal branch.
  \end{lemma}

Let us observe that this operation becomes singular when $z \approx \pm\sqrt{t}$ which corresponds to a \emph{turning point} in the recursions \eqref{eq:hermite} as well as \eqref{eq:recurrence} where the transfer matrix develops a nontrivial Jordan form.
Therefore, for any fixed $z \in [-1,1],$ we need to distinguish three different regimes:
\begin{itemize}
\item[(i)] For $n \ll N z^2$, the eigenvalues of $\widetilde T_n$ are real--valued and have distinct modulus.
  Such $2\times 2$ real matrices are called \emph{hyperbolic} as the mapping $x \mapsto {\widetilde T_n x}$ has two fixed points on real projective space (corresponding to the two real eigenvectors of $\widetilde T_n$).
  Moreover, by Lemma~\ref{lem:diagonalization}, the matrices $\widetilde T_n$ have slowly varying eigenspaces in $n$, so that the expanding and contracting directions along this recurrence remain nearly aligned as $n$ varies.
  This causes the product of deterministic matrices to approximately degenerate into a product of scalars.  The \emph{separation} between the real eigenvalues should be considered as a measure of the hyperbolicity, with larger separation yielding better approximation by a scalar recurrence.
\item[(ii)]  For $n \approx Nz^2,$ the transfer matrices become almost singular or \emph{parabolic}, corresponding to $x \mapsto {\widetilde T_n x}$ having a single fixed point on real projective space. The recurrence takes on a transitional type behavior, which naturally gives rise to the Airy type asymptotics for the Hermite polynomials when $z\approx 1$ in a scaling window in $n$ around the turning point of width $N^{1/3}.$ Specifically,
\begin{equation} \label{Airy}
  \pi_n(z) \sim (2\pi)^{1/4} e^{Nz^2} 2^{-n} n^{-1/12} \sqrt{\tfrac{n!}{N^{n}}}\text{Ai}(-k) \quad \text{where}\quad 
  k = (n-Nz^2)(Nz^2)^{-1/3}.
  \end{equation}
  
\item[(iii)] For $n \gg N z^2$,  the eigenvalues of  $\widetilde T_n$  are complex conjugate and $\widetilde T_n$ are called \emph{elliptic}, as the map $x \mapsto {\widetilde T_n x}$ has no fixed points on real projective space. This part of the recurrence gives rise to the oscillatory portion of the Hermite asymptotics which is observed as $z$ varies in the support of the semicircle. 
\end{itemize} 
\noindent We refer to  \cite[Chapter 12]{DeiftMcLaughlin}, in which the Hermite polynomial asymptotics are recovered from the recurrence \eqref{eq:hermite} by analyzing these three different regimes. 

\subsection{G$\beta$E recurrence}

The recurrence \eqref{eq:recurrence} for the characteristic polynomial of the Gaussian $\beta$-ensemble
can be understood as a random perturbation of the Hermite recurrence: for all $1 \leq n \leq N,$  $T_n(z) = \widetilde{T}_n(z) + \O(1/\sqrt{N}),$ where these $\O(1/\sqrt{N})$ terms are independent, centered, and have good tails.  \emph{How} this random perturbation affects the transfer matrix recurrence \emph{strongly} depends on the previous regimes. 
%This is reflected in the evolution of the \emph{phase}
%\[
%  \alpha_n(z) = 
%  \frac{\begin{pmatrix} \Phi_{n}(z) , \Phi_{n-1}(z) \end{pmatrix}}
%  {\biggl\|\begin{pmatrix} \Phi_{n}(z) , \Phi_{n-1}(z) \end{pmatrix}\biggr\|}.
%\]
In the hyperbolic regime, we show that even with these random perturbations, the product of $\left\{ {T}_n \right\}$ are well--approximated by a (noisy) scalar recurrence.  
Moreover, this approximation can be performed at the level of moderate deviations, allowing us to uniformly control the difference at many different $z$ and $n$ with very high probability. % One can view the projectivization of the columns of $A(z)$ as representing the fixed \emph{direction} that the hyperbolic transfer matrix recurrence tends toward.
%We give effective estimates for this approximation, which can be viewed as saying that the phase $\alpha_n(z)$ is sufficiently close to a deterministic $\widetilde{\alpha}_n(z)$ that the entire transfer matrix recurrence $T_n(z) T_{n-1}(z) \cdots T_1(z)$ is well approximated by $L_n(z) L_{n-1}(z) \cdots L_1(z) \widetilde{\alpha}_n^t(z)x_0(z)$ for a deterministic row vector $x_0(z)$ and for independent random scalars $\left\{ L_j(z) \right\}.$\footnote{We do not directly control the phase function, although it will be a corollary of what we show.}

In the parabolic regime, the behavior of the recurrence is related to a second order differential operator, the \emph{stochastic Airy operator} of \cite{RiderVirag}.
In \cite{LambertPaquette03} we construct a random entire function $\text{SAi},$ which is the scaling limit of the characteristic polynomial in a neighborhood of the edge: specifically
\begin{equation} \label{SAi}
 w_N \varphi_N(1+\frac{\lambda}{2 N^{2/3}})
 \tfrac
    {\Exp \exp(\mathfrak{G}_N)}
    {\exp(\mathfrak{G}_N)}
    \Asto[N]
    \bigl(\text{SAi}_\lambda : \lambda \in \R\bigr),
\end{equation}
where 
\(w_N^{-1}(z) \coloneqq (2\pi)^{1/4}e^{N z^2} 2^{-N} N^{-1/12}\sqrt{\tfrac{N!}{N^N}}\), see \eqref{Airy},
and $\mathfrak{G}_N$ is a Gaussian random variable of logarithmic variance, primarily measurable with respect to hyperbolic portion of the transfer matrix recurrence.
We emphasize that the results obtained here are crucial input to that limit. 

In the elliptic regime, the transfer matrices become highly oscillatory, the microscopic fluctuations of which are precisely what gives rise to the \emph{Brownian carousel} \cite{ValkoVirag}.  The scaling limit of the characteristc polynomial in this regime has not been considered, but it should be the Stochastic $\zeta$-function of   \cite{ValkoVirag20}, which for $\beta=2$ originates in \cite{CNN17} and \cite{CHNNR}. The analysis in this paper is the first, independent step to solving this problem.

So, we treat of these problems (the hyperbolic, parabolic and elliptic transfer matrix recurrences of G$\beta$E) separately, and in this paper we solely focus on the analysis of the hyperbolic portion of the random recurrence.
This is needed to describe the asymptotics of the characteristic polynomial for all $z \in \C$ (save for a small window around $z=0$).  
%As we are ultimately interested in $z \in (-1,1)$ for applications of Gaussian multiplicative chaos, our main result Theorem \ref{thm:real} will concern the hyperbolic portions of the recurrence at any real $z.$  However, to illustrate the technique, we have also formulated a result Theorem \ref{thm:planar} for $z \in \C$ well--separated from $(-1,1).$

\subsection{Gaussian coupling}\label{sec:coupling}

Our goal is to build a probability space on which the characteristic polynomial $\{\varphi_n(z): z, n\}$ is well approximated by the exponential of a Gaussian field.  The underlying coupling will be between two random walks defined in terms of $\mathbf{A}$ and Brownian motions.
Define for all $k \geq 1,$
\begin{equation} \label{def:XY}
X_k = \frac{b_k}{\sqrt{2}} 
\qquad\text{and}\qquad 
Y_k = \frac{a_{k-1}^2 - \beta(k-1)}{\sqrt{2\beta(k-1)}}, 
\end{equation}
where we take by convention $Y_1=0.$  Then these are independent, mean $0$ and variance $1$.  We shall work on a probability space that supports two independent standard Brownian motions $(\widehat{\mathbf{X}}_t : t \geq 0)$ and $(\widehat{\mathbf{Y}}_t : t \geq 0)$ which are coupled to $\left\{ X_k \right\}$ and $\left\{ Y_k \right\}$ in such a way that
\[
  \sum_{j=1}^n X_j = \widehat{\mathbf{X}}_n
  \quad
  \text{for all $n \geq 1$ and}
  \quad
  \limsup_{n \to \infty} 
  \frac{1}{\log n}\biggl|\sum_{j=1}^n Y_j - \widehat{\mathbf{Y}}_n\biggr|
  < \infty\quad \As
\]
Such an embedding is usually referred to as a \emph{strong embedding} of random walk or \emph{KMT} embedding.  For a specific discussion of such embeddings, see Appendix \ref{sec:embeddings} and Theorem \ref{thm:shaoKMT}.  In particular, we use a version which gives some exponential moment control for $\max_{1 \leq j \leq N}\bigl|\sum_{j=1}^n Y_j - \widehat{\mathbf{Y}}_n\bigr|$.  
\begin{remark}
  The random variables $\{ X_k, Y_k \}$ we consider are mean $0$, independent, variance $1$ and have uniform control on their exponential moments.  For any such random variables, the version of the KMT embedding in Theorem \ref{thm:shaoKMT} applies.  On the other hand, for the exact G$\beta$E recurrence, it is possible to avoid appealing to KMT type theorems and rather exploit that the $X_k$ are exactly normal while $Y_k$ become increasingly close to normal as $k$ grows.  Moreover the rate is fast enough to conclude the same control on the random walk as guaranteed by Theorem \ref{thm:shaoKMT}, for a different entry-by-entry coupling (and indeed one in which $\{ (X_k, Y_k, \widehat{\mathbf{X}}_k - \widehat{\mathbf{X}}_{k-1},\widehat{\mathbf{Y}}_k - \widehat{\mathbf{Y}}_{k-1}): k \in \N \}$ are independent random vectors).  In any case, the only hypothesis on the coupling that we need is that the conclusion of Theorem \ref{thm:shaoKMT} holds.
\end{remark}

We will express the random corrections to $\varphi_N$ in terms of explicit functionals of these Brownian motions.  Let $({\mathbf{X}}_t : t \in [0,1]) = (N^{-1/2}\widehat{\mathbf{X}}_{tN} : t \in [0,1])$ and $({\mathbf{Y}}_t : t\in[0,1]) = (N^{-1/2}\widehat{\mathbf{Y}}_{tN} : t \in [0,1]),$ which remain standard Brownian motions by Brownian scaling.   
Let $J: \C \setminus [-1,1] \to \D$ be the inverse Joukowsky transform,
\begin{equation}\label{eq:J}
J(z) = z- \sqrt{z^2-1}  ,
\end{equation}
where $\sqrt{\cdot}$ is chosen so that $J$ is a conformal map. 

Define for $t \in [0,1)$ and $z \in \C \setminus [-\sqrt{t},\sqrt{t}],$
\begin{equation}
    \mathfrak{g}_t(z) = -\frac{1}{2}\int_0^t \frac{\d \mathbf{X}_u+ J(z/\sqrt{u}) \d\mathbf{Y}_u}{\sqrt{z^2-u}}.
  \label{eq:Wtz}
\end{equation}
Note by reflection symmetry, $\overline{  \mathfrak{g}_t(z)} = {  \mathfrak{g}_t(\overline{z})}.$
For $z \in \C \setminus [-1,1]$ this extends continuously on sending $t \to 1,$ and we define for $z \in \C \setminus [-1,1],$ $\mathrm{W}(z) \coloneqq \lim_{t \to 1}   \mathfrak{g}_t(z).$ For $z \in [-1,1],$  $\mathrm{W}(z) \coloneqq \lim_{t \to 1}   \mathfrak{g}_t(z)$ exists as a log-correlated Gaussian field (c.f.\ Remark \ref{rk:logcorrelated}).  The covariance structure of this field can be given explicitly (c.f.\ Lemma \ref{lem:magic})
\begin{equation}
  \E\left[  \mathfrak{g}_t(z)   \mathfrak{g}_s(q)  \right]
=  - \log\big(1- J\big(z/\sqrt{t\wedge s}\big)J\big(q/\sqrt{t\wedge s}\big)\big) , 
\qquad  z\in \C \setminus[-\sqrt{s},\sqrt{s}] \, ,  q\in \C \setminus[-\sqrt{t},\sqrt{t}].
  \label{eq:Wtzmagic}
\end{equation}

\begin{remark} \label{rk:GAF}
Let $\xi_1,\xi_2 \dots$ be a sequence of i.i.d.\ real standard Gaussian random variables and define the Gaussian process $\xi(q) = \sum_{k\ge 1} \frac{\xi_k}{\sqrt{k}} q^k$ for $q \in \D$. This is a \emph{Gaussian analytic function} (GAF) which satisfies  $\xi(\overline{q}) = \overline{\xi(q)}$ and has covariance structure:
\[
  \E[\xi(q){\xi(w)}] = -\log(1-q{w}), 
\qquad w,q \in \D. 
\]
The Gaussian field $\mathrm{W} \coloneqq   \mathfrak{g}_1$ is the pull--back of $\xi$ by the map  $J$, i.e.\  
$\mathrm{W} = \xi \circ J$. This field satisfies $\mathrm{W}(\overline{z}) = \overline{W(z)}$ and has covariance:
\begin{equation}\label{eq:Wdef}
  \E[\mathrm{W}(z)\mathrm{W}(w)] = - \log\big(1- J(z)J(w)\big)  , \qquad z,w\in \C\setminus[-1,1]. 
\end{equation}
Note that due to the scaling invariance of Brownian motion, we have that for any $s \in (0,1),$ $(  \mathfrak{g}_s(z\sqrt{s}) : z \in \C \setminus [-1,1]) \overset{\rm law}{=} (\mathrm{W}(z) : z \in \C \setminus [-1,1]),$
and hence for any $s \in (0,1),$ $  \mathfrak{g}_s(z\sqrt{s})$ is also expressible in terms of this GAF in law.
\end{remark}

\begin{remark}[Boundary values] \label{rk:logcorrelated}
The real and imaginary parts of the  Gaussian process  $\mathrm{W}$  are not independent. 
The boundary values of the GAF $\xi$ are well--defined as a log--correlated random field (random distribution in $H^{-\epsilon}(\partial \D)$ for any $\epsilon>0$ given by a Fourier series). %In fact, we see from the definitions that for $z\in \partial \D$, . 
This shows that  we can extend the process $\mathrm{W}(z)$ as a (complex--valued) Gaussian  log--correlated random field for $z\in [-1,1]$ which satisfies:
\[
\Re \mathrm{W}(\cos\theta)  = \frac{\zeta(e^{\i\theta})+\zeta(e^{-\i\theta}) }{\sqrt{2}} , \qquad
\Im \mathrm{W}(\cos\theta)  = \frac{\widetilde\zeta(e^{\i\theta})-\widetilde\zeta(e^{-\i\theta}) }{\sqrt{2}} , \qquad 
\text{for}\quad \theta\in [0,\pi] ,
\]
where $\zeta$ and $\widetilde\zeta$  are two independent copies of the \emph{Gaussian Free Field} on $ \partial \D$.
This shows that $\Re\mathrm{\mathrm{W}}(q)$ and $\Im\mathrm{\mathrm{W}}(q)$ for $q\in[-1,1]$ are independent Gaussian fields with covariance structure:
\begin{equation} \label{logcov}
\begin{aligned}
\E\big[ \Re \mathrm{W}(x) \Re \mathrm{W}(y) \big] & = \frac{\log|2(x-y)|^{-1}}{2}  , &\quad  x,y \in [-1,1] \ \\
\E\big[ \Im \mathrm{W}(x) \Im \mathrm{W}(y) \big]  &= - \frac{1}{2} \log\bigg|\frac{x-y}{1-xy+\sqrt{1-x^2}\sqrt{1-y^2}}\bigg| , 
&\quad y, x\in [-1,1] . 
\end{aligned}
\end{equation}
These formulae are consistent with those obtained for the real and imaginary parts of the logarithm of the characteristic of the Gaussian unitary ensemble ($\beta=2$) in \cite[Theorem 1.1]{BWW} and \cite[Section~2.1]{CFLW}. 

\end{remark}

\subsection{Strong approximation in the plane and GMC applications}

Let $\alpha>0$ and 
\begin{equation} \label{def:P}
\mathscr{P} =\left\{ z \in \C :\ |\Im z| \ge N^{-\alpha} \text{ or }|\Re z| \ge 1+ N^{-2\alpha}/2 \right\}.
\end{equation}
We show that if $\alpha$ is sufficiently small, then we can couple the above GAF and the characteristic polynomial to be uniformly close in this domain:

\begin{theorem}\label{thm:planar}
%Let $\Phi_{n}$ denote the (rescaled) monic Hermite polynomial orthogonal of degree $n\in\N$ with respect to the weight $e^{-2 N x^2}$  on $\R$. 
Choose $\alpha =1/9$ and $\delta = 1/45.$  For any compact set $K \subset \C$, 
\[
\P\left[ 
  \sup_{z \in K\cap \mathscr{P} } \Bigg|
  \frac{\varphi_{N}(z) \E\big[\exp\big(\sqrt{\frac{2}{{\beta}}}\mathrm{W}(z)\big)\big]  }{\pi_N(z) \exp\big(\sqrt{\frac{2}{{\beta}}}\mathrm{W}(z)\big)} -1 \Bigg|  \ge C_{\beta, K} N^{-\frac{1}{15}}  \right] \le e^{-c_\beta N^{\delta}} , 
 \]
where $C_{\beta, K} , c_\beta>0$ are constants. 
\end{theorem}

The main steps of the proof of this theorem are explained in Section~\ref{sec:transfermatrices}, while the details are given in Sections~\ref{sec:rmp} and~\ref{sect:coupling1}.
This settles the question of the asymptotics of the characteristic polynomial of the Gaussian $\beta$-ensemble for fixed $z \in \C\setminus [-1,1]$.  Theorem~\ref{thm:planar}, gives a strong approximation in that it holds uniformly as a random process with a polynomial rate, but it requires $z$ quantitatively far from $[-1,1]$, the support of the semicircle law.

This result already implies a multiplicative chaos convergence in a regularized sense, as we illustrate below.
Let $X= \Re\mathrm{W}$. 
Since $\mathrm{W}$ is a GAF on $\C\setminus[-1,1]$, it can be recovered from its boundary values, that is, for $x\in[-1,1]$ and $\epsilon>0$, 
\begin{equation} \label{Poisson}
X(x+ \i \epsilon) = \int_\R X(t) P_\epsilon(x-t) \d t
\end{equation}
where $P_\epsilon(t) = \tfrac{1}{\epsilon \pi(1+ (t/\epsilon)^2)}$ is a probability measure on $\R$ called the Poisson kernel. 
Observe that for $x\in\R\setminus[-1,1]$,
\[
\E X(x)^2 = \E \mathrm{W}(x)^2 =   - \log\big(1- J(x)^2\big) \sim \frac{1}{4x^2} \qquad\text{as }|x| \to\infty
\]
so that $X \in L^2(\R\setminus[-1,1])$  almost surely and the integral \eqref{Poisson} makes sense. 
Hence, we can view $x\mapsto X(x+ \i \epsilon)$ for small $\epsilon>0$ as a regularization of the log-correlated field $X$ with covariance structure \eqref{logcov}.
By Gaussian multiplicative chaos arguments\footnote{This is different from the usual setting since the field $X(x)$ is defined on $\R$, but only log-correlated  on $[-1,1]$ and the Poisson kernel is not a usual (compactly supported) mollifier used to regularized. However, one can easily adapt the elementary proof from \cite{Berestycki} to show convergence in this case by using the explicit formula \eqref{eq:Wdef} for the correlation kernel of the filed $z\in \mathbb{H} \mapsto X(z)$.}, it holds in probability as $\epsilon\to0$, 
\[
 \frac{\exp\big(\gamma X(x+\i\epsilon)\big)}{\E\big[\exp\big(\gamma X(x+\i \epsilon)\big)\big]} \d x \to 
 \mu_\gamma(\d x)
\]
where $ \mu_\gamma$ is a (random) measure on $[-1,1]$ and the convergence holds weakly. 
The measure  $ \mu_\gamma$ is non-trivial if  $\gamma \in(0,2)$. 
This convergence actually holds in $L^1$ for the probability space defined in Section~\ref{sec:coupling} and it follows from Theorem~\ref{thm:planar} that with $\epsilon_N =N^{-\alpha}$, for any $\beta>0$, 
\begin{equation} \label{GMC}
\frac{|\varphi_N(x+\i\epsilon_N)|^\gamma}{\Gamma_N^{\gamma,\beta}(x)} \to  \mu_{\gamma_{\beta}}(\d x) \qquad\text{in probability as $N\to\infty$ where $\gamma_\beta = \gamma \sqrt{\tfrac{2}{\beta}}$}
\end{equation}
 and 
\[
\Gamma_N^{\gamma,\beta}(x) 
= |\pi_N(x+\i\epsilon_N)|^\gamma   \frac{\exp\big(\frac{\gamma^2}{\beta}\E[X(x+\i\epsilon_N)^2]\big)}{\exp\big(\frac{\gamma^2}{{\beta}}\Re\E[\mathrm{W}(x+\i\epsilon_N)^2]\big)}
= |\pi_N(x+\i\epsilon_N)|^\gamma  \exp\big(\tfrac{\gamma^2}{{\beta}}\E[(\Im\mathrm{W}(x+\i\epsilon_N))^2]\big).
\]
In particular, the limit \eqref{GMC} is non-trivial in the subcritical regime $\gamma_\beta < 2$, that is for any exponent $\gamma \in (0,\sqrt{2\beta})$. 

Using recent (optimal) rigidity results from \cite{BMP} for the eigenvalues of the Gaussian $\beta$-ensembles, one should be able to upgrade \eqref{GMC} to convergence in $L^1$ and then go back the usual normalization, 
\begin{equation} \label{GMC2}
\frac{|\varphi_N(x+\i\epsilon_N)|^\gamma}{\E|\varphi_N(x+\i\epsilon_N)|^\gamma} \to  \mu_{\gamma_{\beta}}(\d x) \qquad\text{as $N\to\infty$ in $L^1$} . 
\end{equation}
%
%https://arxiv.org/abs/2012.09969
These limits have been obtained in \cite{BWW} on the spectrum ($\epsilon_N=0$) for $\beta=2$ and $\gamma <\sqrt{2}$ (the $L^2$ phase) and in \cite{Kivimae} for $\beta=1,4$ and small enough $\gamma$. 
However, let us emphasize that this is the first result valid for general Gaussian $\beta$-ensembles and with a convergence in probability (in usual applications from random matrix theory, there is no coupling for different $N\in\N$).

%One also expects that the (asymptotic) relationships \eqref{charpolymom} between moments of the characteristic polynomial for general $\beta>0$ and $\beta=2$ remain valid if $\epsilon_N=0$. 
%In passing, one obtains the asymptotics of moments of the characteristic polynomial (above the spectrum),
%\begin{equation} \label{charpolymom}
%\E\big[|\varphi_N^{(\beta)}(x+\i\epsilon_N)|^\gamma \big]\simeq \Gamma_N^{\gamma,\beta}(x) \simeq
%\E\big[ |\varphi_N^{(\beta=2)}(x+\i\epsilon_N)|^{\gamma_\beta} \big]
%  \frac{ |\pi_N(x+\i\epsilon_N)|^{\gamma_\beta}}{ |\pi_N(x+\i\epsilon_N)|^\gamma} . 
%\end{equation}

Like \eqref{Poisson}, one can view $\big(\varphi_N(x+\i\epsilon_N)\big)_{x\in[-1,1]}$ as a (mesoscopic) regularization of the characteristic polynomial, hence \eqref{GMC} indicates that  $|\varphi_N(x)|^\gamma$ suitably re-normalized converges to the same GMC measures.

\medskip
As another application, we show in the Appendix~\ref{sec:CLT} how we can recover from Theorem \ref{thm:planar} Johansson's central limit theorem for linear statistics, \cite{Johansson}.
%In the next section, we take a crucial step in this direction. Namely, by analyzing the hyperbolic part of the matrix recurrence, we obtain an approximation for the characteristic polynomials $\Phi_n(x)$ for  $x\in[-1,1]$ and $n \ll Nx^2$. 
%To obtain the full picture, one needs to extend these asymptotics around the turning point (cf.~\cite{LambertPaquette03}) and also to study the elliptic regime which will be the subject of futur works. 

\subsection{Strong approximation near the real line}

Our main theorem gives an effective approximation for the hyperbolic portion of the recurrence when $z$ lies near the real line.  As it happens, the point $z=0$ is special, in that the entire transfer matrix recurrence is elliptic in a mesoscopic window around $0.$  For this reason, we make the following definition.

\begin{definition} \label{def:hyper}
  Let $N_p(z) = \lfloor N (\Re z)^2 \rfloor.$ Fix a small $\delta \in (0,\frac{1}{2})$ and let $\omega_N(z) =  N_p^{1/3}( \Omega \log  N_p)^{2/3}$ where $\Omega$ is allowed to depend on $N$ in such a way that $\Omega \le N^{\delta/6}.$% \frac{N^{\delta/6}}{6\delta \log N}$}.
  We let
  \[
    \mathscr{D}_H = \left\{ z\in \C : 0 \le \Im z \le 2\Re z  , \Re z \ge N^{\delta - \frac{1}{2}}  \right\} 
  \] 
  and for any $z\in \mathscr{D}_H$,
  \[
    N_H(z) = (N_p - \omega_N) \wedge N.
  \]
\end{definition}

Note that for $z \in \mathscr{D}_H,$ we always have $N_p \geq N^{2\delta}$ and $N_p \ge 2\omega_N$.
We call 
\[
\big\{ T_k : 1 \leq k \leq N_H(z) , z\in  \mathscr{D}_H  \big\}
\]
the hyperbolic region. 

\begin{theorem} \label{thm:real}
Let us define the random vector $\bigl(\Gamma_N(z) : z\in\C\setminus [-\sqrt{t_H}, \sqrt{t_H}]\ \bigr)$ implicitly by 
\begin{equation} \label{def:Upsilon}
 \begin{pmatrix} \Phi_{{N_H}} \\ \Phi_{{N_H}-1} \end{pmatrix}
 = \pi_{{N_H}-1} \tfrac{\exp\big(\sqrt{\frac{2}{\beta}}\mathfrak{g}_{t_H}\big)}{\E\big[\exp\big(\sqrt{\frac{2}{\beta}}\mathfrak{g}_{t_H}\big)\big]}
 \begin{bmatrix}
      \lambda_+(t_H) & \lambda_{-}(t_H) \\
      1 & 1 \\
    \end{bmatrix} \left(  \begin{pmatrix} 1 \\ 0
    \end{pmatrix}  + \Gamma_N \right)
\end{equation}
where $t_H = N_H/N$, the GAF $\mathfrak{g}$ is given by \eqref{eq:Wtz} and $\lambda_{\pm}$ are as in \eqref{def:root}. 
There exists a small constant $c_R = c_R(\beta,\delta)$ such that the $\Gamma_N$ satisfies for any $R,\eta>0$, 
\begin{equation} \label{errorcontrol}
  \max_{z \in \mathscr{D}_H}
  \P\bigg[ |\Gamma_N(z)| \ge  \eta \bigg] \le c_R^{-1} (N^{-c_R \eta^2\Omega} + N^{4-R\Omega} \big).
\end{equation}
\end{theorem}

This results shows that in the hyperbolic region, the G$\beta$E characteristic polynomials can be approximated by the Hermite polynomial times the exponential of a GAF which arise by \emph{linearizing} the transfer matrix recurrence. The (multiplicative) error is controlled uniformly over the domain $\mathscr{D}_H$ and it is small only by choosing the parameter $\Omega$ which controls the distance to the turning point sufficiently large. 
In our subsequent applications of this result, we can choose $\Omega = \Omega_N = (\log N)^{1/3-\epsilon}$ for a small $\epsilon>0$ as explained in  \cite[Appendix~A]{LambertPaquette03}.  

For $z\in \R$, this distinguishes between the following cases:
\begin{itemize}[leftmargin=*]   \setlength\itemsep{0em}
\item If $|z|>1$, that is outside of the spectrum, Theorem~\ref{thm:real} covers  the full transfer matrix recurrence and $t_H=1$. In this case, extracting the first column of \eqref{def:Upsilon}, one recovers the asymptotics\footnote{Using the fact that $\pi_{N_H}= \lambda_+(t_H)\pi_{N_H-1}$ up to a small multiplicative error, cf. Proposition~\ref{prop:Hermite}.} from Theorem~\ref{thm:planar}. Note that these asymptotics actually hold under the optimal condition $|z| \ge 1 + \big( \frac{\Omega \log N}{N} \big)^{2/3}$.
\item At the edge, $z=1+ \frac{\lambda}{2N^{2/3}}$, $N_H(z) = N - N^{1/3} (\Omega\log N)^{1/3}$ and Theorem~\ref{thm:real} covers almost the entire transfer matrix recurrence except for the last block of size $N^{1/3}$. 
Hence, this result does not capture the full asymptotics of the characteristic polynomial. 
 This \emph{parabolic part} of the recurrence can be approximated by the \emph{stochastic Airy equation} (\cite{LambertPaquette03}) and this leads to a correction in  the form of the \emph{stochastic Airy function} \eqref{SAi}. 
\item If $|z|<1$, that is in the bulk, the situation is more involved and Theorem~\ref{thm:real} only yields a partial result in the sense that it provides the input (with overwhelming probability) to study the  \emph{parabolic} and \emph{elliptic} $(n \gg Nz^2)$ region of the recurrence. 
In this case, we expect a very different behavior due to the dense zeros of $\Phi_N$, akin to the classical Plancherel--Rotach asymptotics for the Hermite polynomial. 
\end{itemize}

%The main steps of the proof of this theorem are explained in Section~\ref{sect:main}, while the details are given in Section~\ref{sec:linearized}. 
%As for the error $\Upsilon_n$, we expect that the control from \eqref{errorcontrol} is \emph{optimal} in the sense that if $\Omega(N) \to 0$ as $N\to+\infty$, then error $\Upsilon_{N_H(z)}$ is not small.  In particular, for small $\Omega,$ we begin to enter the scaling window of the stochastic Airy function.  Indeed, from the decay of the Airy function $\text{Ai}(t) \approx e^{ -\frac{2}{3}t^{3/2}},$ we can anticipate roughly that the contribution of the stochastic Airy function to $\varphi_{N_H}$ is on the order of $N_H^{-c\Omega}.$  We emphasize that this Airy behavior will not be captured by the type of approximation we have made, which is only correct in the hyperbolic regime.  
%%In this sense, Theorem~\ref{thm:real} covers the whole hyperbolic regime of the transfer matrix recurrence. 

\begin{remark}
It is worth pointing out that the methods of the proofs of Theorems~\ref{thm:planar} and \ref{thm:real} are rather insensitive to the distribution of the noise in the matrix \eqref{def:trimatrix}. 
The only important conditions are that (up to small corrections), the random variables \eqref{def:XY} have mean~0, variance~1 and sub--Gaussian tails in a certain moderate deviation regime.  
This is in sharp contrast with the case of the characteristic polynomial of the circular $\beta$-ensembles, for which the Gaussian behavior inside of the unique disk comes specifically from the law of the Verblunsky coefficients in   \eqref{eq:Szego}, see e.g. \cite{CN19}. 
\end{remark}

\subsection{Organization}

In Section~\ref{sect:prel}, we introduce the notation as well as the formalism for concentration of random variables  that we will use in the remainder of this paper.
In Section~\ref{sec:transfermatrices}, we explain the general strategy of the proofs of Theorems~\ref{thm:planar} and \ref{thm:real}. 
Some auxiliary results regarding the coupling between the noise and Brownian motions known as \emph{strong embedding} are given in Section~\ref{sect:coupling} and the Appendix~\ref{sec:embeddings}. 
In Sections~\ref{sect:diag}, we reformulate the transfer matrix recurrence \eqref{eq:recurrence} by isolating the contribution of the noise from that of the deterministic recurrence \eqref{eq:hermite} by suitable conjugations.
In Section~\ref{sect:main}, we present our main results, Proposition~\ref{prop:rec} and  Theorem~\ref{thm:main}, for products of $2\times 2$ random matrices with a hyperbolic character. 
Then, we discuss applications to the characteristic polynomial of the Gaussian $\beta$-ensembles.
In Section~\ref{sec:rmp}, we introduce the general framework for the proof of our main results and we give the proof of Proposition~\ref{prop:rec}. 
In Section~\ref{sec:linearized}, we give the proof of Theorem~\ref{thm:main} which is the most technical part of this paper. 
The main steps of this proof are summarized in Section~\ref{sect:strategy}. 
Finally, in the Appendices~\ref{sect:properties} an \ref{sect:est}, we review the asymptotics of Hermite polynomials and provide estimates for the noise that are required for our proofs. 

\paragraph{Acknowledgements}

We would like to thank Diane Holcomb, conversations with whom helped launched this project.  We acknowledge support from the Park City Mathematics Institute 2017, at which this program was begun, and in particular acknowledge NSF grant DMS:1441467.  We also thank the anonymous referees for their careful reading of the manuscript which has greatly improved its quality.

\section{Preliminaries} \label{sect:prel}

In this section, we review some formalism for concentration and moderate deviations for certain random variables which we will use for the proofs of Theorem~\ref{thm:planar} and Theorem~\ref{thm:real}.

\subsection{Notation}

Throughout this article, the parameter $\beta>0$ is fixed and we do not keep track of the $\beta$-dependency of the various constants.
We make use of the symbols $\lesssim$ and $\gtrsim$ in the following form.  We write $f(x) \lesssim_{\alpha} g(x)$ if there is a finite function $C(\beta,\alpha)>0$ so that for all $x$ for which $f$ is being compared to $g$,  $|f(x)| \le C(\beta,\alpha)g(x)$.
We alternatively use $f(x) =\O_{\alpha}\big(g(x) \big)$ or $f(x) \leq \O_\alpha\big(g(x)\big)$ to mean  $f(x) \lesssim_{\alpha} g(x)$.  If we omit the subscript $\alpha$ in either case, we mean the inequality holds with a constant $C_\beta$ which only depends on $\beta>0$. 

 For a matrix $M$, we let $\|M\|$ be the operator norm of $M$ and we will frequently use that 
 $\|M\| \le c_d \sup_{i,j \le d}|M_{i,j}|$ for a constant $c_d$ which depends only on the dimension $d\in\N$ of $M$.
 We also let  $\diag(M)$ denote the diagonal matrix matching the diagonal of matrix $M.$
We take the convention that for a sequence of matrices $\left\{ M_n \right\}$, 
\[
 {\textstyle \prod_{j=p}^n} M_j = M_n M_{n-1} \dots M_{p+1}M_p.
\]

\subsection{Concentration} \label{sec:concentration}

We make crucial use of the theory of sub--Gaussian and sub--exponential random variables.  Furthermore, we will formulate many standard concentration results in terms of this theory.  For clarity, we briefly overview this theory, following \cite[Chapter 2]{Vershynin}, where one may find the proofs of all the claims in this section.  

Define, for any $p \geq 1,$ and any complex valued random variable $X$, 
\[
  \VERT X \VERT_p = \inf\left\{ t \geq 0 : \Exp e^{|X|^p/t^p} \leq 2\right\}.
\]
For all those $X$ for which $\VERT X \VERT_p < \infty,$ this defines a norm.  In the cases of $p=1$ and $p=2,$ these are the sub--exponential norm and the sub--Gaussian norm, respectively, and those are the only two cases we will use in this paper.
  For a matrix--valued random variable~$X,$ we will write $\VERT X \VERT_p$ as a shorthand for $\VERT \| X \| \VERT_p.$ 

By Markov's inequality, it follows that if $\VERT X \VERT_p < \infty,$ then for all $t \geq 0$
\begin{equation}\label{eq:Xtail}
  \Pr \left[ |X| \geq t \right] \leq 2\exp( -t^p/ \VERT X \VERT_p^p ),
\end{equation}
on observing the infimum in the definition of $\VERT \cdot \VERT_p$ is attained whenever it is finite.  Moreover, this concentration inequality is equivalent to the finiteness of $\VERT\cdot\VERT_p,$ in that if there exists $s \geq 0$ such that
\begin{equation}\label{eq:normup}
  \Pr \left[ |X| \geq t \right] \leq s\exp( -t^p ) \quad \text{for all } t \geq s,
\end{equation}
then $\VERT X \VERT_p \leq C_{p,s}$ for some absolute constant $C_{p,s} > 0.$  As a corollary, it follows that $\VERT \cdot \VERT_p$ is essentially monotone in $p$ in that for any $p \leq q,$ there is an absolute constant $C_{p,q}$ so that for all $X$,
\[
  \VERT X \VERT_p \leq C_{p,q} \VERT X \VERT_q.
\]

Control of $\VERT X \VERT_p$ can also be formulated in terms of the moments of $X.$  For our purposes, it will be enough to observe that for any $k \in \N$ and any $p \geq 1,$ there is a constant $C_{k,p}$ so that for all $X$,
\begin{equation}\label{eq:Xmoment}
  \Exp |X|^k \leq C_{k,p} \VERT X\VERT_p^k.
\end{equation}
Furthermore, centering a random variable can not greatly deterioriate its concentration in that there is an absolute constant $C_p$ so that for all $X$
\begin{equation}\label{eq:Xcentering}
  \VERT X - \Exp X \VERT_p \leq C_p \VERT X \VERT_p.
\end{equation}

Finally we observe as a consequence of Young's inequality that for any $p,q \geq 1$ satisfying $1/p + 1/q = 1,$
for any two random variables $X$ and $Y$, 
\[
  \VERT XY \VERT_{1} \leq  \VERT X\VERT_p \VERT Y \VERT_q.
\]
See the proof of \cite[Lemma 2.7.7]{Vershynin} for details.

\subsection{Moderate deviations} \label{sec:moddev}

We shall use a variety of concentration inequalities, which we will formulate in terms of the $\VERT \cdot \VERT_p$ norms.  
We begin with what can be viewed as a version of Hoeffding's inequality.

\begin{theorem}[{\cite[Proposition 2.6.1]{Vershynin}}]\label{thm:hoeffding}
If $X_1, X_2, \dots, X_n$ are independent, centered, sub--Gaussian random variables, then.
  \[
    \VERT \textstyle{\sum_{i=1}^n X_i} \VERT_2^2 \lesssim \sum_{i=1}^n \VERT X_i \VERT_2^2.
  \]
\end{theorem}
Using \eqref{eq:Xcentering}, it also holds that upon dropping the assumption that $\left\{ X_i \right\}$ are centered, we have that 
\[
    \VERT \textstyle{\sum_{i=1}^n (X_i-\Exp X_i }) \VERT_2 \lesssim  \sum_{i=1}^n \VERT X_i \VERT_2^2.
\]

We shall also encounter sums of random variables which are only subexponential, in which case such strong concentration is not possible for the entire tail of the sum, but remains true in the regime of moderate deviations. The following is roughly a corollary of Bernstein's inequality.
\begin{theorem} \label{thm:bernstein}
  Let $X_1, X_2, \dots, X_n$ be independent, centered, sub--exponential random variables.  There is an absolute constant $C>0$ and  an event $\A \in \sigma( X_i : 1 \leq i \leq n)$ having
  \[
    \VERT (\textstyle{\sum_{i=1}^n X_i})\one_\A \VERT_2^2 \lesssim \sum_{i=1}^n \VERT X_i \VERT_1^2
    \quad\text{and}\quad
    \Pr[ \A^c ] \leq 2\exp\left(-\frac{\textstyle{\sum_{i=1}^n} \VERT X_i \VERT_1^2}{C\max_{1 \leq i \leq n} \VERT X_i \VERT_1^2}\right).
  \]
\end{theorem}
\begin{proof}
  Theorem 2.8.1 of \cite{Vershynin} states that for all $t \geq 0$
  \begin{equation}\label{eq:actualbernstein}
    \Pr\left[ \left| \textstyle{\sum_{i=1}^n} X_i \right| \geq t \right] \leq 2\exp\left( -\frac{1}{c}\min\left\{ \frac{t^2}{\sum_{i=1}^n \VERT X_i \VERT_1^2}, \frac{t}{\max_i \VERT X_i\VERT_1} \right\} \right).
  \end{equation}
  Hence on letting
  \[
    \A =\left\{ \left| \textstyle{\sum_{i=1}^n} X_i \right| \leq \textstyle{\sum_{i=1}^n} \VERT X_i \VERT_1^2 /\max_i \VERT X_i\VERT_1  \right\},
  \]
  we may write the unconditional tail bound for all $t \geq 0,$
  \begin{equation*}
    \Pr\left[ \left| \textstyle{\bigl(\sum_{i=1}^n} X_i\bigr)\one_{\A} \right| \geq t \right] \leq 2\exp\left( -\frac{t^2}{c\sum_{i=1}^n \VERT X_i \VERT_1^2}\right),
  \end{equation*}
  noting that $\bigl|\sum_{i=1}^n X_i\bigr|\one_{\A}$ is bounded above by $\textstyle{\sum_{i=1}^n} \VERT X_i \VERT_1^2 /\max_i \VERT X_i\VERT_1$ almost surely.  This implies the subgaussian norm bound.  The claimed probability bound on $\A^c$ also follows from \eqref{eq:actualbernstein}.
\end{proof}

As a small generalization, we can apply this inequality to sums which are constructed from families of independent variables.  Such tail bounds have appeared in the literature in the context of dependency graphs \cite{Janson}.
\begin{theorem} \label{thm:Svante}
Let $\mathcal{J}$ be a finite set and $Y= \sum_{\alpha \in \mathcal{J} } X_\alpha$ where $X_\alpha$ are centered random variables.
Let $\gamma \in \N$ and assume that we have a partition $\mathcal{J}  = \mathcal{J}_1 \cup \cdots \cup \mathcal{J}_\gamma$ such that each family $\left\{ X_\alpha : \alpha \in \mathcal{J}_k \right\}$ is one of independent variables, for $1 \leq k \leq \gamma.$ 
Suppose that 
$b$ and $\sigma^2$ are chosen so that
\[
b \ge \max_{\alpha\in \mathcal{J} }\VERT X_\alpha \VERT_1
\qquad\text{and}\qquad
\sigma^2 \ge \max_{k=1,\dots, \gamma}\Big( \sum_{\alpha \in \mathcal{J}_k } \VERT X_\alpha\VERT_1^2 \Big). 
\]
There exists an event $\A$ depending on $(b,\sigma^2)$ that is measurable with respect to $\sigma( X_\alpha : \alpha \in \mathcal{J})$ and a numerical constant $c$ such that 
\[
\VERT  Y \1_{\A}\VERT_2 \lesssim \gamma \sigma
\quad
\text{and}
\quad
\P[\A^c] \leq 2\gamma e^{-c^{-1}(\sigma/b)^2}
\]
\end{theorem}
\begin{proof}
For each $1 \leq k \leq \gamma,$ define the event
\[
  \A_k = \left\{ \left|\textstyle{\sum_{\alpha \in \mathcal{J}_k}} X_\alpha\right| \leq \sigma^2/b \right\},
\]
and define $\A = \cap_{k=1}^\gamma \A_k.$  Applying \eqref{eq:actualbernstein} exactly as in the proof of Theorem~\ref{thm:bernstein}, we obtain that for any $1 \leq k \leq \gamma$,
\begin{equation} \label{eq:wawa}
  \VERT  (\textstyle{\sum_{\alpha \in \mathcal{J}_k}} X_\alpha)\one_{\A_k} \VERT_2
  \lesssim \sigma
  \quad
\text{and}
\quad
  \Pr[\A_k^c] \leq 2e^{- c^{-1}(\sigma/b)^2}.
\end{equation}
Applying the triangle inequality, we have
\[
  \VERT  (\textstyle{\sum_{k=1}^\gamma \sum_{\alpha \in \mathcal{J}_k}} X_\alpha)\one_{\A} \VERT_2
\leq 
  \sum_{k=1}^\gamma 
  \VERT  (\textstyle{\sum_{\alpha \in \mathcal{J}_k}} X_\alpha)\one_{\A_k} \VERT_2   \lesssim \gamma \sigma \, ,
\]
where  we have used that if $|X| \leq |Y|$ almost surely then $\VERT X \VERT_2 \leq \VERT Y \VERT_2.$
The desired conclusion now follows from \eqref{eq:wawa} by a union bound to estimate $\P[\A^c]$. 
\end{proof}

We will also heavily use a matrix martingale concentration inequality of \cite{Tropp}, which generalizes a scalar martingale inequality of \cite{Freedman}.

\begin{theorem} \label{thm:M}
Let $(M_n)_{n\ge 0}$ be a $\d \times \d$ matrix--valued  $\F_n$--martingale such that $M_0 = 0$.
Let $n \in \N$ and suppose that there is an $\alpha \geq 0$ such that
\[
\max_{k < n} \| M_{k+1}-M_k \|  \le \alpha ~ ~ ~\As
\]
Define 
\[
V_n = \sum_{k < n} \E\big[ \| M_{k+1}-M_k \|^2 | \F_k\big].
\]
%Then there is an absolute constant $C>0$ so that for any $\Sigma, t \geq 0,$
%  \begin{equation}
%    \P\left[\{ \sup_{k\le n}\|M_k\|   \ge t\} \cap \{ V_n \leq \Sigma^2\}\right] \le 2\d \exp\left(-\frac{1}{C}\min\left\{ 
%    \frac{t^2}{\Sigma^2},
%    \frac{t}{\alpha}
%    \right\} \right). 
%\end{equation}
There exists a constant $c>0$ so that for any $\Sigma > 0$, if we define the event $\A = \{ \sup_{k\le n}\|M_k\|  \le c \frac{\Sigma^2}{\alpha}\} \cap \{V_n \leq \Sigma^2\}$,
then
\[
  \big\VERT \max_{k\le n}\|M_k\| \1_{\A} \big\VERT_2 \lesssim_\d \Sigma
\quad
\text{and}
\quad
\P\big[\A^c \cap \{V_n \leq \Sigma^2\}\big] \le 2\d e^{-(\frac{\Sigma}{\alpha})^2}.
\]
\end{theorem}

\begin{proof}
  Apply the Freedman--Tropp's inequality (\cite[Corollary 1.3]{Tropp}) to the martingale $M_n$, for any $t \geq 0$,
  \begin{equation} \label{eq:FT}
  \P\left[\left\{ \max_{k\le n}\|M_k\|   \ge t \right\} \cap \left\{ V_n \leq \Sigma^2\right\}\right] \le 2\d \exp\left(-\frac{t^2/2}{\alpha t/3+ \Sigma^2} \right). 
\end{equation}
By taking $t =  c \Sigma^2/\alpha$ and choosing $c>0$ such that $\frac{3c^2}{2(c+3)} =1$,  this implies that
\[
\P[\A^c \cap \left\{V_n \leq \Sigma^2\right\}] \le 2\d \exp\left(-\frac{\Sigma^2}{\alpha^2} \right) . 
\] 
We also obtain the tail bound
\[
  \P\left[ \left\{ \sup_{k\le n}\|M_k\|   \ge t \right\} \cap \A \right] \le 2\d \exp\left(-\frac{t^2}{c^2\Sigma^2} \right).
\]
By definition of the sub--Gaussian norm and \eqref{eq:normup}, this yields the claim.  
\end{proof}

\section{Transfer matrix reformulation}
\label{sec:transfermatrices}

\subsection{Approximate diagonalization} \label{sect:diag}
The goal of this section is to reformulate our main theorems by performing a conjugation of the transfer matrices.  
Using Lemma~\ref{lem:diagonalization}, we have $\widetilde T_k = V_k \Lambda_k V_k^{-1}$  where for $k\ge 2$, 
\begin{equation} \label{def:V}
V_{k} =     \begin{bmatrix}
      \lambda_+(\frac{k-1}N) & \lambda_{-}(\frac{k-1}N) \\
      1 & 1 \\
    \end{bmatrix}
 \qquad\text{and}\qquad 
 \Lambda_{k} =     \begin{bmatrix}
      \lambda_+(\frac{k-1}N) & 0 \\  0 & \lambda_{-}(\frac{k-1}N) 
    \end{bmatrix} . 
\end{equation}
Let us also record that the eigenvalues of the deterministic transfer matrices  satisfy for any $t \in (0,1]$ and $z\in \C$ with $\Im z \ge 0$, 
\begin{equation} \label{def:lambda}
  \lambda_{\pm}(t) = \frac{\sqrt{t} J(z/\sqrt{t})^{\mp 1}}{ 2},
\end{equation}
where $J$ is the inverse Joukowsky map (see \eqref{eq:J}).
This map is continuous in the open upper half plane and extends continuously to its closure.  Hence for $z\in[-\sqrt{t},\sqrt{t}]$,  $\lambda_{\pm}$ are still well--defined by continuity from the upper half plane. 
From \eqref{def:V}, we verify that for $k\ge 2$, 
\begin{equation} \label{eq:delta}
  V_{k+1}^{-1} V_{k}
  =  \Id
  -  \delta_k\begin{bmatrix}
    1 & -1 \\
    -1  & 1
  \end{bmatrix}
  \quad
  \text{where}
  \quad
  \delta_k 
  \coloneqq \frac{\lambda_+(\frac{k}N) -  \lambda_+(\frac{k-1}N)}{\lambda_+(\frac kN) - \lambda_{-}(\frac kN)}
  =\frac{\sqrt{z^2 - \frac{k-1}{N}}-\sqrt{z^2 - \frac{k}{N}}}{2\sqrt{ z^2 - \frac{k}{N}}} .
\end{equation}
From Taylor expanding the numerator of $\delta,$ we can estimate $\delta_k = \O( |Nz^2-k|^{-1})$ uniformly on sets of $z \in \C$ and $k,N \in \N$ such that $|Nz^2-k|$ is large.  Indeed for any $z \in \mathscr{D}_H,$ and for any $k=1,\dots, N_H$,
we have $|Nz^2 - k| \geq \frac{\omega_N}{\sqrt{2}}$ (see \eqref{est0}).

\medskip

We can use the eigenvector matrices of the deterministic recurrence to approximately diagonalize the transfer matrix recurrence \eqref{eq:recurrence}. For any $n\ge 2$, 
\begin{equation*} 
 T_n \cdots T_2=V_{n+1} \left( \prod_{k=2}^n V_{k+1}^{-1} T_k V_k  \right) V_2^{-1}  . 
\end{equation*}
Let us define the matrices for $k\ge 2$, 
  \begin{equation} 
    \label{eq:fluctuations}
%    \begin{aligned}
\boldsymbol{\epsilon}_k   =  V_{k+1}^{-1}(\tilde T_k - T_k)V_k 
%\\
%&
= \frac{1}{\sqrt{2\beta N}}V_{k+1}^{-1}\begin{pmatrix} X_k &  \frac 12\sqrt{\frac{k-1}{N}} Y_k \\ 0 &0  \end{pmatrix}V_k , 
%\end{aligned}
\end{equation}
where we recall from \eqref{def:XY}: for $k\ge1$,
\begin{equation*} %\label{def:XY}
X_k = \frac{b_k}{\sqrt{2}} 
\qquad\text{and}\qquad 
Y_k = \frac{a_{k-1}^2 - \beta(k-1)}{\sqrt{2\beta(k-1)}}.
\end{equation*}
Note that the random variables $X_1, X_2, X_3, \dots $ and $Y_2, Y_3, \dots$ are all independent with mean 0, variance 1 and some moderate deviation estimates for the $Y_k$ are recorded in Lemma~\ref{lem:XY}. 

\medskip

Define for $k =1, \dots, N$,
\begin{equation} \label{Jlambda}
\rho_k(z)=  \frac{\lambda_{-}(\frac kN)}{\lambda_{+}(\frac kN)}  =  J( \tfrac{z}{\sqrt{k/N}})^2 . 
\end{equation} 
The quantities $\rho_k$ and $\delta_k$ measure the hyperbolicity of transfer matrices $\widetilde T_k$. 
The matrices $\boldsymbol{\epsilon}_k$ represents the (independent) noise at each step of the transfer matrix recurrence. 
Since $\widetilde T_k = V_k \Lambda_k V_k^{-1}$, we obtain
\begin{equation} \label{prod1}
 T_n \cdots T_2=V_{n+1}  \prod_{k=2}^n\left(V_{k+1}^{-1} \tilde T_k V_k - \boldsymbol{\epsilon}_k  \right) V_2^{-1}   = V_{n+1}\prod_{k=2}^n\left( V_{k+1}^{-1} V_k \Lambda_k - \boldsymbol{\epsilon}_k  \right) V_2^{-1} . 
\end{equation}
To exploit the hyperbolic feature of the transfer matrices, it will be useful to factor the first entry of the matrix on the RHS of \eqref{prod1}. This is recorded by the next lemma.

\begin{lemma} \label{lem:diag}
We have for any $n \ge 2$,
\[
 T_n \cdots T_2 =  \prod_{k=2}^{n}  \lambda_+(\tfrac{k-1}N)\big(1- \delta_k- \eta_{k,11} \big)   V_{n+1}\prod_{k=2}^n U_k V_2^{-1} ,
\]
where
\begin{equation} \label{def:eta1}
U_k = \begin{pmatrix} 
1 &   \eta_{k,12} \\   \eta_{k,21} & \rho_k - \eta_{k,22}
\end{pmatrix} 
 \qquad\text{and}\qquad 
 \eta_{k,11} =    \sqrt{\frac{1/2\beta}{Nz^2-k}}\left( X_k +  Y_k J\big(z \sqrt{\tfrac{N}{k-1}}\big) \right)
\end{equation}
Moreover, the random matrices $U_1, U_2, \dots$ are independent and given by \eqref{def:eta2} below. 
\end{lemma}

\begin{proof}
By \eqref{eq:fluctuations}, we verify from that  for any $k\ge 2$,  
\begin{equation*}
  \label{eq:he_epsy}
  \begin{aligned}
\boldsymbol{\epsilon}_k
 &=\frac{   \lambda_+(\frac{k-1}N)}{\sqrt{2\beta N}\big(\lambda_+(\frac kN) - \lambda_-(\frac kN) \big)} 
 \begin{bmatrix}
X_k + \breve{Y}_k   & \rho_{k-1}  X_k +\breve{Y}_k   \\
   -X_k -\breve{Y}_k & -\rho_{k-1}X_k -\breve{Y}_k \\
 \end{bmatrix},
 \end{aligned}
\end{equation*}
where we set (recall \eqref{def:lambda} and \eqref{def:V}), 
\[
\breve{Y}_k = \tfrac 12  Y_k  \sqrt{\tfrac{k-1}{N}} \lambda_+\left(\tfrac{k-1}N\right)^{-1} =Y_k J\big(z \sqrt{\tfrac{N}{k-1}}\big)
\qquad\text{and}\qquad
 \rho_k = \frac{  \lambda_-(\tfrac{k}N)}{  \lambda_+(\tfrac{k}N)} . 
\]
As for the mean,  let us observe that by \eqref{eq:delta}, we  have
\begin{equation*}
  \label{eq:he_mean}
  \begin{aligned}
    V_{k+1}^{-1} V_{k}\Lambda_k
    - \boldsymbol{\epsilon}_k 
    =   \lambda_+(\tfrac{k-1}N)
    \left(
    \begin{bmatrix}
      1-\delta_k  &   \rho_{k-1} \delta_k \\
      \delta_k & \rho_{k-1} (1-\delta_k)
    \end{bmatrix}
    +
    \sqrt{\tfrac{1/2\beta}{Nz^2-k}}
    \begin{bmatrix}
      -X_k -\breve{Y}_k   & -\rho_{k-1}  X_k -\breve{Y}_k   \\
      X_k +\breve{Y}_k & \rho_{k-1}X_k +\breve{Y}_k \\
    \end{bmatrix}
    \right).
  \end{aligned}
\end{equation*}
Hence, if we denote $\eta_{k,11} =  \sqrt{\frac{1/2\beta}{Nz^2-k}}\big( X_k + \breve{Y}_k \big)$ then we have $\boldsymbol{\epsilon}_{k,11} = \lambda_+\left(\tfrac{k-1}N\right)\eta_{k,11}$.  Furthermore, if we define 
\begin{equation} \label{def:eta2}
\eta_{k,21} = \frac{\delta_k + \eta_{k,11}}{1-  \delta_k - \eta_{k,11}} , \quad
\eta_{k,12} = \frac{  \rho_{k-1} \delta_k -  \sqrt{\frac{1/2\beta}{Nz^2-k}} \big( \rho_{k-1}  X_k +\breve{Y}_k \big)     }{1-\delta_k- \eta_{k,11}} ,
\quad
\eta_{k,22} = \frac{\rho_{k-1}(-\eta_{k,11}-\delta_k)}{1- \delta_k- \eta_{k,11}} + \eta_{k,12} ,\end{equation}
we obtain for any $k\ge 2$,
\[ \begin{aligned}
V_{k+1}^{-1} V_k \Lambda_k - \boldsymbol{\epsilon}_k 
& =  \lambda_+\left(\tfrac{k-1}N\right) (1 -\delta_k-\eta_{k,11}) \begin{pmatrix} 
1 &   \eta_{k,12} \\   \eta_{k,21} & \rho_{k-1}-\eta_{k,22} 
\end{pmatrix}  \\
&= \lambda_+\left(\tfrac{k-1}N\right) (1- \delta_k -\eta_{k,11}) U_k . 
\end{aligned}\]
Hence, by formula \eqref{prod1}, this completes the proof.
\end{proof}

From \eqref{eq:recurrence} and Lemma~\ref{lem:diag}, we obtain that
\begin{equation}  \label{rec2}
\begin{pmatrix} \Phi_{n}(z) \\ \Phi_{n-1}(z) \end{pmatrix}
= \biggl[ \prod_{k=2}^{n}  \lambda_+(\tfrac{k-1}N)\big(1- \delta_k - \eta_{k,11} \big) \biggr]   V_{n+1} \biggl[\prod_{k=2}^n U_k \biggr] V_2^{-1} \begin{pmatrix}   z- \frac{b_1}{ 2\sqrt{N\beta}}   \\ 1\end{pmatrix} . 
\end{equation}
Then, in order to obtain the asymptotics of the characteristic polynomial $\Phi_n(z)$ for large $n$, we need an approximation for  $\prod_{k=2}^n U_k$ where the random matrix $U_k$ are as in Lemma~\ref{lem:diag}.  
In fact, the precise form of the noise $\{ \eta_{k,12}, \eta_{k,21}, \eta_{k,22} \}_{k=1}^N$ is not important for our applications. What will be relevant is a set of estimates that are summarized in the next section.

\subsection{Control of product of hyperbolic transfer matrices} \label{sect:main} %{Reformulation of main theorems}

In this section, we use the notation from Lemma~\ref{lem:diag}.  
We begin by stating a result which will be key for proving Theorem~\ref{thm:planar}. 

\begin{proposition}\label{prop:rec}
Let $\alpha \in (0, \frac 19]$, and $\epsilon, \delta>0$ be sufficiently small so that  $\delta+\epsilon \le \frac \alpha 2$.
 Suppose that the random variables \eqref{def:eta1}--\eqref{def:eta2} satisfy for all $k\in\{1,\dots, N\}$,
   \begin{equation} \label{cond:eta}
    |\Exp {\eta_{k,ij}}| \le N^{2\alpha-1},
    \quad
    \Exp |{\eta_{k,ij}}|^2 \le N^{2\alpha-1},
    \quad
    |{\eta_{k,ij}}| \le N^{\alpha-1/2+ \epsilon} 
    \quad\As,
  \end{equation}
  and that $|\rho_k| \le 1- c N^{-\alpha}  $ for some absolute constant $c>1$. Then,   if $N$ sufficiently large, %(depending only on $\alpha, \beta, \delta, \epsilon, c, C$),
   there is an event $\A$ and a constant $c_\beta > 0$
  with $\P[\A^c] \le e^{-c_\beta N^\delta}$    on which
  \[
  \left\| {\textstyle \prod_{k=1}^N U_k } - \left(\begin{smallmatrix} 1 & 0 \\ 0 & 0 \end{smallmatrix}\right)  \right\| \le
  3  N^{\frac{15\alpha}{2} -1+ \epsilon + \frac{3\delta}{2}}   . 
  \]
\end{proposition}

The proof of Proposition~\ref{prop:rec}  is given in Section \ref{sec:rmp}. 
It relies on a perturbative expansion of $\prod_{k=p}^n U_k$ for $N\ge n > p \ge 1$  for which we can get control of the successive terms by induction. This control is achieved by decomposing the terms in the expansion as martingales and exploiting the moderate deviation estimates from Section~\ref{sec:moddev}. 
The key input lies in that the transfer matrices have a \emph{hyperbolic} character since $|\rho_k| \le 1- c N^{-\alpha} $  and the noise is sufficiently small.

\medskip

To apply Proposition~\ref{prop:rec}  in the context of Theorem~\ref{thm:planar}, by formula \eqref{rec2}, we must verify that the noise $\eta_{k,ij}$ satisfies the conditions \eqref{cond:eta} uniformly for all $z\in\mathscr{P}$. 
Note that the almost sure estimates only hold after truncating the random variables \eqref{def:XY} by conditioning on an event $\mathscr{T}_{N^\delta}$ of overwhelming probability, defined in \eqref{truncation}. Then, we verify in Appendix~\ref{sect:est} that these conditions are satisfied -- see Lemma~\ref{lem:planarest}. 
We obtain the following corollary of  formula \eqref{rec2} and Proposition~\ref{prop:rec}\footnote{We apply this proposition with $\alpha= 1/9$ and choosing $\delta = 2\epsilon/3 = 1/45$ to control the probability of failure of $\A$.}.

\begin{corollary} \label{cor:transfer}
With $\alpha = 1/9$ and $\delta= 1/45$, there exists an event $\A \subset \mathscr{T}_{N^\delta}$ with  $\P[\A^c] \le e^{-cN^\delta}$ such that on the event $\A$, it holds uniformly for all $z\in\mathscr{P}$,
\begin{equation*}
\begin{pmatrix} \Phi_{N}(z) \\ \Phi_{N-1}(z) \end{pmatrix}
= \biggl[ \prod_{k=2}^{N}  \lambda_+(\tfrac{k-1}N)\big(1- \delta_k - \eta_{k,11} \big) \biggr]  V_{N+1} \left[\begin{pmatrix} 1 & 0 \\ 0 & 0 \end{pmatrix}  + \O\big(N^{-\frac{1}{15}}\big) \right] V_2^{-1} \begin{pmatrix}   z- \frac{b_1}{ 2\sqrt{N\beta}}   \\ 1\end{pmatrix} . 
\end{equation*}
\end{corollary}

This reduces the problem to a scalar one. Indeed, since $\varphi_N(z) = \Phi_{N}(z) $, upon extracting the first entry from the above formula, we obtain that on the event $\A$,  uniformly for $z \in K\cap \mathscr{P}$,
\begin{equation} \label{charpoly1}
\varphi_N(z) =  \biggl[ \prod_{k=1}^{N}  \lambda_+(\tfrac{k}N)\big(1- \delta_k - \eta_{k,11} \big) \biggr]  \left( 1 + \O\big(N^{-\frac{1}{15}}\big) \right) ,
\end{equation}
where $K\subset \C$ is any fixed compact set. 
We have used that by \eqref{def:V},    $V_{N+1} = \O(1)$ and, 
\begin{equation} \label{bc}
V_2^{-1} \left(\begin{smallmatrix}   z- \frac{b_1}{ 2\sqrt{N\beta}}   \\ 1\end{smallmatrix} \right)   
=
\frac{1}{2\sqrt{z^2 - \tfrac{1}{N}}}
\left(
\begin{smallmatrix}
  2 & -z + \sqrt{z^2 - \tfrac{1}{N}} \\
  -2 & z + \sqrt{z^2 - \tfrac{1}{N}} \\
\end{smallmatrix}
\right)
 \left(\begin{smallmatrix}   z- \frac{b_1}{ 2\sqrt{N\beta}}   \\ 1\end{smallmatrix} \right)
=
 \begin{pmatrix} 1 \\ 0
 \end{pmatrix} + \O\biggl(\frac{1+ |b_1|}{|Nz^2|^{1/2}}\biggr) 
    \end{equation}
uniformly in $K\cap\big\{|z| \ge 2N^{-1/2}\big\}$.   
%it holds with probability at least $1-e^{-cN^\delta}$

 Thus, to complete the proof of Theorem~\ref{thm:planar},  it remains to obtain asymptotics of the product on the RHS of \eqref{charpoly1}.  
 First, let us remark that in the case where there is no noise ($\eta_{k,11}=0$ for all $k$), according to \eqref{eq:hermite}, we recover the Hermite polynomial asymptotics. This is the content from Proposition~\ref{prop:Hermite} where we show that our approximation is consistent with the classical Plancherel--Rotach expansion (see e.g. \cite{Deiftuni, Deiftstrong}). 
 The remainder of the proof relies on the semi--explicit coupling of the process $\{ \eta_{k,11}(z) \}_{k=1}^N$ with the GAF $  \mathfrak{g}_t(z).$% \eqref{eq:Wtz}. 

\begin{proposition}\label{prop:coupling1}
  With $\delta=1/45,$ there is an event $\A$ with $\P[\A^c] \le 2e^{-N^\delta}$ for all $N$ sufficiently large such that on $\A,$
  \[
 \prod_{k=1}^{N}  \lambda_+(\tfrac{k}N)\big(1- \delta_k - \eta_{k,11} \big)=    \pi_N(z)  \exp\left( -\sqrt{\frac{2}{{\beta}}}\mathrm{W}(z) - \frac{1}{\beta} \E\left[ \mathrm{W}(z)^2\right] \right)  \left( 1+ \O\big(N^{-\frac{1}{15}}\big) \right) . 
\]
where the error is uniform for all  $z \in K\cap  \mathscr{P}$ where $K\subset \C$ is compact.
\end{proposition}

The details of the proof of Proposition~\ref{prop:coupling1} are given in Section~\ref{sect:coupling1}. 
By combining Proposition~\ref{prop:coupling1} with the asymptotics \eqref{charpoly1}, we obtain that uniformly for $z \in K\cap \mathscr{P}$,
\begin{equation*}
\varphi_N(z) = \pi_N(z) \frac{\exp\left( - \sqrt{\frac{2}{{\beta}}} \mathrm{W}(z)\right)}{\E \big[ \exp\left( - \sqrt{\frac{2}{{\beta}}} \mathrm{W}(z) \right) \big]}  \left( 1 + \O\big(N^{-\frac{1}{15}}\big) \right) ,
\end{equation*}
where we have used that since $\mathrm{W}$ is a Gaussian process, 
$\exp\left( \E[\mathrm{W}(z)^2]/\beta\right) = \E \big[ \exp\left( -\sqrt{\frac{2}{{\beta}}}\mathrm{W}(z) \right) \big]$. 
This completes the proof of  Theorem~\ref{thm:planar}. 

\medskip

The proof of Theorem~\ref{thm:real} relies on similar ideas, but there is major difference in that for $z\in[-1,1]$,  the transfer matrices $U_k$ loose their hyperbolic character for $k \approx Nz^2$, near the turning point. In conjunction, the noise also does not satisfy the conditions \eqref{cond:eta} uniformly along the recurrence. 
Nevertheless, refining the method from the proof of Proposition~\ref{prop:rec}, we obtain the following approximation result which will be crucial to deduce the asymptotics of the Gaussian $\beta$-ensemble recurrence near the turning point. 

\begin{theorem} \label{thm:main}
  Let $N\in\N,$ $R > 0$ and $\omega_N= N^{1/3}(\Omega \log N)^{2/3}$ with  $\Omega = o(N^{1/15}/\log N)$. 
Assume $\{U_k\}_{k=1}^N$ are independent random matrices as in  \eqref{def:eta1}--\eqref{def:eta2}  where $\rho_k \in\C$ is deterministic,  $|\rho_k| \le  \exp\left( - c_0 \sqrt{\frac{\omega_N + \hat{k}}{N}}\right)$ with $c_0>1$  and the random variables $\{ \eta_{k,ij}\}_{i,j \in\{1,2\}}$ satisfy for all $k\in\{1,\dots, N\}$, 
\[
  |\Exp \eta_{k,ij}| \leq \frac{C}{\omega_N + \hat k},
  \quad
  \Var(\eta_{k,ij}) \leq \frac{C}{(\omega_N + \hat k)},
  \quad
  \text{and}
  \quad
  |\eta_{k,ij}| \leq  \sqrt{\frac{R\Omega \log N}{\omega_N + \hat k}},
\]
for some $C > 0$ and where $\hat{k} = N - k. $ 
Then, there exists a small constant $c_R>0$ such that if $N$ is sufficiently large $($depending on $C,R$ and $\Omega)$, it holds for any $\varepsilon \ge (\Omega \log N)^{-1/2}$, 
\begin{equation} \label{eq:main}
\P\left[ \bigg\| \prod_{k=1}^ N U_k - \begin{pmatrix} 
1 &   0 \\   0 &  0
\end{pmatrix} \bigg\| \ge  \varepsilon
%\frac{c_1}{\sqrt{\log N}}
\right] \le N^{-c_R\varepsilon^2\Omega} + c_R^{-1} N^{4-R\Omega}  .
\end{equation}
\end{theorem}

{Let us emphasize that in the formulation of Theorem~\ref{thm:main}, the parameter $\Omega>0$ is allowed to vary with $N$ while $C, R, c_R$ are fixed constants.}
The proof of Theorem~\ref{thm:main} is the central technical contribution of this paper, and it is given in Section~\ref{sec:linearized}. For the convenience of the readers, the general strategy of the proof is explained in Section~\ref{sect:strategy}. 
This strategy also relies on the perturbative expansion developed in  Section~\ref{sect:gf} but it differs significantlly from the proof of Proposition~\ref{prop:rec} in that it is substantially sharper than the induction argument in Section~\ref{sect:ind}.  This improvement is needed to get Theorem~\ref{thm:main} to hold to optimal scales.

\medskip

For applications to Gaussian $\beta$-ensembles, we need to truncate the noise to apply Theorem~\ref{thm:main}. 
This truncation procedure relies on the fact that the random variables $\chi_{\alpha}$ from \eqref{def:trimatrix} have uniform exponential tails and it is explained in the Appendix~\ref{sect:est}.
In particular, it relies on  Lemma~\ref{lem:noise} with $\mathrm{S}= R \Omega \log N$ and $R \Omega \ge r_\beta$. 
Hence, by formula \eqref{rec2}, Theorem~\ref{thm:main} with $N = N_H(z)$ implies that with overwhelming probability, for any $z\in\mathscr{D}_H$, 
\begin{equation} \label{charpoly4}
\begin{pmatrix} \Phi_{N_H}(z) \\ \Phi_{N_H-1}(z) \end{pmatrix}
= \biggl[ \prod_{k=2}^{N_H}  \lambda_+(\tfrac{k-1}N)\big(1- \delta_k - \eta_{k,11} \big) \biggr]  
\begin{pmatrix}  \lambda_+(\frac{N_H}{N}) \\ 1   \end{pmatrix} \big(1+\O(\varepsilon) \big) .  
\end{equation}
By \eqref{Rtrunc}, adjusting the constant $R$, our control of the error term $\O(\varepsilon)$ is exactly the same as \eqref{eq:main} and it is uniform for all $z\in\mathscr{D}_H$. 
In particular, note that the condition $\Omega \le N^{\delta/6}$ from Definition~\ref{def:hyper} implies that $\Omega = o\big(\frac{N_H^{1/15}}{\log N_H} \big) $ uniformly for all $z\in\mathscr{D}_H$; it also implies that $\varepsilon \geq (\log N)^{-1/2}N^{-\delta/12}$. 
As for the initial condition, we have from \eqref{bc} that for $z\in \mathscr{D}_H$
\[
V_2^{-1} \left(\begin{smallmatrix}   z- \frac{b_1}{ 2\sqrt{N\beta}}   \\ 1\end{smallmatrix} \right)   
=
 \begin{pmatrix} 1 \\ 0
 \end{pmatrix} + \O\biggl( N^{-\delta/4}\biggr) 
=
 \begin{pmatrix} 1 \\ 0
 \end{pmatrix} + \O\biggl( \varepsilon\biggr) 
\]
with probability $1-e^{-N^{\delta/4}}.$  As $|\lambda_-| \le |\lambda_+|$ uniformly for $z\in \mathscr{D}_H$ (see \eqref{def:V}--\eqref{def:lambda}) so that 
\[
 V_{N_H+1} \left[\begin{pmatrix} 1 & 0 \\ 0 & 0 \end{pmatrix}  + \O(\varepsilon) \right] V_2^{-1} \begin{pmatrix}   z- \frac{b_1}{ 2\sqrt{N\beta}}   \\ 1\end{pmatrix}  = \begin{pmatrix}  \lambda_+(\frac{N_H}{N}) \\ 1   \end{pmatrix} \big(1+\O(\varepsilon) \big) .
\]

Like for $z \in \mathscr{P}$,  we then show that this scalar process is well-approximated by $  \mathfrak{g}_t(z).$
In Section~\ref{sect:coupling2}, we obtain the following approximation.\footnote{We can also obtain a coupling as processes indexed by $\big(z\in \mathscr{D}_H, n\in\{1,\dots, N_H(z)\}\big)$. For simplicity, we only state our result for $n=N_H(z)$.}. 

%  This allows us to obtain the required coupling between the characteristic polynomial of the Gaussian $\beta$-ensemble and the Gaussian analytic function from definition~\ref{def:GAF}. As our next proposition shows, this coupling applies both for $z\in\C$ separated from the support of the spectral measure $[-1,1]$ and  for $z \in \mathscr{D}_H$ provided that $n \le N_H(z)$ (see definition~\ref{def:hyper}).

\begin{proposition}\label{prop:coupling2}
  %There is a coupling of   $\left\{ (X_k,Y_k) \right\}_{k=1}^N$ with a pair of independent standard Brownian motions $( (\mathbf{X}_t , \mathbf{Y}_t) : t \in [0,1])$ and an 
  There is an event $\mathscr{G}$ with $\P[\mathscr{G}^c] \le 2e^{-N^\epsilon}$ for all $N$ sufficiently large and for a small $\epsilon>0$ $($depending on $\delta>0)$ such that on this event, it holds uniformly for $z\in \mathscr{D}_H$, 
\begin{equation}   \label{charpoly7}
\biggl[ \prod_{k=2}^{N_H}  \lambda_+(\tfrac{k-1}N)\big(1- \delta_k - \eta_{k,11} \big) \biggr] \begin{pmatrix}  \lambda_+(\frac{N_H}{N}) \\ 1   \end{pmatrix}
=  \frac{\exp\big(\sqrt{\frac{2}{{\beta}}}   \mathfrak{g}_t(z) \big)}{\E\big[\exp\big(\sqrt{\frac{2}{{\beta}}}  \mathfrak{g}_t(z)\big)\big]}
\begin{pmatrix} \pi_{N_H}(z) \\  \pi_{N_H-1}(z)  \end{pmatrix}   \left(1+  \O\left(N^{-\epsilon} \right) \right)
\end{equation}
with $t(z)= N_H(z)/N$. 
%Moreover $\big\{   \mathfrak{g}_s(z) \big\}_{\left\{s\ge 0, z\in \C \setminus[-\sqrt{s},\sqrt{s}]\right\}}$ is a Gaussian process which is realized as
%\begin{equation} \label{Wtprocess}
%  \mathfrak{g}_s(z)  = \int_0^s \frac{\d \mathbf{X}_u+ J(z/\sqrt{u}) \d\mathbf{Y}_u}{\sqrt{z^2-u}} . 
%\end{equation}
%In particular, for any $z\in \mathscr{D}_H$, the RHS of \eqref{charpoly7} (including the error term) is independent from $\F_{>N_H} =  \sigma\big\{  (X_k,Y_k) : k >N_H\big\}$. 
\end{proposition}
%The covariance structure of the Gaussian process $\big\{   \mathfrak{g}_s(z) \big\}$ is given by formula \eqref{Wcov} below. 
Hence, combining the asymptotics \eqref{eq:main}--\eqref{charpoly4} and Proposition~\ref{prop:coupling2}, this concludes the proof of Theorem~\ref{thm:real}.  %In the remainder of this paper, we go over the details of the proofs of the results presented in this section. 

\section{Moderate deviations for perturbations products of random matrices} 
\label{sec:rmp}

\subsection{General framework} \label{sect:gf}

In what follows, we suppose that $\left\{ X_n = (U_n, V_n) \right\}$ is a sequence of independent random variables, where each $U_n$ and $V_n$ is a $\d \times \d$ random matrix.    Let $\filt_{n} = \sigma\big(X_1, \dots, X_n\big)$ for any $n\ge 1$.  We think of $U_n$ as a random perturbation of $V_n$ and we would like to compare $\prod_{n=1}^NU_n $ with $\prod_{n=1}^NV_n $. To this end, we develop successive approximations for $\prod_{n=1}^NU_n $ for which we obtain good moderate deviations control by using martingale arguments. 
We first give general estimates (Proposition~\ref{prop:ta}) which are of independent interest. Then, in the context of Proposition~\ref{prop:rec}, we obtain more specific estimates (Proposition~\ref{prop:induction})  by using the hyperbolicity of $U_k$ and the smallness of the noise $\eta_{k,ij}$.

\medskip

For any $n \geq p \geq 1$ and $j \in \{0,1,2,\dots\}$,  define 
\begin{equation}\label{eq:psi}
  \psi^{(j)}_{n,p} = \sum_{|S|=j}\prod_{k=p}^n \left\{ (U_k - V_k)\one_{k \in S} + V_k \one_{k \not\in S} \right\},
\end{equation}
where the sum is over all subsets $S$ of $\left\{ p,p+1, \dots, n \right\}$ having cardinality $|S|=j.$
It follows that $\psi^{(0)}_{n,p} = \prod_{j=p}^n V_j$ and we obtain a perturbative expansion for the the product of $U_k$, 
\begin{equation*}\label{eq:psiexpansion}
  \prod_{k=p}^n U_k
  = \sum_{S}\prod_{j=p}^n \left\{ (U_j - V_j)\one_{j \in S} + V_j \one_{j \not\in S} \right\}
  =\sum_{k=0}^\infty \psi^{(k)}_{n,p} ,
\end{equation*}
with the first sum over all subsets $S$ of $\left\{ p, \dots, n \right\}.$ 
We will also use the shorthand $\psi^{(>j)}_{n,p} = \sum_{\ell=j+1}^\infty \psi^{(\ell)}_{n,p}$ for any  $j \in \{0,1,2,\dots\}$. This allows us to express the product of $U_k$  for any degree of accuracy $j \in \N$  as
\begin{equation}\label{eq:psiTaylor}
  \prod_{k=p}^n U_k
  =\sum_{k=0}^j \psi^{(k)}_{n,p}
  +
  \psi^{(>j)}_{n,p}.
\end{equation}

Then, we can write a recurrence that holds for any $0 \leq \ell \leq j$, 
\begin{equation}\label{eq:psi1step}
  \begin{aligned}
  &\psi^{(j+1)}_{n,p} = \sum_{k=p}^n %U_n \dots U_{k+1} 
  \psi^{(j-\ell)}_{n,k+1} (U_k - V_k) \psi^{(\ell)}_{k-1,p}, \\
  &\psi^{(>j)}_{n,p} = \sum_{k=p}^n %U_n \dots U_{k+1} 
  \psi^{(\geq j-\ell)}_{n,k+1} (U_k - V_k) \psi^{(\ell)}_{k-1,p},
  \end{aligned}
\end{equation}
which follows from decomposing \eqref{eq:psi} according to the location $k$ of the $(\ell+1)$ largest element of $S$. In particular, taking $\ell=j$, we can use \eqref{eq:psi1step} to control the error in approximation $\psi^{(>j)}$ in terms of $\psi^{(j)}$ provided that we have an a priori control of the norm of $U_n \cdots U_{k+1}$. The following proposition allows  us to quantify this control. 
This general proposition is directly used in the context of the proof of Theorem~\ref{thm:real} in Section~\ref{sec:shortblock}.
% Then, using hyperbolicity, we bootstrap this estimate in Section  to get control of the norm of the full product  $\prod_{k=1}^N U_{k}$. 

\begin{proposition} \label{prop:ta}
  Fix $n \in \N$ and $j \in \{0,1,2,\dots\}$.  Define the following deterministic quantities:
  \[
    \begin{aligned}
      u = \max_{1 \leq p \leq r \leq n} \|  \textstyle{\prod_{k=p}^r \Exp U_k} \|
      \quad
      &\text{and}
      \quad
      \mu = \sum_{k=1}^n \| \Exp(U_k - V_k) \|. \\
    \end{aligned}
  \]
  We shall suppose that $u\mu < \frac{1}{2}$ and choose $\Delta$ and $\sigma$ such that
  \[
    \begin{aligned}
      \Delta \geq \max_{1 \leq p \leq n} \max\{ \| U_p - \Exp U_p \|, \| { V_p -\Exp V_p}\| \} \ , 
      \qquad
      \sigma^2 \geq \sum_{p=1}^n \Exp(\| U_p - \Exp U_p\|^2 + \| U_p - V_p \|^2) \ ,
    \end{aligned}
  \]
  and $  8 u^2 \sigma^2 +  \Delta u   \leq \frac{1}{8\log n}$. 
  Then, there is  an event $\A \in  \F_n$
  such that for any event $\mathcal{E} \in \F_n$ on which  
$\big\VERT\max_{1 \leq p \leq n} \| \psi^{(j)}_{p,1} \| \one_\mathcal{E}\big\VERT_2 \leq D,$
we have
  \[
    \big\VERT \psi^{(>j)}_{n,1} \one_{\A \cap \mathcal{E}} \big\VERT_1 \lesssim_\d  Du(\sigma+\mu), \quad
    \text{and}
    \quad
    \Pr\left[ \A^c \cap \mathcal{E} \right] \le  6 \d \exp\left(-\frac{1}{256 u^2\sigma^2} \wedge   \frac{ \sigma^2}{\Delta^2}  \right) .  \]
\end{proposition}

\begin{proof}
  We begin by giving a martingale decomposition for $\psi^{(>j)}_{n,1}$. For any $1 \leq p \leq n$, we define
  \begin{equation} \label{AM1}
    \begin{aligned}
    A_p &= \sum_{k=1}^p \Exp\left[ U_n \dots U_{k+1} \right] \Exp[ U_k - V_k] \psi^{(j)}_{k-1,1}
=  A_{p-1} + \Exp\left[ U_n \dots U_{p+1}\right]\Exp[ U_p - V_p] \psi^{(j)}_{p-1,1} , \\
    M_{p} &= M_{p-1} + \Exp\left[ U_n \dots U_{p+1}\right]\left( (U_p - \Exp[U_p])\psi^{(>j)}_{p-1,1} + (U_p-V_p-\Exp[U_p-V_p])\psi^{(j)}_{p-1,1} \right),
    \end{aligned}
  \end{equation}
  where $M_0 = A_0= 0.$  Then $M_p$ is an $\filt_p$--martingale, and $A_p$ is $\filt_p$--predictable. By decomposing the sets $S$ in \eqref{eq:psi} according to whether or not $p \in S$, we may write
  \[
    \psi_{p,1}^{(>j)} = U_p\psi^{(>j)}_{p-1,1} + (U_p-V_p)\psi^{(j)}_{p-1,1} 
  \]
so that
  \[
    M_p + A_p = M_{p-1} + A_{p-1} + \Exp\left[ U_n \dots U_{p+1}\right]\left( \psi^{(>j)}_{p,1} - \Exp[U_p]\psi^{(>j)}_{p-1,1} \right).
  \]
 By independence of $U_k$, this makes a telescoping sum, from which we conclude that for all $p=1,\dots, n$
 \begin{equation} \label{AM2}
   M_p + A_p  =  \Exp\left[ U_n \dots U_{p+1}\right] \psi^{(>j)}_{p,1} . 
 \end{equation}
 
   For $S > 0$, let us introduce the stopping times 
  \[
    T_j = \inf\left\{ p \ge 1 : \|\psi^{(j)}_{p,1}\| \geq S \right\} , 
    \qquad
    T_{>j} = \inf\left\{p \ge 1  : \|\psi^{(>j)}_{p,1}\| \geq S \right\} , 
  \]
 and let $M^T_p = M_{p \wedge T}$, $A^T_p = A_{p \wedge T}$ for $p \in \{1,\dots, n\}$ with  {$T =    T_j  \wedge    T_{>j}$}.   Then, we verify from \eqref{AM1} that for all $p \in \{1,\dots, n\}$, 
  \[
    \| M^T_{p} - M^{T}_{p-1} \| \leq  3 u\Delta S,
  \]
  and
  \[
    \sum_{p=1}^n \Exp\left[ \| M_p^T - M_{p-1}^T\|^2 \vert \filt_{p-1} \right] \leq {4} u^2\sigma^2S^2 . 
  \]
 By applying the estimate \eqref{eq:FT} to the martingale $\{ M_p^T \}_{p= 1}^n$ with $\Sigma^2 = 4 u^2\sigma^2S^2$ and $\alpha=3 u\Delta S$,  we obtain for all $ t,S \geq 0$, 
  \begin{equation} \label{eq:MTtail}
    \Pr\left[ \|M_n^T\|   \ge t    \right]
    \le 2\d \exp\left(- \frac{t^2/(2 S u)}{ 4 u \sigma^2S +  \Delta t  } \right) . 
  \end{equation}
 
Since we assume that $u\mu\le \frac 12$, we can also bound the predictable part uniformly by 
  \[
    \| A_n^{T} \| \leq u\mu S \leq S/2.
  \]
   Hence applying \eqref{eq:MTtail} and using that by \eqref{AM2}, $M_n^T + A_n^T = \psi^{(>j)}_{n,1}$ on the event $\{T_{>j} = n, T_j \geq n\}$,  we obtain 
  \[
    \Pr\left[  T_{>j} = n, T_j \geq n  \right]
    \leq \Pr\left[  \|M_n^T\|   \geq S/2 \right]
    \leq 2\d \exp\left(-\rho^{-1}/4\right) . 
  \]
  where $ \rho=  8 u^2 \sigma^2 +  \Delta u  $.
  Furthermore, we can use that the previous bound is uniform in $n \in\N$ to conclude that
  \begin{equation*} 
  \begin{aligned}
    \Pr\left[ T_j \geq n, T_{>j} < n \right]
    & \leq 
    \sum_{p=1}^{n-1}
    \Pr\left[ T_j \geq p, T_{>j} = p \right]
    \leq 2\d n  \exp\left(-\rho^{-1}/4\right)  \\
    & \le 2\d   \exp\left(-\rho^{-1}/8\right)  
    \end{aligned}
  \end{equation*}
where we have used that the condition $\rho^{-1} \ge 8 \log n$ for the last bound. 
By assumption, the event $\mathcal{E}$ satisfies $\VERT\max_{1 \leq p \leq n} \| \psi^{(j)}_{p,1} \| \one_\mathcal{E}\VERT_2 \leq D$ so that  by \eqref{eq:Xtail}, we obtain
\[  \begin{aligned}
 \Pr[ \{T <n\} \cap   \mathcal{E}] 
 &\le \Pr\left[ T_j \geq n, T_{>j} < n \right] +   \Pr\left[ \{ T_j < n \} \cap  \mathcal{E}\right] \\
  &\le  2\d   \exp\left( - \tfrac{\rho^{-1}}{8}\right)   +  2 \exp\left( -\tfrac{S^2}{D^2} \right)  
  \end{aligned}\]

  The RHS of the estimate \eqref{eq:MTtail} is increasing as a function of $S$, and therefore for any $t \geq 0$ and for any $S\ge 2 t$, we have that
  \[
     \exp\left( - \frac{\rho^{-1}}{8}\right)
     =
 \exp\left(- \frac{t^2/(4 S u)}{ 4 u \sigma^2S +  \Delta t  } \right) \bigg|_{S=2t}
     \leq
     \exp\left(- \frac{t^2/(4 S u)}{ 4 u \sigma^2S +  \Delta t  } \right).
  \]
  This implies that for any choice of $t \geq 0$ and for any $S\ge 2 t$ %$t\le \frac{S}{2\sqrt{2}} \le  \frac{\sigma D}{2\sqrt{2}\Delta}$, 
\[\begin{aligned}
  \Pr\left[ \{ \| M_n \| \geq t \} \cap   \mathcal{E}\right]  
  &\le  \Pr\left[ \| M_n^T \| \ge t \right]  +  \Pr\left[ \{ T< n \} \cap  \mathcal{E}\right] \\
  & 
  \leq4\d \exp\left(- \frac{t^2/(4 S u)}{ 4 u \sigma^2S +  \Delta t  } \right)  +  2 \exp\left( -\frac{S^2}{D^2} \right).
  %\\
  %&\le  4\d \exp\left(- \frac{t^2/2}{ 4 u \sigma^2S^2 +  u\sigma Dt  } \right)  +  2 \d \exp\left( -\frac{S^2}{D^2} \right) . 
\end{aligned}\]
We now pick $S^2 = \frac{Dt}{4u\sigma}$ and restrict $0 \leq t \leq  \frac{D}{16 u\sigma} \wedge   \frac{4^2 u \sigma^3 D}{\Delta^2}$, so that  $\Delta S \le 2  \sigma D $ and 
\begin{equation}
  \label{eq:MTtail2}
  \begin{aligned}
  \Pr\left[ \{ \| M_n \| \geq t \} \cap   \mathcal{E}\right]  
  & \le  
  4\d \exp\left(- \frac{t^2/4}{4 u^2 S^2 \sigma^2 + 2 u\sigma D t } \right)  +  2 \exp\left( -\frac{t/4}{4 u\sigma D} \right)\\
  & \le  
  4\d \exp\left(- \frac{\alpha t/4}{3 D u \sigma } \right)  +  2 \exp\left( -\frac{t/4}{4 u\sigma D} \right) \\
  & \le  
  (4\d +  2)\exp\left( -\frac{t}{4^2u\sigma D}\right).
  %\\
  %\\
  %\\
  %&\le  4\d \exp\left(- \frac{t^2/2}{ 4 u \sigma^2S^2 +  u\sigma Dt  } \right)  +  2 \d \exp\left( -\frac{S^2}{D^2} \right) . 
\end{aligned}
\end{equation}
Let $\A =\left\{ \| M_n \| \le  \frac{D}{16 u\sigma} \wedge   \frac{16 u \sigma^3 D}{\Delta^2} \right\} $. The estimate   \eqref{eq:MTtail2} implies that  $\VERT M_n \one_{\mathcal{A} \cap \mathcal{E}}\VERT_1 \lesssim_\d Du\sigma$ 
and that
\[
    \Pr[  \mathcal{A}_n^c \cap \mathcal{E} ]  \le  6 \d \exp\left(-\frac{ \frac{D}{ 4^2 u\sigma} \wedge   \frac{ 16 u \sigma^3 D}{\Delta^2}  }{4^2 u\sigma D} \right)  = 6 \d \exp\left(-\frac{1}{256 u^2\sigma^2} \wedge   \frac{ \sigma^2}{\Delta^2}  \right) . 
\]

  Finally by \eqref{AM1}, \eqref{AM2}  and  since the predictable part satisfies $\|A_n \| \le u \mu  \max_{1 \leq p \leq n} \| \psi^{(j)}_{p,1}\| $,  we conclude that
  \[ \begin{aligned}
    \VERT \psi_{n,1}^{(>j)} \one_{\mathcal{A} \cap \mathcal{E}}\VERT_1 
    & \leq    \VERT M_n \one_{\mathcal{A} \cap \mathcal{E}}\VERT_1
    +\VERT A_n \one_{ \mathcal{E}}\VERT_1 \\ 
  &  \lesssim_\d Du\sigma + u\mu \VERT \max_{1 \leq p \leq n} \| \psi^{(j)}_{p,1} \|\one_{\mathcal{E}} \VERT_2,
\end{aligned}  \]
  where we used that if $|X|\le Y$ almost surely, then $\VERT X \VERT_1 \lesssim \VERT Y\VERT_2$. This completes the proof. 
\end{proof}

\subsection{Proof of Proposition~\ref{prop:rec}} \label{sect:ind}

Let us now apply the formalism from Section~\ref{sect:gf} to the case where the transfer matrices $U_k$ are of the form \eqref{def:eta1} and the noise $\eta_{k,ij}$ is sufficiently small. We allow $\eta_{k,ij}$ to grow along the sequence, but we must have a diffusive scaling.
% relationship between mean and variance.

\begin{assumption}\label{diffusive:ass}
  The sequence of random matrices $\left\{ U_k :1 \leq k \leq N \right\}$ are independent,
  and there is $\mathrm{S} \ge 0$ and  a sequence $a_k \geq 0$ so that   for all $1 \leq k \leq N,$
  \[
    |\Exp \eta_{k,ij}| \leq a_k
    \quad
    \Exp |\eta_{k,ij}|^2 \leq a_k
    \quad
    \text{and}
    \quad
    |\eta_{k,ij}| \leq \mathrm{S} \sqrt{ a_k} \, \As . 
  \]
\end{assumption}

It is crucial to use that the random matrices $\{U_k\}$  have a hyperbolic character -- they map a small ball around $\left( \begin{smallmatrix} 1 \\ 0 \end{smallmatrix} \right)$ in projective space into itself.\footnote{In effect, we have $U_k \approx \left( \begin{smallmatrix} 1 & 0 \\ 0 & 0 \end{smallmatrix} \right)$. Such ``invariant cone conditions'' are common in the literature on Lyapunov exponents, appearing implicitly in early works like \cite{Furstenberg} and much more explicitly \cite{Hafouta, Dubois08, Dubois09, Sturman}.} 
 We measure the hyperpolic character of the transfer matrices $U_k$  by asking that the product of the $U_{k,22}$ entries tends to be small.
 We introduce $\filt_{n,p}$ for the $\sigma$--algebra $\sigma( U_k : p \leq k \leq n)$  with $N\ge n\ge p\ge 0$. 
Let $b_k \geq 0$ be a decreasing sequence and define for any integers $N\ge n>p \ge 1$, the events
 \begin{equation}\label{hyperbolic:ass}
  \mathscr{B}_{n,p}
  =
  \left\{ 
    |\psi^{(0)}_{\ell,k,22}| 
    \leq C\exp(-(\ell-k)b_k)
    : p \leq k \leq \ell \leq n
  \right\} . 
  \end{equation}

\paragraph{Recurrence.} 
We take $V_k = \diag(U_k)$ so that for all $k \ge 1$, 
\begin{equation} \label{antidiag}
U_k - V_k = 
\begin{pmatrix} 0 & \eta_{k,12} \\
  \eta_{k,21} & 0 \\ \end{pmatrix} . 
\end{equation}
This makes $U_k - V_k$ small and also give the matrices $\psi^{(j)}$ defined in \eqref{eq:psi} an \emph{alternating structure}, which we summarize in the following lemma.  
\begin{lemma}\label{lem:psistructure}
  For any integer $j \geq 0,$
  \begin{equation}
  \psi^{(2j+1)}_{n,p}
  =
  \begin{bmatrix}
    0 & \psi^{(2j+1)}_{n,p,12} \\
    \psi^{(2j+1)}_{n,p,21} & 0 
  \end{bmatrix}
  \quad\text{and}\quad
  \psi^{(2j)}_{n,p}
  =
  \begin{bmatrix}
    \psi^{(2j)}_{n,p,11} & 0\\
    0 & \psi^{(2j)}_{n,p,22}
  \end{bmatrix}.
  \label{eq:HBpsi}
\end{equation}
For the case of $j=0,$ we further have
\[
  \psi^{(0)}_{n,p,11} = 1 
  \quad
  \text{and}
  \quad
  \psi^{(0)}_{n,p,22} = \textstyle{\prod_{k=p}^n} U_{k,22}.
\]
\end{lemma}
\begin{proof}
Since $V_k = \diag(U_k)$, the matrix $\psi^{(0)}_{n,p} = \prod_{k=p}^n V_k$ are diagonal.
  The key point is that by \eqref{eq:psi1step} with $\ell=0$, we can express for any $j,p,n \geq 0,$
  \begin{equation} \label{eq:psijstart}
    \psi^{(j+1)}_{n,p} = \sum_{k=p}^n %U_n \dots U_{k+1} 
    \psi^{(j)}_{n,k+1} (U_k - V_k) \psi^{(0)}_{k-1,p}.
  \end{equation}
  Then, using    \eqref{antidiag},  the proof follows by induction on $j \in \{0,1,2,\dots\}$.   
\end{proof}

By \eqref{eq:psijstart}, upon extracting the first column, we obtain
  \[
    \psi^{(j+1)}_{n,p}\begin{pmatrix} 1 \\ 0 \end{pmatrix}
    =
    \sum_{k=p}^n %U_n \dots U_{k+1} 
    \psi^{(j)}_{n,k+1}\begin{pmatrix} 0 \\ 1 \end{pmatrix} \eta_{k,21}.
  \]
  This shows that  $\psi^{(j+1)}_{n,p}\left(\begin{smallmatrix} 1 \\ 0 \end{smallmatrix}\right)$ has a single nonzero entries,  the location of which varies according to parity (see Lemma \ref{lem:psistructure}) -- the same holds for  $\psi^{(j+1)}_{n,p}\left(\begin{smallmatrix} 0 \\ 1 \end{smallmatrix}\right)$ with reversed parity.  Letting $\psi^{(j)}_{n,p,*k}$ be the nonzero entry in the $k$--th column for $k=1,2$, we have
  \begin{equation}\label{eq:psijbackward}
    \psi^{(j+1)}_{n,p,*1}
    =
    \sum_{k=p}^n %U_n \dots U_{k+1} 
    \psi^{(j)}_{n,k+1,*2}
    \eta_{k,21}.
  \end{equation}
%  If, in contrast, we extract the first row of $ 
%  \psi^{(j+1)}_{n,p} $
%  then
%  \begin{equation*}%\label{eq:psijstart}
%    \begin{pmatrix} 1 &  0 \end{pmatrix}\psi^{(j+1)}_{n,p}
%    = \sum_{k=p}^n %U_n \dots U_{k+1} 
%    \begin{pmatrix} 1 &  0 \end{pmatrix}\psi^{(0)}_{n,k+1} (U_k - V_k) \psi^{(j)}_{k-1,p}
%    = \sum_{k=p}^n %U_n \dots U_{k+1} 
%    \psi^{(0)}_{n,k+1,11} \eta_{k,12} \begin{pmatrix} 0 & 1 \end{pmatrix}\psi^{(j)}_{k-1,p}.
%  \end{equation*}
In a similar fashion, if we extract the first (or second) row of $ \psi^{(j)}_{n,p} $ and let $\psi^{(j)}_{n,p,k*}$ be the single nonzero entry in the $k$--th row for $k=1,2$, we obtain
  \begin{equation}\label{eq:psijforward}
    \psi^{(j+1)}_{n,p,1*}
    =
    \sum_{k=p}^n %U_n \dots U_{k+1} 
    %\psi^{(0)}_{n,k+1,11} 
    \eta_{k,12}\psi^{(j)}_{k-1,p,2*} . 
  \end{equation}

  In the case that $j$ is even, we therefore can express both nonzero entries of $\psi^{(j+1)}$ through\eqref{eq:psijforward} and \eqref{eq:psijbackward}.  In the case that $j$ is odd, this gives two expressions for the first entry and we also have 
  \begin{equation}\label{eq:simple22}
    \psi^{(j+1)}_{n,p,22} 
    = 
    \sum_{k=p}^n %U_n \dots U_{k+1} 
    \psi^{(j)}_{n,k+1,21}\eta_{k,12}\psi^{(0)}_{k-1,p,22} . 
  \end{equation} 

Employing \eqref{eq:psijbackward}--\eqref{eq:simple22}, our next proposition provide some simple moderate deviation bounds for $\psi^{(j+1)}$ which are effective once control has been given on $\psi^{(j)}.$

%Employing \eqref{eq:psijbackward}, \eqref{eq:psijforward},  and \eqref{eq:simple22},

 \begin{proposition}\label{prop:induction}
  Assume \ref{diffusive:ass} and let us fix integers  $j \geq 0$ and $n > p \ge 1$.   Let $\left\{ c_k \right\} ,\{ c_k' \} $  be non--negative sequences and $\mathscr{C}_1 = \left\{ |\psi^{(j)}_{k-1,p,2*}| \le c_k ,\  |\psi^{(j)}_{n,k+1,*2}|  \leq c_k' : p \le k \le n \right\}.$ There is an absolute constant $\varkappa >0$ so that for any $R > 0,$ there is an event $\A_1 \in \filt_{n,p}$ for which  $\Pr[ \A_1^c ] \geq 4e^{-R}$ and on which
  \begin{equation} \label{control1}
    \Big\VERT  \sup_{p \leq k \leq n}  | \psi^{(j+1)}_{k,p,1*}|\one_{\A_1 \cap \mathscr{C}_1} \Big\VERT_2 
    \leq \varkappa \Sigma + \sum_{k=p}^n a_kc_k 
     \quad
     \text{and}
     \quad
        \Big\VERT  \sup_{p \leq k \leq n} | \psi^{(j+1)}_{k,p,*1}|  \one_{\A_1 \cap \mathscr{C}_1} \Big\VERT_2 
    \leq \varkappa \Sigma' + \sum_{k=p}^n a_kc_k'
  \end{equation}
  where  $\Sigma = \Sigma(c_k)  = \max \biggl\{ 2R\mathrm{S}  \max_{p \leq k \leq n} \left( \sqrt{a_{k}}c_k \right), \sqrt{{\textstyle \sum_{k=p}^{n} } a_{k} c_k^2 } \biggr\}$ and $\Sigma '= \Sigma(c_k')$.
  
Moreover using hyperbolicity, we have  for odd $j,$ with 
$ \mathscr{C}_2 =  \left\{ |\psi^{(j)}_{n,k+1,21}| \leq c_k'' : p < k < n \right\}$ and $C>0$ as in \eqref{hyperbolic:ass}, the deterministic bound 
    \begin{equation} \label{control2}
    |\psi^{(j+1)}_{n,p,22}|\one_{\mathscr{C}_2 \cap \mathscr{B}_{n,p}} \leq 
  C \mathrm{S} \sum_{k=p+1}^n c_k'' \sqrt{a_k} e^{-(k-p-1)b_p}.
  \end{equation}
\end{proposition}
\begin{proof}
By \eqref{eq:psijforward}, the process $n \mapsto \psi^{(j+1)}_{n,p,1*}$ is adapted to the filtration $\{ \filt_{n,p} \}$ and it has a simple martingale decomposition  $ \psi^{(j+1)}_{n,p,1*} = M_n + A_n$  with
  \[
    \begin{aligned}
    &M_{n} =\sum_{k=p}^n %U_n \dots U_{k+1} 
    %\psi^{(0)}_{n,k+1,11} 
    (\eta_{k,12}-\Exp \eta_{k,12})\psi^{(j)}_{k-1,p,2*}  \, , \qquad
    A_{n} =\sum_{k=p}^n %U_n \dots U_{k+1} 
    %\psi^{(0)}_{n,k+1,11} 
    \Exp(\eta_{k,12}) \psi^{(j)}_{k-1,p,2*}.
    \end{aligned}
  \]
  Let $T = \inf\left\{ k\ge 1 : |\psi^{(j)}_{k-1,p,2*}| > c_k \right\}$.  Then the stopped process $A_{n \wedge T}$ satisfies $|A_{n\wedge T}| \leq \sum_{k=p}^n a_k c_{k}$.
 Moreover, by Theorem~\ref{thm:M} with $\alpha = 2 \mathrm{S} \max_{p \leq k \leq n} (\sqrt{a_k}c_k)$  and $V_n \le {\textstyle \sum_{k=p}^{n} } a_{k} c_k^2$ applied to the stopped martingale $M_{n \wedge T},$  we conclude that with 
  \[
    \Sigma =  \max \biggl\{ 2R\mathrm{S}  \max_{p \leq k \leq n} \left( \sqrt{a_{k}}c_k \right), \sqrt{{\textstyle \sum_{k=p}^{n} } a_{k} c_k^2 } \biggr\} , 
  \]
  there is an event $\A$ measurable with respect to $\filt_{n,p}$ and an absolute constant $\varkappa$ such that
  \[
    \Big\VERT  \sup_{p \leq k \leq n}| M_{k \wedge T} | \one_{\A} \Big\VERT_2 \leq \varkappa \Sigma
    \quad
    \text{and}
    \quad
    \Pr(\A^c) \le 4 e^{-R}.
  \]
Since $ \psi^{(j+1)}_{n,p,1*} = M_n + A_n$ and using the almost sure bound for $|A_{n\wedge T}|$, this implies that on the event  $\mathscr{C}_1 = \left\{ |\psi^{(j)}_{k-1,p,2*}| \vee  |\psi^{(j)}_{n,k+1,*2}|  \leq c_k : p < k < n \right\},$
  \[
     \Big\VERT  \sup_{p \leq k \leq n}| \psi^{(j+1)}_{k,p,1*}| \one_{\A \cap \mathscr{C}_1}  \Big\VERT_2 
    \leq \varkappa \Sigma+ \sum_{k=p}^n a_kc_k  . 
  \]
  Using the representation \eqref{eq:psijbackward}, we can perform nearly the same argument, now using that 
    \( p \mapsto \psi^{(j+1)}_{n,p,*1} \)
    is adapted to the reversed filtration $\{ \filt_{n,p} \}$.   This gives the estimate \eqref{control1}. 
    
   \medskip
    
   As for the deterministic bound, using \eqref{eq:simple22}, we obtain  on the event $  \mathscr{B}_{n,p} \cap   \mathscr{C}_2$, 
   \[
     |\psi_{n,p,22}^{(j+1)}|
     \leq 
     \sum_{k=p+1}^n %U_n \dots U_{k+1} 
     |\psi^{(j)}_{n,k+1,21}\eta_{k,12}\psi^{(0)}_{k-1,p,22}|
     \le C \mathrm{S}
     \sum_{k=p+1}^n 
     c_k'' \sqrt{a_k} e^{-(k-p-1)b_p}.
   \]
\end{proof}

\paragraph{Application.}
We can use  iteratively the moderate deviation estimates from Proposition~\ref{prop:induction} to control the norm of the matrices $\psi^{(j)}$ for all $j\ge 1$ in the perturbative expansion \eqref{eq:psiTaylor}.
On a similar basis as a Taylor expansion, if the noise if sufficiently small, we show in the proof of the next proposition that on an event of overwhelming probability, the corrections $\psi^{(j)}$ become smaller as $j$ increases. 
The proof starts by using hyperbolicity (i.e.\ the events \eqref{hyperbolic:ass} hold almost surely) to get a priori control of the norm of $\psi^{(1)}.$ Then, it proceeds by induction from $2j$ to $2(j+1)$. This 2--step induction occurs because the estimates  \eqref{control1}--\eqref{control2} work together. 
It turns out that the gain from hyperbolicity comes solely from the estimate \eqref{control2}, so that this strategy is suboptimal.  Nonetheless, it turns out to be sufficient in the context of Theorem~\ref{thm:planar}.

\begin{proposition}\label{prop:planar}
Let us assume that the matrices $\{U_k\}$ satisfies the assumptions from Proposition~\ref{prop:rec} with $\alpha \in (0, \frac 19]$ and $\epsilon>0$.
For any $\delta>0$ such that $\delta+\epsilon \le \alpha/2$,  there is an event $\A_N$ and a $c>0$ such that for all $N$ sufficiently large
  \[
    \| \psi^{(>0)}_{N,1} \|  \one_{\A_N} \leq 2  N^{\frac{15\alpha}{2} -1+ \epsilon + \frac{3\delta}{2}}
    \quad
    \text{and}
    \quad
    \Pr[\A_N] \geq 1- e^{-cN^\delta}. 
  \]
  
\end{proposition}

\begin{proof}
  Recalling the assumptions in Proposition \ref{prop:rec}, the noise satisfies Assumption \ref{diffusive:ass}  with $a_k = N^{2\alpha -1}$ and $\mathrm{S} = N^{\epsilon}$ for any $k\in\{1,\dots, N\}$.  Likewise, by assumption there is a $c > 1$ so that
  \[
    |\rho_k - \eta_{k,22}| 
    \le 1- cN^{-\alpha} + N^{\alpha-1/2+\epsilon}
    \le 1- N^{-\alpha},
  \]
  for all $N$ sufficiently large,
  since $\alpha+\epsilon <1/4.$
  Thus in the context of \eqref{hyperbolic:ass}, since 
$\psi^{(0)}_{\ell,k,22} =   \textstyle{\prod_{k=p}^n}  ( \rho_k - \eta_{k,22})$ and  $|\rho_k - \eta_{k,22}| \le 1- N^{-\alpha} $, if we choose $C=1$ and 
$b_k = - \log(1- N^{-\alpha}) >0 $ for $k\in\{1,\dots, N\}$, 
it holds that $\Pr( \mathscr{B}_{N,1}) = 1.$

\medskip

\noindent\underline{Step 1:}\ We apply Proposition \ref{prop:induction} to estimate $\|\psi^{(1)}_{n,p}\|$, with $j=0$, $R=N^{\delta}$, $c_k = (1- N^{-\alpha})^{k-1-p}\1_{k>p}$ and $c_k' = (1- N^{-\alpha})^{n-k-1}\1_{n>k}$.
With these choices, $\P[\mathscr{C}_1]=1$ and
\[
  {\textstyle  \sum_{k=p}^n }  a_k c_k^2 \le   {\textstyle  \sum_{k=p}^n }  a_k c_k \leq N^{3\alpha - 1}.
\]
This implies, using $\delta+\epsilon \le \alpha/2$,
\[
\Sigma \vee \Sigma' \le\max \biggl\{ 2 N^{\delta+\epsilon + \alpha -1/2},  N^{3\alpha/2 - 1/2}  \biggr\} 
\le 2  N^{\frac{3\alpha - 1}{2}} 
\]
for all $N.$
Hence Proposition \ref{prop:induction} part \eqref{control1} implies there is an absolute constant $\varkappa > 0$ so 
that for all $N \ge p \ge 1$ there exists an event $\A_1(N,p)$ such that 
\begin{equation} \label{control3}
 \Big\VERT  \sup_{p \leq n \leq N} \| \psi^{(1)}_{n,p}\| \one_{\A_1(N,p)} \Big\VERT_2 
  \leq \varkappa N^{\frac{3\alpha - 1}{2}} , \qquad 
  \Pr[ \A_1^c(N,p) ] \leq 4e^{-N^\delta} . 
\end{equation}
From the estimate \eqref{control3}, we deduce that there is a constant $\varkappa' \in (0,1),$ depending only on $\varkappa,$ such that 
  \begin{equation*}
    \Pr[  \mathscr{G}_1^c ]\leq  8Ne^{-\varkappa' N^\delta}  \qquad \text{where}\qquad  \mathscr{G}_1 = \bigg\{ \sup_{1 \leq p \leq n \leq N} \|\psi^{(1)}_{n,p}\| \le N^{\frac{3\alpha - 1 + \delta}{2}} \bigg\}
 \end{equation*} 
 by doing a union bound over all $1 \leq p \leq N$.

\noindent\underline{Step 2:}\ 
This also allows us to get control over $\|\psi^{(2)}_{n,p}\|$. 
If we apply Proposition \ref{prop:induction} part \eqref{control2} with $j=1$ and $c_k'' =  N^{\frac{3\alpha - 1 +\delta}{2}}$ for all  $k\in\{1,\dots, N\}$, we obtain that on the event $ \mathscr{G}_1$, 
 \[
     |\psi^{(2)}_{n,p,22}| \leq N^{\frac{5\alpha+ \delta}{2}+\epsilon-1} \sum_{k=p+1}^n  (1- N^{-\alpha})^{k-p-1}   \leq  N^{\frac{7\alpha+ \delta}{2}+\epsilon-1}  .
 \]

\noindent\underline{Induction:}\ 
  We  proceed by induction on $j\in\N$, again applying Proposition \ref{prop:induction}. 
     Fix $j \ge 1$, $\gamma_j \in \R$ and suppose on an event $ \mathscr{G}_j$ it holds   
     \begin{equation} \label{control4}
    \sup_{1 \leq p \leq n \leq N} |\psi^{(2j)}_{n,p,22}| \leq N^{\gamma_j-1} ,\qquad \Pr[  \mathscr{G}^c_j ]\leq  8jN e^{-\varkappa' N^\delta} . 
     \end{equation}
  If we apply the estimate \eqref{control1} with $c_k = c_k' = N^{\gamma_j - 1}$, we have as  $\delta+\epsilon < 1/2$,
   \[
 \Sigma =   \max \biggl\{ 2 N^{\delta+\epsilon + \alpha+\gamma_j - 3/2}, N^{\alpha+ \gamma_j-1}\biggr\}
 \le 2N^{\alpha+ \gamma_j-1}
 \quad\text{and}\quad
   \sum_{k=p}^n a_kc_k  \le N^{2\alpha+ \gamma_j-1}.
 \]
 Thus we again conclude as in \underline{Step 1} that there is an event $\A_{2j+1}(N,p)$ with $\Pr(\A_{2j+1}^c(N,p)) \leq 4e^{-N^{\delta}}$ for which
 \[
   \Big\VERT  \sup_{p \leq n \leq N} \| \psi^{(2j+1)}_{n,p}\| \one_{\A_{2j+1}(N,p)} \Big\VERT_2 
  \leq \varkappa N^{{2\alpha + \gamma_j - 1}}.
 \]
 Thus also as in \underline{Step 1} we obtain from a union bound that with the same constant $\varkappa' \in (0,1)$,   
  \begin{equation*}
    \Pr[  \mathscr{G}_{j+1}^c \cap  \mathscr{G}_j ]\leq  8Ne^{-\varkappa' N^\delta}  \qquad \text{where}\qquad  \mathscr{G}_{j+1} = \bigg\{ \sup_{1 \leq p \leq n \leq N} \|\psi^{(2j+1)}_{n,p}\| \le N^{2\alpha+ \gamma_j-1 + \frac{\delta}{2}} \bigg\}.
 \end{equation*}
Moreover, exactly as in \underline{Step 2} with $c_k'' = N^{2\alpha+ \gamma_j-1 + \frac{\delta}{2}} $, it holds on the event $ \mathscr{G}_{j+1}$, 
\[
     \sup_{1 \leq p \leq n \leq N}    |\psi^{(2j+2)}_{n,p,22}|  \leq  N^{\epsilon+ 4\alpha+ \gamma_j- 3/2 + \frac{\delta}{2}} . 
\]

This shows that as the parameters satisfy $\epsilon+ 4\alpha+  \delta \le 1/2$ (because $\alpha \le 1/9$ and  $\delta+\epsilon \le \alpha/2$), then the condition \eqref{control4} is fulfilled at step $j+1$ with $\gamma_{j+1} =  \gamma_j - \frac\delta2 $ and we have   
\[ 
  \Pr[  \mathscr{G}_{j+1}^c] \le \Pr[\mathscr{G}_j^c ] + 8N e^{-\varkappa' N^\delta} \le 8N(j+1)e^{-\varkappa' N^\delta} .  
\]

\noindent\underline{Conclusion:}\
 Let us define $\mathscr{J}_N = \bigcap_{j=1}^{N/2} \mathscr{G}_j $  and observe  that is follows from the previous estimate that 
     \begin{equation} \label{control5}
\Pr[  \mathscr{J}_N^c] \le {\textstyle \sum_{j=1}^{N/2}  } \Pr[  \mathscr{G}_{j}^c] \lesssim  N^3 e^{-\varkappa' N^\delta}  .
\end{equation}
Hence, if we apply the induction step starting with $\gamma_1 = \frac{7\alpha+ \delta}{2}+\epsilon$ (see  \underline{Step 2}),  we obtain that 
\[
\mathscr{J}_N \subseteq \left\{
\sup_{1 \leq p \leq n \leq N} \big( \|\psi^{(2j-1)}_{n,p}\| \vee  |\psi^{(2j)}_{n,p,22}|  \big) \le N^{ 2\alpha+ \gamma_j-1 + \frac{\delta}{2}} :  j \in \big\{1,\dots, \lceil N/2 \rceil\big\}   \right\} . 
\]

To complete the proof, it remains to obtain a similar (uniform) large deviation bound for  $|\psi^{(2j)}_{n,p,11}|$.  
 If we apply Proposition \ref{prop:induction} \eqref{control1} with $c_k =  N^{ 2\alpha+ \gamma_j-1 + \frac{\delta}{2}} $ for all  $k\in\{1,\dots, N\}$, we obtain for  all $N\ge n \ge p \ge 1$ and $j\in\{1,\dots, \lceil N/2 \rceil\}$ that there exists  an event $\A_1^j(n,p)$ such that $\Pr[ \A_1^j(n,p)^c ] \leq 4e^{-N^\delta}$ and 
 \[
     \Big\VERT  \sup_{p \leq k \leq n}  | \psi^{(2j)}_{k,p,11}|\one_{\A_1^j(n,p) \cap \mathscr{J}_N} \Big\VERT_2 
    \lesssim  N^{ 4\alpha+ \gamma_j-1 + \frac{\delta}{2}}   .
 \]
We have used that  as  $\delta+\epsilon < 1/2$,
 \[
 \Sigma  \le 2 N^{ 3\alpha+ \gamma_j-1 + \frac{\delta}{2}} 
 \quad\text{and}\quad
   \sum_{k=p}^n a_kc_k  \le N^{ 4\alpha+ \gamma_j-1 + \frac{\delta}{2}}  . 
 \]
 
Consequently, by taking a union bound over all the events $ \big\{\A_1^j(n,p)\big\}$, this implies that  if $N$ is sufficiently large, 
\[
    \Pr[  \mathscr{V}_N^c  \cap \mathscr{J}_N ]\leq  e^{-cN^\delta}  \qquad \text{where}\qquad \mathscr{V}_N= \bigg\{ \sup_{1 \leq p \leq n \leq N}  | \psi^{(2j)}_{n,p,11}| \le   N^{ 4\alpha+ \gamma_j-1 + \delta} :   j \in \big\{1,\dots, \lceil N/2 \rceil\big\}   \bigg\} .
\]

Thus, if we let $\A_N =  \mathscr{V}_N  \cap \mathscr{J}_N $,  since $ \psi^{(>0)}_{N,1} = \sum_{j=1}^{N}  \psi^{(j)}_{N,1} $ and  $\gamma_j = \frac{7\alpha}{2}+ \delta+\epsilon - \frac{j\delta}{2}$,  the previous estimates show that on the event  $\A_N$, 
\[
\| \psi^{(>0)}_{N,1} \| = \sum_{j=1}^{N} \| \psi^{(j)}_{N,1}  \| \le  \sum_{j=1}^{N/2} N^{ 4\alpha+ \gamma_j-1 + \delta}\lesssim  N^{\frac{15\alpha}{2} -1+ \epsilon + \frac{3\delta}{2}}
\]
and that $\Pr[  \mathscr{A}_N^c ]\leq  e^{-cN^\delta}$ provided that $N$ is sufficiently large (by \eqref{control5}).
\end{proof}

 We are now ready to complete the proof of Proposition~\ref{prop:rec}.

\begin{proof}[Proof of Proposition~\ref{prop:rec}]
By \eqref{eq:psiTaylor}, we have 
$  \prod_{k=1}^N U_k = \psi^{(0)}_{N,1} + \psi^{(>0)}_{N,1}$ 
where $ \psi^{(0)}_{N,1} =  \left(\begin{matrix} 1 & 0 \\ 0 & \prod_{k=1}^N U_{k,22} \end{matrix}\right)$. 
Moreover, under the assumptions of Proposition~\ref{prop:rec}, we have for $N$  is sufficiently large,
$| U_{k,22}| = |\rho_k - \eta_{k,22}| \le 1- N^{-\alpha} $ for all  $k\in\{1,\dots, N\}$ almost surely   (because $c>1$ and $\alpha+\epsilon <1/4$). This  implies that 
\[
{\textstyle  \prod_{k=1}^N} |U_{k,22}| \le \big( 1- N^{-\alpha}  \big)^N \le e^{-N^{1-\alpha}} .
\]
We conclude that on the event $\A_N$ from Proposition~\ref{prop:planar},  if $N$ sufficiently large,  then 
  \[
  \left\| {\textstyle \prod_{k=1}^N U_k } - \left(\begin{smallmatrix} 1 & 0 \\ 0 & 0 \end{smallmatrix}\right)  \right\| \le 
  3  N^{\frac{15\alpha}{2} -1+ \epsilon + \frac{3\delta}{2}}   \le 3  N^{\frac16-\epsilon -\frac{3\delta}{2} }  
  \]
  where we used that $\alpha \le 1/9$. 
  This completes the proof.
\end{proof}

\section{Proof of Theorem \ref{thm:main}}  %Linearized G$\beta$E recurrence estimate
\label{sec:linearized}

\subsection{Overview of the proof}
\label{sect:strategy} 

Let us recall that we have a sequence of independent random matrices of the form
\begin{equation*} %\label{recall:U}
U_k = \begin{pmatrix} 
1 &   \eta_{k,12} \\   \eta_{k,21} & \rho_k - \eta_{k,22} 
\end{pmatrix} , 
\end{equation*}
whose entries satisfy for all $1 \leq k \leq N$,  with $\hat{k}=N-k$, 
\begin{equation} \label{recall:eta}
|\eta_{k,ij} | \le  \sqrt{\frac{ R\Omega\log N}{\omega_N + \hat{k}}}, 
\quad
\big| \E \eta_{k,ij} \big|  \le \frac{c}{\omega_N +\hat{k}},
\quad
\text{and}
\quad
\E \big| \eta_{k,ij}\big|^2 \le   \frac{c}{\omega_N +\hat{k}}  , 
\end{equation}
where $\omega_N =  N^{1/3}(\Omega\log N)^{2/3}$ and  $c,R>0$ are constants.  
Moreover, the complex scalar $\rho_k$ satisfies with $c_0>1$ for all $1 \leq k \leq N$, 
\begin{equation} \label{recall:rho_k}
  |\rho_k| \le e^{-c_0 \sqrt{\frac{\omega_N + \hat k}{N}}}  . 
\end{equation}
All the statements in Section~\ref{sec:linearized} holds for $N$ sufficiently large, $\Omega = o\big( \frac{N^{1/15}}{\log N}\big)$ and $R \Omega\ge r_\beta$ for a fixed sufficiently large $r_\beta$. 
In particular, we assume that $\omega_N \in[N^{1/3},N]$ and $\max_{i,j,1\le k\le N} |\eta_{k,ij} | \le 1/10\ \As$
We use the formalism from Section~\ref{sect:gf} and our goal is to obtain a tail estimate for $\prod_{k=1}^N U_k - \prod_{k=1}^N V_k$ where $V_k  =  \diag(U_k)$. 
Recall that we  denote $\filt_{n,p}$ for the $\sigma$--algebra $\sigma( U_k : p \leq k \leq n)$  with $N\ge n\ge p\ge 0$. 

\medskip 

Let us now go swiftly over the strategy of the proof of Theorem~\ref{thm:main}.
Our ultimate goal is to show that  $\prod_{k=1}^NU_{k}  \simeq \psi_{N,1}^{(0)}$ and that $\psi^{(0)}_{N,1, 22}$ is small with overwhelming probability.
We rely on the  perturbative expansion \eqref{eq:psiTaylor} and the equations \eqref{eq:psi1step} to control the errors.

\smallskip
\noindent\underline{Section~\ref{sect:MD0}.}\
The first step consists in exploiting the hyperbolic character of the transfer matrices $U_k$ to obtain moderate deviation estimates for 
$\psi^{(0)}_{n,p, 22}$ for all $N\ge n>p\ge 1$.
The argument relies on the decay of $|\rho_k|$ away from the turning point (which corresponds to $\hat{k}=0$ by convention) and an application of Bernstein's inequality. 

\smallskip
\noindent\underline{Section~\ref{sect:MD1}.}\ Then, we can deduce moderate deviation estimates for the matrices $\psi^{(1)}_{n,p}$ for all $N\ge n>p\ge 1$.
The proof relies on Proposition~\ref{prop:induction} and it is analogous to \underline{Step 1} in the proof of Proposition~\ref{prop:planar}. 
Note that our estimates for $\psi^{(1)}_{n,p,12}$ and  $\psi^{(1)}_{n,p,21}$ are not symmetric. 

\smallskip
\noindent\underline{Section~\ref{sec:exptransrec}.}
This part is independent from the rest of the proof and it deals with deterministic estimates. 
We give a general lemma about the stability of $2\times2$ hyperbolic matrices under  off--diagonal perturbations that we use to control $\| \Exp[ \prod_{k=p}^n U_k ]\|$ for   all $N\ge n>p\ge 1$.

\smallskip
\noindent\underline{Section~\ref{sec:shortblock}.}
Using  Proposition \ref{prop:ta} together with the estimates from Sections~\ref{sect:MD1} and~\ref{sec:exptransrec}, we obtain  moderate deviation estimates for the matrices $\psi^{(>1)}_{n,p}$  on short blocks, meaning when $(n-p) \leq \sqrt{N/\omega_N}.$ 
These estimates do not rely on \emph{hyperbolicity} and therefore cannot hold as such on longer blocks. 
After the blocking, the coarse-grained transfer matrices have better hyperbolicity properties.
%for all $N\ge n>p\ge 1$

\smallskip
\noindent\underline{Section~\ref{sec:longblock}.}
We now use the improved \emph{hyperbolicity} from Section~\ref{sec:shortblock} to make estimates on longer blocks. In this Section we focus on giving a uniform  bound for the norm $\| U_n \cdots U_p\|$ which holds with overwhelming probability for all $N\ge n>p\ge 1$. 

\smallskip
\noindent\underline{Section~\ref{sec:final}.}
Having controlled $\| U_n \cdots U_p\|$ and $\psi^{(1)}_{n,p}$ over all $N \geq n > p \geq 1$ we use
\eqref{eq:psi1step} with $\ell=j=1$ to control $\psi^{(> 1)}_{n,p}$
using the estimates from Sections~\ref{sect:MD0},~\ref{sect:MD1} and~\ref{sec:longblock}.
This allows us to conclude that  $\prod_{k=1}^NU_{k}  \simeq \psi_{N,1}^{(0)}$ with overwhelming probability.  The estimates on $\psi^{(\geq 1)}_{n,p}$ are just as good as that for $\psi^{(1)}_{n,p}$ we obtained in  Section~\ref{sect:MD1}.

On a technical note, we treat 
$\psi_{n,p}^{(>1)} \left( \begin{smallmatrix} 0 & 0 \\ 0 &1\end{smallmatrix} \right)$ 
and
$\psi_{n,p}^{(>1)} \left( \begin{smallmatrix} 1 & 0 \\ 0 &0\end{smallmatrix} \right) $
separately as these processes behave very differently (see the equations~\eqref{psidecomp}). 
The first term can be handled directly and it is small because of \emph{hyperbolicity}. As for the second term, it has a \emph{regenerative structure} and we perform a martingale decomposition (see \eqref{eq:Zinc}) to control it.
This allows us to take advantage of the independence of $U_k$ and use Freedman--Tropp's inequality (Theorem~\ref{thm:M}) to obtain  moderate deviation estimates for $\psi_{n,p}^{(>1)} \left( \begin{smallmatrix} 1 & 0 \\ 0 &0\end{smallmatrix} \right).$ 

\smallskip
\noindent\underline{Section~\ref{sect:extraproofs}.}
This is the last step of the proof. We obtain estimate for the martingale part in the decomposition of 
$\psi_{n,p}^{(>1)} \left( \begin{smallmatrix} 1 & 0 \\ 0 &0\end{smallmatrix} \right)$ and its quadratic variation. 
These estimates are required in Section~\ref{sec:final} to apply Freedman--Tropp's inequality.

\subsection{Moderate deviation bounds for $\psi^{(0)}$ }
\label{sect:MD0}

\begin{proposition} \label{prop:1}
There exists $C>0$ such that for any $R>0$ and all $1 \le p< n \le N$, there is an event $\A^1_{n,p}$ measurable\footnote{The events $\{ \A^1_{n,p} \}_{1 \le p< n \le N}$ depend on the constant $R$ and the parameters $\Omega, N$ even though it is not emphasized. The same holds for the events $\{ \A^\ell_{n,p} \}_{1 \le p< n \le N}$ defined below and we also have  $\A^\ell_{n,p} \subset \A^{\ell-1}_{n,p}$ for any  $\ell\in\{2,3,4\}$ and all ${1 \le p< n \le N}$.} with respect to $\F_{n,p}$ with $\P[\A^{1,c}_{n,p}] \le N^{-R\Omega}$ so that 
\[ 
| \psi^{(0)}_{n,p, 22}| \1_{\A^1_{n,p}} \le C_R  \exp\left(- \frac{\hat{p}-\hat{n}}{4}\sqrt{\frac{\omega_N+\hat{p}}{N}} \right) \, , \qquad C_R = e^{CR} . 
\]
Moreover, if $(n-p) \ge 12 R^{-1}\sqrt{N/\omega_N}$, 
\[
  \P\big[| \psi^{(0)}_{n,p, 22}| \ge e^{-1/R}  \big] \le \exp\left( - \frac{{\omega_N}\sqrt{\omega_N+ \hat{p}}}{C R\sqrt{N}} \right) . 
\]
\end{proposition}

\begin{proof}
  From Lemma \ref{lem:psistructure}, 
  \[
    \log | \psi^{(0)}_{n,p, 22}|  = \sum_{k=p}^n  \log| \rho_k - \eta_{k,22}|,
  \]
  which is a sum of independent random variables with 
  \[
    \log|\rho_k| \leq - c_0 \sqrt{ \tfrac{\omega_N + \hat k}{N}}.
  \]
  Let $S$ be the set of $k\in [p,n]$ such that
  \(
    \log|\rho_k| \geq -2.
  \)
  For any $k \in S^c,$ by \eqref{recall:eta}, it holds  
  \[
    \log|\rho_k - \eta_{k,22}|
    \leq \log \left(e^{-2} + 1/10 \right)
    \leq -\sqrt{2} ,
  \]
so that $ \displaystyle \sum_{k \in S^c}\log| \rho_k - \eta_{k,22}| \leq  -\sum_{k \in S^c} \sqrt{ \tfrac{\omega_N + \hat k}{N}} \quad \As$

\medskip

On the other hand, if we define $\xi_k$  for $k \in S$ by
\begin{equation}\label{eq:2nddumpkin}
  \log|\rho_k - \eta_{k,22}|
  =\log|\rho_k| - \Re\left[ \frac{\eta_{k,22}}{\rho_k} + \frac{\eta_{k,22}^2}{2\rho_k^2}\right] + \xi_k,
\end{equation}
then as $\big|\log(1-z) + z + z^2/2\big| \leq |z|^3/(1-|z|),$ we have that $|\xi_k| \lesssim \left( \frac{ \Omega \log N}{\omega_N +\hat k}\right)^{3/2} \As$
  Using \eqref{recall:eta}--\eqref{recall:rho_k}, since $c_0>1$, it follows that for $k \in [1,N] $ and for all $N$ sufficiently large (depending on $\Omega$),
  \[
    \Exp\log| \rho_k - \eta_{k,22}|
    \leq - \sqrt{ \frac{\omega_N + \hat k}{N}}.
  \]
  Here we have used that $\frac{\omega_N^{3/2}}{\sqrt{N}}  = \Omega \log N$ to control the error --  a fact that we use several times in Section~\ref{sec:linearized}.  
  Observe that
  \begin{equation*}%\label{eq:sqrtbnd}
    \sum_{\hat{k}=\hat{n}}^{\hat{p}} \sqrt{\frac{\omega_N+\hat{k}}{N}}  
    \ge \sum_{\hat{k}=\frac{\hat{n}+\hat{p}}{2}}^{\hat{p}}  \sqrt{\frac{\omega_N+\hat{k}}{N}}  
    \ge \frac{\hat{p}-\hat{n}}{2}  \sqrt{\frac{\omega_N+\frac{\hat{n}+\hat{p}}{2}}{N}} 
    \ge  \frac{\hat{p}-\hat{n}}{2\sqrt{2}}  \sqrt{\frac{\omega_N+\hat{p}}{N}} ,
  \end{equation*}
  so we conclude that
  \begin{equation}\label{eq:logrhomean}
    \sum_{k \in S} \Exp\log| \rho_k - \eta_{k,22}| + \sum_{k \in S^c} \log| \rho_k - \eta_{k,22}| 
    \leq  - \sum_{\hat{k}=\hat{n}}^{\hat{p}} \sqrt{\frac{\omega_N+\hat{k}}{N}} 
    \le   - \frac{\hat{p}-\hat{n}}{2\sqrt{2}}  \sqrt{\frac{\omega_N+\hat{p}}{N}} \quad \As
  \end{equation}

  We must also estimate the fluctuations of the terms in $S$.  
  From \eqref{eq:2nddumpkin} and\eqref{recall:eta} again, it follows easily that for $k \in S,$
  \[
    \Var\left( \log| \rho_k - \eta_{k,22}| \right)
    \lesssim \frac{1}{\omega_N + \hat k}
  \]
  and
  \[
 \max_{k\le n} \big|  \log|\rho_k - \eta_{k,22}| - \E  \log|\rho_k - \eta_{k,22}|  \big| \lesssim  \sqrt{\frac{R\Omega\log N}{\omega_N+ \hat{n}}} . 
  \]
Then,  
\[
  \sum_{k \in S} \Var\left[ \log| \rho_k - \eta_{k,22}| \right] \lesssim \log\left(\frac{\omega_N+\hat{p}}{\omega_N + \hat{n}}\right) \le \frac{\hat{p}-\hat{n}}{\omega_N+\hat{n}}.
\]
Hence, by setting $X_S = \sum_{k \in S} \log| \rho_k - \eta_{k,22}| - \Exp \log| \rho_k - \eta_{k,22}|$
and using  Freedman's inequality \eqref{eq:FT} with
$\Sigma^2 = \frac{C}{16}\frac{\hat{p}-\hat{n}}{\omega_N}$ and $\alpha = C  \sqrt{\frac{ R\Omega  \log N}{\omega_N}}$ for a  sufficiently large constant $C>0,$
\begin{equation} \nonumber % \label{eq:firstfreeman}
\P[X_S \ge t] 
\leq
\exp\left( 
-\frac{t^2/2}{\alpha t/3+ \Sigma^2}  
\right).
\end{equation}
For any $R \geq 0$ and any $t \ge \frac{\hat{p}-\hat{n}}{16}  \sqrt{\frac{\omega_N+\hat{p}}{N}} +  R,$ we have for  $N$ is sufficiently large (depending on $\Omega$),
\[
\frac{t^2}{\alpha t/3+ \Sigma^2}  
\ge \min\left\{ \frac{3t}{\alpha} , \frac{t^2}{\Sigma^2} \right\} 
\ge \frac{2R}{C} \min\left\{   \sqrt{\frac{\omega_N}{ R\Omega  \log N}} ,   \omega_N \sqrt{\frac{\omega_N+\hat{p}}{N}} \right\}
\ge \frac{2R}{C}  \frac{\omega_N^{3/2}}{\sqrt{N}},
\]
 so we conclude that
\begin{equation} \label{eq:firstfreeman}
\P[X_S \ge t] 
\le \exp\left(-\frac{R}{C}  \frac{\omega_N^{3/2}}{\sqrt{N}} \right)  = 
N^{- \frac{R \Omega}{C} } . 
\end{equation}

As $ \log | \psi^{(0)}_{n,p, 22}|  = X_S  +  \sum_{k \in S} \Exp \log| \rho_k - \eta_{k,22}|  +   \sum_{k \in S^c}\log| \rho_k - \eta_{k,22}|$, using the deterministic estimate \eqref{eq:logrhomean}, we obtain that for any  $t \le \frac{\hat{p}-\hat{n}}{4}  \sqrt{\frac{\omega_N+\hat{p}}{N}} - R $, 
\begin{equation}\label{est:psi0}
 \begin{aligned}
\P\big[| \psi^{(0)}_{n,p, 22}|  \ge e^{-t} \big]
& \le   \P\big[  X_S  \ge -t - {\textstyle  \sum_{k \in S} \Exp \log| \rho_k - \eta_{k,22}|  -   \sum_{k \in S^c}\log| \rho_k - \eta_{k,22}|} \big] \\
& \le \P\big[  X_S  \ge \tfrac{\hat{p}-\hat{n}}{16}  \sqrt{\tfrac{\omega_N+\hat{p}}{N}}  +R\big]
\le N^{- \frac{R \Omega}{C} } . 
\end{aligned}
\end{equation}

Hence, if we set $\A^1_{n,p} = \left\{ | \psi^{(0)}_{n,p, 22}|  \le C_R e^{- \frac{\hat{p}-\hat{n}}{4}  \sqrt{\frac{\omega_N+\hat{p}}{N}} }\right\}$ with $C_R=e^{CR}$,  after adjusting $R$, we obtain $\P\big[ \A^{1,c}_{n,p}\big] \le N^{-R\Omega} $
and $| \psi^{(0)}_{n,p, 22}| \1_{\A^1_{n,p}} \le  C_R  \exp\left(- \frac{\hat{p}-\hat{n}}{4} \sqrt{\frac{\omega_N+\hat{p}}{N}} \right) $. 

\medskip

For the second claim, 
we use \eqref{eq:firstfreeman}, with $t=2 R = \frac{\hat{p}-\hat{n}}{12}  \sqrt{\frac{\omega_N+\hat{p}}{N}}$ to conclude in the same way as \eqref{est:psi0}  that
\[
  \P\big[| \psi^{(0)}_{n,p, 22}|  \ge e^{-\frac{\hat{p}-\hat{n}}{12}  \sqrt{\frac{\omega_N+\hat{p}}{N}}} \big]
  \leq 
  \P\big[  X_S  \ge \tfrac{\hat{p}-\hat{n}}{4}  \sqrt{\tfrac{\omega_N+\hat{p}}{N}} \big]
  \leq 
  \exp\left(-\frac{\hat{p}-\hat{n}}{12C}\sqrt{\frac{\omega_N+\hat{p}}{N}}\frac{\omega_N^{3/2}}{\sqrt{N}} \right).
\]
If $n-p \geq 12 M^{-1}\sqrt{N/\omega_N}$ for a constant $M>0,$ then
$\frac{\hat{p}-\hat{n}}{12}  \sqrt{\frac{\omega_N+\hat{p}}{N}} \ge M^{-1} $ and we conclude that
\[
  \P\big[| \psi^{(0)}_{n,p, 22}|  \ge e^{-M^{-1}} \big]
  \leq
  \exp\left(-\frac{\omega_N}{CM}\sqrt{\frac{\omega_N+\hat{p}}{N}}
  \right),
\]
which gives the second claim of the Proposition.
\end{proof}

\subsection{Moderate deviation bounds for $\psi^{(1)}$}  
\label{sect:MD1}

\begin{proposition} \label{prop:2}
For any $R>0$ and for all $1 \le p< n \le N$, there is an event $\A^2_{n,p} \subset \bigcap_{p \leq j < k \leq n} \A^1_{k,j}$ measurable wrt $\F_{n,p}$ with $\Pr[ \A^{2,c}_{n,p} ] \lesssim N^{2-R\Omega}$ so that 
\[
  \begin{aligned}
  &\big\VERT \max_{p \leq  k \leq n}|\psi^{(1)}_{k,p,12}| \1_{\A^2_{n,p}} \big\VERT_2  \le  \frac{C_R N^{1/4}}{(\omega_N + \hat{p})^{5/8} \omega_N^{1/8}}
  \quad
  \text{and}
  \quad \\
  &\big\VERT \max_{p \leq j < k \leq n}|\psi^{(1)}_{k,j,21}| \1_{\A^2_{n,p}} \big\VERT_2  \le \frac{ C_R N^{1/4}}{(\omega_N + \hat{n})^{5/8} \omega_N^{1/8}} .
  \end{aligned}
\]
\end{proposition}

\begin{proof}
  We apply Proposition \ref{prop:induction} with $j=1$ and $R=\mathrm{S}^2= R\Omega\log N$.
   Let $C_R = e^{CR}$ as in Proposition \ref{prop:1}. We begin with the estimate for $\psi^{(1)}_{k,p,12}.$  By \eqref{recall:eta}, we take $a_k = \frac{1}{\omega_N + \hat k}$, and 
  \[
    c_k \coloneqq C_R \exp\left(- \frac{\hat{p}-\hat{k}}{4} \sqrt{\frac{\omega_N+\hat{p}}{N}} \right).
  \]
   First, we verify that $\max_{p \leq k \leq n}( \sqrt{a_k} c_k ) \le \frac{C_R}{\sqrt{\omega_N + \hat p}} $ because the maximum is attained when $k = p$ (here we used  that $\omega_N^{3/2} \ge \sqrt{N}$). 
   Second, by bounding the summands on blocks of length $\sqrt{N}/\sqrt{\omega_N + \hat p},$  we obtain 
    \[
    \sum_{k=p}^n a_kc_k^2 = \sum_{\hat k = \hat n}^{\hat p }
    C_R^2
    \frac{
      \exp\left(- \frac{\hat{p}-\hat{n}}{2} \sqrt{\frac{\omega_N+\hat{p}}{N}} \right)
    }{{\omega_N + \hat k}}
    \lesssim
    \frac{C_R^2\sqrt{N}}{ (\omega_N + \hat p)^{3/2}}.
  \]
  Hence, we conclude that if $N$ is sufficiently large (depending on $\Omega$), then for any $1 \le p< n \le N$,
  \[
  \Sigma  = \max \biggl\{ 2\mathrm{S}^3  \max_{p \leq k \leq n} \left( \sqrt{a_{k}}c_k \right), \sqrt{{\textstyle \sum_{k=p}^{n} } a_{k} c_k^2 } \biggr\} \lesssim  \frac{C_R N^{1/4}}{ (\omega_N + \hat p)^{5/8} \omega_N^{1/8}}
  \]
  
  Since we also have $ \sum_{k=p}^n a_kc_k \lesssim     \frac{C_R\sqrt{N}}{ (\omega_N + \hat p)^{3/2}} \le \frac{C_R}{\sqrt{\Omega\log N}}  \frac{N^{1/4}}{ (\omega_N + \hat p)^{5/8} \omega_N^{1/8}}$, by Proposition \ref{prop:induction}, this implies that there is an event $\A^2_{n,p},$ which we may assume lies within $\bigcap_{p \leq j < k \leq n} \A^1_{k,j}$ for which
  \[
     \big\VERT  \sup_{p \leq k \leq n}| \psi^{(1)}_{k,p,12}| \one_{\A^2_{n,p}} \big\VERT_2 
     \lesssim 
     \frac{ C_R N^{1/4}}{ \omega_N^{1/8} (\omega_N + \hat p)^{5/8}}
     \quad
     \text{and}
     \quad
     \Pr[ \A^{2,c}_{n,p} ] \lesssim N^{2-R\Omega} .  
  \]
  Note that the previous estimate for the probability of the event $\A^{2,c}_{n,p}$ comes from the fact that
  $\A^2_{n,p} \subset\bigcap_{p \leq j < k \leq n} \A^1_{k,j}$,  Proposition \ref{prop:1}  and a union bound. 

  We turn to the estimate for $\psi^{(1)}_{k,p,21},$ which will develop on the estimate for $\psi^{(0)}_{n,k+1,22}.$  The argument is essentially identical, save for that we let
  \[
    c'_k \coloneqq C_R \exp\left(- \frac{\hat{k}-\hat{n}}{4} \sqrt{\frac{\omega_N+\hat{k}}{N}} \right).
  \]
  Then, by bounding the summands on blocks of length $\sqrt{N}/\sqrt{\omega_N + \hat n},$ we now have
    \[ 
    \sum_{k=p}^n a_kc_k^{\prime 2} \le  \sum_{\hat k = \hat n}^{\hat p }
    C_R^2
    \frac{
       \exp\left(- \frac{\hat{k}-\hat{n}}{2}  \sqrt{\frac{\omega_N+\hat{n}}{N}} \right)
    }{{\omega_N + \hat k}}
    \lesssim
    \frac{C_R^2\sqrt{N}}{ (\omega_N + \hat n)^{3/2}},
  \]
  which differs from the bound in the previous case.
  Also differing is that $\max_{p \leq k \leq n}( \sqrt{a_k} c_k' ) \le \frac{C_R}{\sqrt{\omega_N + \hat n}}$, so that  if $N$ is sufficiently large (depending on $\Omega$), 
\[
\Sigma' \lesssim  \frac{C_R N^{1/4}}{ (\omega_N + \hat n)^{5/8} \omega_N^{1/8}} . 
\]
Then, by Proposition \ref{prop:induction}, it holds under the same event  event $\A^2_{n,p},$ 
  \[
     \big\VERT  \sup_{p \leq k \leq n}| \psi^{(1)}_{k,p,12}| \one_{\A^2_{n,p}} \big\VERT_2 
     \lesssim 
     \frac{C_RN^{1/4}}{ \omega_N^{1/8} (\omega_N + \hat n)^{5/8}}
  \]

  To complete the second part of the proof, let us observe that for $p \leq j \leq k \leq n,$
  \[
    \psi^{(1)}_{k,j}\psi^{(0)}_{j-1,p} = \psi^{(1)}_{k,p} - \psi^{(0)}_{k,j}\psi^{(1)}_{j-1,p}\ ,
  \]
  and hence by Lemma~\ref{lem:psistructure},
  \[
    \psi^{(1)}_{k,j,21} = \psi^{(1)}_{k,p,21} - \psi^{(0)}_{k,j,22}\psi^{(1)}_{j-1,p,21}.
  \]
  On $\A^2_{n,p} \subset \bigcap_{p \leq j < k \leq n}^n \A^1_{k,j},$ we have $|\psi^{(0)}_{k,j,22}| \leq C_R$ and so
  \[
    |\psi^{(1)}_{k,j,21}|\one_{\A^2_{n,p}} \leq 2C_R \max_{p < \ell \leq n} |\psi^{(1)}_{\ell,p,12}|,
  \]
  which completes the proof (after adjusting the constant $C>0$).  \end{proof}

%  The increments of $M^{(1)}$ satisfy that for all $\hat p \leq N$
%  \begin{equation} \label{eq:M21doob}
%    \begin{aligned}
%      &|M^{(1)}_{2,1,\hat p+1,\hat n}
%      -M^{(1)}_{2,1,\hat p, \hat n}|^2 \lesssim \frac{(\log N)^2}{\omega_N + \hat p}
%      \left|\Exp_{n}\left[ 
%	\psi_{2,2,n,p}^{(0)}
%      \right]\right|^2
%      , \\
%      &\hat \Exp_{\hat p}\left[(M^{(1)}_{2,1,\hat p+1, \hat n}
%    -M^{(1)}_{2,1,\hat p, \hat n})^2\right] 
%    \lesssim \frac{ \exp\left(-4(\hat p -\hat n)\frac{\sqrt{\omega_N + \hat n}}{C\sqrt{N}}\right) }{(\omega_N + \hat p)}. \\
%    \end{aligned}
%  \end{equation}
%  Hence using Freedman's inequality, there is a $c >0$ sufficiently small, so that uniformly in $z \in K$ and $\hat p \geq \hat n,$ and for all $t > 0,$
%  \[
%    \Pr\left[ 
%      |M^{(1)}_{2,1,\hat p, \hat n}| > t
%    \right] 
%    \lesssim
%    \exp\left( 
%    -\frac{c t^2}{ \frac{t \log N}{\sqrt{\omega_N  + \hat n}} +
%    \frac{1}{\omega_N + \hat n}
%    \min\left\{
%    \hat p - \hat n,
%    \frac{\sqrt{N}}{ (\omega_N + \hat n)^{1/2}}
%  \right\}
%}\right).
%  \]
%\end{proof}
%

\subsection{Estimates for the expected transfer matrix recurrence.} \label{sec:exptransrec}

Before we can turn to estimating $\psi^{(>1)}$, we need a priori estimates on the expected transfer matrix recurrence, see Corollary~\ref{cor:summability} below. 
These estimates are easy to deduce from the conditions \eqref{recall:eta}--\eqref{recall:rho_k} and the following (deterministic) lemma which shows that the norm of a hyperbolic matrix remains basically insensitive to  off--diagonal perturbations.

\begin{lemma}
  Suppose that $d_1,d_2,\epsilon_1,\epsilon_2$ are complex numbers and that $D_i = \sqrt{ |d_i|^2 + |\epsilon_i|^2}$ for $i=1,2.$  Suppose without loss of generality that $D_1 \geq D_2.$  Let $\epsilon = \max\left\{ |\epsilon_1|,|\epsilon_2| \right\}$
\[
  \left\|
  \begin{pmatrix}
    d_1 & \epsilon_1 \\
    \epsilon_2 & d_2 \\
  \end{pmatrix}
  \right\|
  \leq \min\left\{ D_1\biggl(1+\frac{2\epsilon^2}{D_1^2-D_2^2}\biggr), \max\{|d_1|, |d_2|\} + \epsilon \right\}
\]  
  \label{lem:normbound}
\end{lemma}

\noindent The usefulness of this bound lies in the case where diagonal entries of the matrix are well--separated and the off--diagonal entries are perturbatively small.  Then, the  perturbation increases the norm by $O(\epsilon^2)$ as opposed to the usual $O(\epsilon).$
\begin{proof}
  The second inequality is simply subadditivity of the norm, on splitting it into the matrix of $\{d_i\}$ and the matrix of $\{\epsilon_i\}.$ For the first, we explicitly compute the largest singular value $\sigma$, which gives
  \[
    2\sigma^2 = D_1^2 + D_2^2 + \sqrt{ (D_1^2 - D_2^2)^2 + 4U^2 },
  \]
  where $U = |\epsilon_2 \bar d_1 + d_2 \bar \epsilon_1|.$  Bounding the square root using concavity gives
  \[
    \sigma^2 \leq D_1^2 + \frac{U^2}{D_1^2 - D_2^2},
  \]
  again taking square roots and bounding using concavity,
  \[
    \sigma \leq D_1 + \frac{U^2}{2D_1(D_1^2 - D_2^2)}.
  \]
  On estimating $U$ by $2D_1\epsilon,$ the claimed bound follows.
\end{proof}

As a consequence, the norm of an expected $\Exp U_n$ matrix is closer to $1$ than would be expected from a direct estimate.
\begin{lemma}
  There is a constant $C>0$ so that for any $1\le k \leq N,$ 
  \[
    \| \Exp  U_k  \| \leq 1 +  \frac{C\sqrt{N}}{(\omega_N + \hat k)^{5/2}},
    \quad
    \text{ and }
    \quad
    \| U_k \| \leq 1  + \frac{C\sqrt{N}}{(\omega_N + \hat k)^{5/2}} + \frac{C\sqrt{R\Omega\log N}}{(\omega_N + \hat k)^{1/2}} \quad\As
  \]
  \label{lem:randomnormbounds}
\end{lemma}

\begin{proof}
  For the first inequality, we apply the quadratic bound in Lemma~\ref{lem:normbound}, observing that by Proposition~\ref{prop:Hyper} and \eqref{recall:eta}
  for $ \Exp  U_k$, we have $D_1 = 1+ \O(\epsilon^2)$, $D_1-D_2 \gtrsim \sqrt{(\omega_N + \hat k)/N}$
  and  $\epsilon =\O( \frac{1}{\omega_N +\hat{k}})$ --  here we used again that $\frac{1}{\omega_N +\hat{k}} = o\big(\sqrt{\tfrac{\omega_N + \hat k}{N}}\big)$ for all $1\le k \leq N$. 
 For the second inequality, we just use the triangle inequality
  \[
    \| U_k  \|
    \leq
    \| \Exp U_k  \|
    +
    \| U_k  -  \Exp[U_k ] \|,
  \]
  and entrywise bounds on the second term.
\end{proof}

As a corollary, we obtain:
\begin{corollary}\label{cor:summability}
  There is a constant $C>0$ so that for all $1 \leq p < n \leq N,$
  \[
    \| \Exp[ U_n U_{n-1} \dots U_p ] \| \leq 1 + \frac{C}{\Omega \log N}.
  \]
\end{corollary}
\begin{proof}
By submultiplicativity  and using independence, 
\[
 \| \Exp[ U_n U_{n-1} \dots U_p ] \|
 \leq 
 \| \Exp U_n \|
 \cdots
 \| \Exp U_p \|,
\]
and apply Lemma \ref{lem:randomnormbounds}  to obtain $ \| \Exp[ U_n U_{n-1} \dots U_p ] \| \le 1+ \frac{C\sqrt{N}}{\omega_N^{3/2}}$. The final bound follows from the fact that $\omega_N =  N^{1/3} (\Omega \log N)^{2/3}$. 
\end{proof}

\subsection{Moderate deviation for  $\psi^{(>1)}$ on short blocks}
\label{sec:shortblock}

In this section, we give estimates for the moderate deviations for the error $\psi^{(>1)}_{np}$ in the perturbative expansion of products of G$\beta$E transfer matrices $U_n \cdots U_p$ under the  assumption that $(n-p)$ is small, meaning $(n-p) \leq \sqrt{N/\omega_N}.$ 
These estimates follow from a direct application of our general Proposition~\ref{prop:ta} and therefore do not rely on hyperbolicity. 

\begin{proposition} \label{prop:3}
There exists $C>0$ such that  for any $R>0$ and for all $1 \le p< n \le N$ with  $(n-p) \le\frac{\sqrt{N/\omega_N}}{CR}$, there is an event $\A^3_{n,p} \subset \A^2_{n,p} $ measurable with respect to $\F_{n,p}$ with $\P[\A^{3,c}_{n,p}] \lesssim   N^{2-R\Omega} $ so that 
\[
\big\VERT \psi^{(>1)}_{n,p} \1_{\A^3_{n,p}} \big\VERT_1 \le \frac{ C_R\sqrt{N}}{(\omega_N + \hat n)^{9/8}\omega_N^{3/8}}.
\]
%Moreover, we have always have if  $(n-p) \le R^{-1} \sqrt{N/\omega_N}$,
%\[
%\big\VERT \psi^{(>1)}_{n,p} \1_{\A^3_{n,p}} \big\VERT_2 \lesssim \frac{\sqrt{N \log N}}{R \omega_N^{1/4}(\omega_N +\hat{n})^{5/4}} . 
%\]
\end{proposition}

\begin{proof}
To apply Proposition \ref{prop:ta},  we must estimate the parameters $u,\mu, \Delta,$ and $\sigma^2.$ Recall that $V_k = \diag(U_k),$ and therefore using the conditions \eqref{recall:eta}, we obtain the bounds  for any $k\in [1,N]$, 
  \[
    \|U_k - \E U_k\|^2 ,     \|V_k - \E V_k\|^2\lesssim \frac{R\Omega\log N}{\omega_N + \hat k} \As,
    \quad
    \Exp\|U_k - \E U_k\|^2 \lesssim \frac{1}{\omega_N + \hat k} 
    \quad
    \text{and}
    \quad
   \Exp  \|U_k -V_k\|^2 \lesssim \frac{1}{\omega_N + \hat k}.
  \]
  From Corollary \ref{cor:summability}, we have $u \le 2$ if $N$ is sufficiently large.  Since $\|\Exp U_k - \Exp V_k\| \lesssim \frac{1}{\omega_N + \hat k}$ and we assume that $(n-p) \le \frac{\sqrt{N/\omega_N}}{R'}$ for a constant $R'$ to be chosen below, we have
    \[
  \mu = \sum_{k=p}^n \| \Exp(U_k - V_k) \| \lesssim \log\left(\frac{\omega_N+\hat{p}}{\omega_N + \hat{n}}\right) \le \frac{\hat{p}-\hat{n}}{\omega_N+\hat{n}} \le  \frac{\sqrt{N/\omega_N}}{R'(\omega_N+\hat{n})}
  \]
  In particular this shows that $  \mu \lesssim \frac{\sqrt{N}}{R'\omega_N^{3/2}} = \frac{1}{R'\Omega\log N}$ so that $u\mu < 1/2$ as required if $N$ is sufficiently large.  For the other parameters, we can take
  \[
   \sigma^2  = C\mu =  \frac{C}{R'} \frac{\sqrt{N/\omega_N}}{\omega_N+\hat{n}}
    \qquad
    \text{and}
    \qquad
    \Delta^2 = C  \frac{R\Omega \log N}{\omega_N + \hat n} 
  \]
  for a large constant $C \ge 1$. 
Moreover, from Proposition \ref{prop:2}, there is an event $\A^2= \A^2_{n,p}$ on which 
  \[
    \big\VERT \max_{p \leq k \leq n} \| \psi^{(1)}_{k,p}\| \one_{\A^2} \big\VERT_2 \leq  D = \frac{C_RN^{1/4}}{(\omega_N + \hat n)^{5/8}\omega_N^{1/8}} . 
  \]
  Therefore by Proposition \ref{prop:ta}, there is an event $\A^3 \subset \A^2$ which is $\F_{n,p}$ measurable such that
  \[ 
    \VERT \psi^{(>1)}_{n,p} \one_{\A^3} \VERT_1
    \lesssim \frac{C_R}{\sqrt{R'}}\frac{\sqrt{N}}{(\omega_N + \hat n)^{9/8}\omega_N^{3/8}}.
  \]
  Finally,  it holds for $N$ sufficiently large (depending on $\Omega$ and adjusting $C$ if required),  
  \[\begin{aligned}
  \Pr[\A^{2}\setminus \A^{3}]  \lesssim \exp\left(-\frac{1}{128 u^2\sigma^2} \wedge   \frac{ \sigma^2}{\Delta^2}  \right)  \le N^{- \frac{R'\Omega}{C} } 
 \end{aligned} \]
 Here we have used that $u\le 2$,  $\sigma^2 \le  \frac{C}{R'\Omega\log N}$ and $\frac{\sigma^2}{\Delta^2} = \frac{N^{1/3}}{R'R(\Omega\log N)^{4/3}}$. \
 By choosing $R' = C R$ and using Proposition \ref{prop:2} to control $ \Pr[ \A^{2,c} ]$,  this completes the proof.
\end{proof}

\begin{remark} \label{rk:3}
Let $X$ be a random variable such that  $\VERT X\VERT_1 \le \sigma$ and define the event
$\mathscr{E}_{r_\beta} = \big\{|X| \le \mathrm{A} \sigma \log N\big\}$. Then, it follows from the discussion in Section~\ref{sec:concentration}, that $\VERT X\1_{\mathscr{E}_{\mathrm{A}}}\VERT_2 \lesssim \sigma \sqrt{\mathrm{A}\log N}$  and that 
$\P[\mathscr{E}_{r_\beta}^c] \le 2 N^{-\mathrm{A}}$. 
\end{remark}

By  Proposition~\ref{prop:3} and Remark~\ref{rk:3},  we can also control the sub--Gaussian norm of the matrix $ \psi^{(>1)}_{n,p}$ on the event  $\A^4_{n,p}= \A^3_{n,p} \cap \Big\{\| \psi^{(>1)}_{n,p}\| \le  \frac{\Omega RC_R\sqrt{N}\log N}{(\omega_N + \hat n)^{9/8}\omega_N^{3/8}}\Big\}$. 

\begin{corollary} \label{cor:3}
There is an event $\A^4_{n,p} \subset \A^3_{n,p} $ measurable with respect to $\F_{n,p}$ with $\P[\A^{4,c}_{n,p}] \lesssim   N^{2-R\Omega} $ so that if $(n-p) \le\frac{\sqrt{N/\omega_N}}{CR}$, 
\vspace{-.5cm}
\[
\big\VERT \psi^{(>1)}_{n,p} \1_{\A^4_{n,p}} \big\VERT_2 \le \frac{ C_R\sqrt{\Omega N \log N}}{(\omega_N + \hat n)^{9/8}\omega_N^{3/8}}.
\]
\end{corollary}

\subsection{Moderate deviations for long blocks} %{G$\beta$E long block estimate}
\label{sec:longblock}

In this section, we use the hyperbolic character of the transfer matrices $U_k$ to bootstrap the estimates from Proposition~\ref{prop:3} from short blocks to long blocks. 
Although, our result does not give us yet enough control to show that the error $ \psi^{(>1)}_{n,p}$ are small for all $N\ge n\ge p\ge 0$. 
Instead, we obtain uniform  bound for the norm $\| U_n \cdots U_p\|$ which holds with overwhelming probability.

\begin{proposition} \label{prop:4}
There exists a constant $C>0$ such that for any $R>0$, it holds
%For all $R>0$, there exists a constant $C_R>0$ such that 
\[
\P\left[ \max_{1\le p <n \le N}\| U_n \cdots U_p\|  \ge C_R' \right] \lesssim_R N^{5-R\Omega} \, , \qquad C_R' = e^{e^{CR}} . 
\]
\end{proposition}

\begin{proof}
Throughout the proof, let $C_R = e^{CR}$ as in Proposition~\ref{prop:3} and we condition on the event  $\bigcap_{j=1}^\chi \A^3_{\kappa_{j+1}, \kappa_j}$ (conditionally on this event, the blocks $(U_{\kappa_{j+1}} \cdots U_{\kappa_j})_{j=1}^\chi$ remain independent). 
By submultiplicativity, 
\[
\log \| U_n \cdots U_p\|  \le \sum_{j=1}^\chi  \log \| U_{\kappa_{j+1}} \cdots U_{\kappa_j}\|  , 
\]
where $\kappa_1= p$, $\kappa_{\chi+1}=n$ ($p<n$ are fixed for now) and we impose the conditions: $ \frac{\sqrt{N/\omega_N}}{2CR}  \le \kappa_{j+1} - \kappa_j \le \frac{\sqrt{N/\omega_N}}{CR} $.
Observe that for any $x, y >0$, we have $\log(x+y) \le y+ \log_+(x)  $. 
By definition, this implies that for any $j=1, \dots, \chi$, 
\begin{equation} \label{eq:logU}
 \log \| U_{\kappa_{j+1}} \cdots U_{\kappa_j}\|  \le \| \psi^{(>1)}_{\kappa_{j+1}, \kappa_j} \| + \log_+\big(\| \psi^{(\le 1)}_{\kappa_{j+1}, \kappa_j} \| \big) . 
\end{equation}
So $\log \| U_n \cdots U_p\|  \le \sum_{j=1}^\chi  \| \psi^{(>1)}_{\kappa_{j+1}, \kappa_j} \| + \sum_{j=1}^\chi  \log_+\big(\| \psi^{(\le 1)}_{\kappa_{j+1}, \kappa_j} \| \big)$ 
and we will estimate both terms separately. 

For the first term, we have by Proposition~\ref{prop:3}, 
\[ \begin{aligned}
\sum_{j=1}^\chi \VERT \psi^{(>1)}_{\kappa_{j+1}, \kappa_j} \VERT_1
& \le  \frac{C_R \sqrt{N}}{\omega_N^{3/8}}  \sum_{j=1}^{\chi}  \frac{1}{(\omega_N + \widehat{\kappa_{j+1}})^{9/8}}\\
& \lesssim  RC_R \omega_N^{1/4} \sum_{\hat{k}=1}^{+\infty}  \frac{1}{(\omega_N + \hat{k})^{9/8}} \\
\end{aligned}\]
since the block size $ \kappa_{j+1} - \kappa_j \ge \frac{\sqrt{N/\omega_N}}{2CR}$. By \eqref{eq:Xmoment} and adjusting $C>0$, this shows that 
\begin{equation} \label{eq:psichi1}
\sum_{j=1}^\chi \VERT \psi^{(>1)}_{\kappa_{j+1}, \kappa_j} \VERT_1\lesssim C_R
\qquad\text{so that}\qquad
\E\big[\textstyle{\sum_{j=1}^\chi \|  \psi^{(>1)}_{\kappa_{j+1}, \kappa_j} \|}\big] \lesssim C_R .
\end{equation}
Then, to show that  the random variable $\sum_{j=1}^\chi \|  \psi^{(>1)}_{\kappa_{j+1}, \kappa_j} \|$ is concentrated around its mean, we use the Bernstein's inequality \eqref{eq:actualbernstein}.  By a similar computation, we have 
\[ \begin{aligned}
\sum_{j=1}^\chi \VERT \psi^{(>1)}_{\kappa_{j+1}, \kappa_j} \VERT_1^2
& \lesssim  \frac{C_R^2N}{ \omega_N^{3/4}}  \sum_{j=1}^{\chi}  \frac{1}{(\omega_N + \widehat{\kappa_{j+1}})^{9/4}}\\
&\lesssim \frac{RC_R^2\sqrt{N}}{\omega_N^{1/4} } \sum_{\hat{k}=1}^{+\infty}  \frac{1}{(\omega_N + \hat{k})^{9/4}}
 \lesssim \frac{C_R^2\sqrt{N}}{\omega_N^{3/2}}  = \frac{C_R^2}{  \Omega \log N}. 
\end{aligned}\]
Moreover, since $\max_{j \le \chi}\VERT \psi^{(>1)}_{\kappa_{j+1},\kappa_j}\VERT_1 \le\frac{C_R\sqrt{N}}{\omega_N^{3/2}} =  \frac{C_R}{ \Omega \log N}$, 
by \eqref{eq:Xcentering} and \eqref{eq:actualbernstein} with $t= RC_R$, this implies that 
\[
    \Pr\left[ \left| \textstyle{ \sum_{j=1}^\chi \|  \psi^{(>1)}_{\kappa_{j+1}, \kappa_j} \|} - \E\big[\textstyle{\sum_{j=1}^\chi \|  \psi^{(>1)}_{\kappa_{j+1}, \kappa_j} \|}\big] \right| \geq R C_R \right] \leq 2\exp\left( - c^{-1} R \Omega \log N  \right).
\]
After adjusting $R$ and $C$, by \eqref{eq:psichi1},  this shows that so that with probability (at least) $1-2N^{-R\Omega}$, 
\begin{equation} \label{est3}
\sum_{j=1}^\chi \| \psi^{(>1)}_{\kappa_{j+1}, \kappa_j}\| \lesssim C_R . 
\end{equation}

Now, we have to deal with the second term on the RHS of \eqref{eq:logU}. With a constant $M\ge 2$, let us define for $j=1, \dots , \chi$, the events
\begin{equation} \label{def:eventEj}
\mathcal{E}_j = \left\{ |\psi^{(0)}_{\kappa_{j+1},\kappa_j, 22}| \le 1- M^{-2} ,  \| \psi^{(1)}_{\kappa_{j+1}, \kappa_j}\|^2 \le M^{-2}/2 \right\} .
\end{equation}
Using the first bound from Lemma \ref{lem:normbound}, we have  conditionally on ${\mathcal{E}_j }$, 
\[
\| \psi^{(\le 1)}_{\kappa_{j+1}, \kappa_j} \|  \le  D_1\biggl(1+\frac{2 \| \psi^{(1)}_{\kappa_{j+1}, \kappa_j}\|^2}{D_1^2-D_2^2}\biggr)  \le 1+ c M^2 \| \psi^{(1)}_{\kappa_{j+1}, \kappa_j}\|^2, 
\]
where we used  $\diag(\psi^{(\le 1)}_{\kappa_{j+1}, \kappa_j} ) = \psi^{(0)}_{\kappa_{j+1}, \kappa_j} $  so that $1 \le D_1 \le 1+ \| \psi^{(1)}_{\kappa_{j+1}, \kappa_j}\|^2/2 $ and $D_2^2 \le 1- M^{-2}/2$. 
This shows that
\begin{equation} \label{est7}
\sum_{j=1}^\chi  \log_+\big(\| \psi^{(\le 1)}_{\kappa_{j+1}, \kappa_j} \| \big) \1_{\mathcal{E}_j } \lesssim M^2 \sum_{j=1}^\chi \| \psi^{(1)}_{\kappa_{j+1}, \kappa_j}\|^2 . 
\end{equation}

Then, observe that by the Cauchy--Schwartz inequality and by Proposition~\ref{prop:2}, 
 \[
  \sum_{j=1}^\chi \big\VERT \| \psi^{(1)}_{\kappa_{j+1}, \kappa_j}\|^2 \big\VERT_1 \le   \sum_{j=1}^\chi \big\VERT  \psi^{(1)}_{\kappa_{j+1}, \kappa_j}  \big\VERT_2^2
  \lesssim  \frac{\sqrt{N}}{\omega_N^{1/4}} \sum_{j=1}^\chi  \frac{C_R^2}{(\omega_N +  \widehat{\kappa_{j+1}})^{5/4}} .
 \]
 Since the block size $ \kappa_{j+1} - \kappa_j \ge \frac{\sqrt{N/\omega_N}}{2CR}$, this shows that 
  \[
  \sum_{j=1}^\chi \big\VERT \| \psi^{(1)}_{\kappa_{j+1}, \kappa_j}\|^2 \big\VERT_1 \lesssim  \omega_N^{1/4} \sum_{\hat{k}=1}^{+\infty}  \frac{R C_R^2}{(\omega_N +  \hat{k})^{5/4}} \lesssim  C_R^2. 
  \]
  In particular, this implies that $  \E\big[\sum_{j=1}^\chi \| \psi^{(1)}_{\kappa_{j+1}, \kappa_j}\|^2\big] \lesssim  C_R^2$. Moreover, by a similar argument, we also have 
 \[
  \sum_{j=1}^\chi \big\VERT \| \psi^{(1)}_{\kappa_{j+1}, \kappa_j}\|^2 \big\VERT_1^2  
  \lesssim  \frac{N}{\sqrt{\omega_N}}\sum_{j=1}^\chi  \frac{C_R^4}{(\omega_N +  \widehat{\kappa_{j+1}})^{5/2}} 
    \lesssim R C_R^4 \frac{\sqrt{N}}{\omega_N^{3/2}}  =  \frac{ C_R^4}{\Omega\log N}. 
 \]  
 This shows that the  random variable $\sum_{j=1}^\chi \| \psi^{(1)}_{\kappa_{j+1}, \kappa_j}\|^2$ is concentrated around its mean, which is of order $1$. Namely, since the summands are independent and $ \max_{j \le \chi}\VERT \psi^{(1)}_{\kappa_{j+1}, \kappa_j} \VERT_2 \lesssim C_R^2   \frac{\sqrt{N}}{\omega_N^{3/2} }=  \frac{C_R^2}{\Omega\log N}$,   by Bernstein's inequality \eqref{eq:actualbernstein} with $t= R C_R^2$, this implies that 
 \[
    \Pr\left[ \left| \textstyle{ \sum_{j=1}^\chi \|  \psi^{(1)}_{\kappa_{j+1}, \kappa_j} \|^2} - \E\big[\textstyle{\sum_{j=1}^\chi \|  \psi^{(1)}_{\kappa_{j+1}, \kappa_j} \|^2}\big] \right| \geq R C_R^2 \right] \leq 2\exp\left( - c^{-1} R \Omega \log N\right) .
\]
After adjusting $R$ and $C$, hese bounds show that $\sum_{j=1}^\chi \|  \psi^{(1)}_{\kappa_{j+1}, \kappa_j} \|^2 \lesssim  C_R$  with probability at least $1-2N^{R\Omega}$. 
By \eqref{est7}, this implies that with the same (overwhelming) probability 
\begin{equation} \label{est2}
\sum_{j=1}^\chi  \log_+\big(\| \psi^{(\le 1)}_{\kappa_{j+1}, \kappa_j} \| \big) \1_{\mathcal{E}_j } \lesssim C_R
\end{equation}

To complete the proof, it remains to show that with overwhelming probability, the random variable 
$\sum_{j=1}^\chi  \log_+\big(\| \psi^{(\le 1)}_{\kappa_{j+1}, \kappa_j} \| \big) \1_{\mathcal{E}_j^c}$ remains bounded by a constant. 
The idea is that the events $\mathcal{E}_j^c$ -- see \eqref{def:eventEj} -- are independent with a small probability to occur. 
Using the second bound form Lemma \ref{lem:normbound}, we have for $j=1, \dots , \chi$, 
\[
\big\| \psi^{(\le 1)}_{\kappa_{j+1}, \kappa_j} \big\|   \le  |\psi^{(0)}_{\kappa_{j+1}, \kappa_j, 22} |  \vee 1+ \|  \psi^{(1)}_{\kappa_{j+1}, \kappa_j} \| . 
\]
By Proposition~\ref{prop:1}, $\max_{j\le \chi} |\psi^{(0)}_{\kappa_{j+1}, \kappa_j, 22} | \le C_R$ with  (at least) probability $1- N^{1-R\Omega}$. 
Moreover, by Proposition~\ref{prop:2},
\[
   \max_{j\le \chi} 
\big\VERT \psi^{(1)}_{\kappa_{j+1}, \kappa_j} \big\VERT_2  \lesssim  C_R \tfrac{N^{1/4}}{\omega_N^{3/4}}  = \tfrac{C_R}{\sqrt{\Omega\log N}} . 
\]

Hence, $\max_{j\le \chi}\|  \psi^{(1)}_{\kappa_{j+1}, \kappa_j} \| \lesssim \sqrt{R} C_R$ with probability $1- N^{1-R\Omega}$. 
This implies that  
$\max_{j\le \chi} \big\| \psi^{(\le 1)}_{\kappa_{j+1}, \kappa_j} \big\|  \le C_R$ with probability $1- 2N^{1-R\Omega}$. Thus, if we set
  $\mathrm{X} := \sum_{j=1}^\chi \1_{\mathcal{E}_j^c}$, as $C_R = e^{CR}$, it holds for any $t>0$
\begin{equation} \label{est5}
\P\left[ \bigg\{ \sum_{j=1}^\chi \log_+\| \psi^{(\le 1)}_{\kappa_{j+1}, \kappa_j} \| \1_{\mathcal{E}_j^c}  \ge CRt  \bigg\} \bigcap_{j=1}^\chi \A^3_{\kappa_{j+1}, \kappa_j} \right] \le  \P\left[ \big\{ \mathrm{X}  \ge t \big\} \cap_{j=1}^\chi \A^3_{\kappa_{j+1}, \kappa_j}  \right]  + 2N^{1-R\Omega} . 
\end{equation}
Observe that $\mathrm{X}$ is a sum of independent Bernoulli random variables whose mean satisfies 
\[
\P[\mathcal{E}_j^c \cap  \A^3_{\kappa_{j+1}, \kappa_j}] \le \P[|\psi^{(0)}_{\kappa_{j+1},\kappa_j, 22}| > 1- M^{-2}] + \P\left[\big\{ \| \psi^{(1)}_{\kappa_{j+1}, \kappa_j}\|^2 > M^{-2}/2 \big\} \cap  \A^3_{\kappa_{j+1}, \kappa_j} \right] .
\]
Then, using the second bound from Proposition~\ref{prop:1} together with Proposition~\ref{prop:2} and \eqref{eq:Xtail} with $p=1$, this shows that
\[ \begin{aligned}
\P[\mathcal{E}_j^c \cap  \A^3_{\kappa_{j+1}, \kappa_j}]  & \le
 \exp\left( - \frac{\sqrt{\omega_N}(\omega_N+ \widehat{\kappa_{j+1}})}{CM^2\sqrt{N}} \right)
 +  \exp\left( -  \frac{\omega_N^{1/4}(\omega_N + \widehat{\kappa_{j+1}})^{5/4}}{ C M^2C_R^2\sqrt{N}} \right)  \\
&\le 2  \exp\left( -  \frac{\omega_N^{3/2} / \sqrt{N}+(\chi-j)}{C_R'}\right)
 \end{aligned}\]
where $C_R' = 2C^2M^2 R C_R^2$  and we used that by assumptions: $ \widehat{\kappa_{j+1}} \ge (\chi-j) \frac{\sqrt{N/\omega_N}}{2CR}   $ for all $j=1,\dots, \chi$. 

\medskip

  Now, we use the previous estimate, to obtain concentration for the random variable $\mathrm{X}$. Observe that with $p_N = \exp\left( -  \frac{\omega_N^{3/2}}{C_R'\sqrt{N}} \right) = N^{-\Omega/C_R'}$, it holds for any $\lambda>0$, 
 \[\begin{aligned}
 \E\bigg[e^{\lambda \mathrm{X} } \1_{\bigcap_{j=1}^\chi \A^3_{\kappa_{j+1}, \kappa_j}}\bigg] & \le  \prod_{j=0}^{+\infty} \left( 1+ 2(e^\lambda-1) p_N e^{- j/C_R'} \right) \\
 &\le \exp\left( C_R''(e^\lambda-1) p_N \right)
\end{aligned} \]
with $C_R'' = 2(1- e^{- 1/C_R'} )^{-1}$. 
 By Markov's inequality, this shows that by picking $\lambda = \log(p_N^{-1})$, then for any $t>0$, 
 \[
 \P\left[ \big\{ \mathrm{X}  \ge t \big\} \cap_{j=1}^\chi \A^3_{\kappa_{j+1}, \kappa_j}  \right] \le \exp\left( C_R'' (e^\lambda-1) p_N - \lambda t \right) = e^{C_R''} p_N^{t} . 
 \]
Choosing $t= R C_R'$ and adjusting $C$,  we deduce from \eqref{est5} that 
 \begin{equation} \label{est4}
 \P\left[ \bigg\{ \sum_{j=1}^\chi \log_+\| \psi^{(\le 1)}_{\kappa_{j+1}, \kappa_j} \| \1_{\mathcal{E}_j^c}  \ge C_R' R^2    \bigg\} \bigcap_{j=1}^\chi \A^3_{\kappa_{j+1}, \kappa_j} \right] \lesssim e^{C_R''} N^{-R\Omega} . 
 \end{equation}

\medskip 

In the end, we can choose $M=2$, so that by adjusting $C$ again,  both $R^2C_R' , C_R'' \lesssim C_R$. 
By combining the estimates  \eqref{est3}, \eqref{est2}, \eqref{est4} with \eqref{eq:logU}, we conclude that conditionally on the event  $\bigcap_{j=1}^\chi \A^3_{\kappa_{j+1}, \kappa_j}$, it holds  with probability (at least) $ 1- 3e^{C_R}N^{-R\Omega}$, 
 \begin{equation} \label{eq:goodUevent}
\log \| U_n \cdots U_p\|  \le \sum_{j=1}^\chi  \log \| U_{\kappa_{j+1}} \cdots U_{\kappa_j}\|  \lesssim  C_R. 
 \end{equation}
 To complete the proof, let us recall that according to Proposition~\ref{prop:3} and by a union bound, 
$\P\left[\bigcap_{1\le p <n \le N}\bigcap_{j=1}^\chi \A^3_{\kappa_{j+1}, \kappa_j}\right] \ge 1-  N^{5-R\Omega}$. 
\end{proof}

\subsection{Final comparison}
\label{sec:final}

In this section, we conclude the proof of Theorem \ref{thm:main} by showing that   $\prod_{k=1}^NU_{k}  \simeq \psi_{N,1}^{(0)}$ with overwhelming probability.
Note that by Proposition~\ref{prop:2}, we already have control of $\psi_{N,1}^{(1)}$ (by \eqref{eq:Xtail}, this proposition implies  that for any small $\epsilon>0$, we have $\| \psi_{N,1}^{(1)}\|  \lesssim \epsilon$  with probability at least $1- N^{-\epsilon^2\Omega/C_R^2}$).
Hence, according to the perturbative expansion \eqref{eq:psiTaylor}, we would like to establish some moderate deviation control for  $\psi^{(>1)}$ knowing that by Proposition~\ref{prop:4}, the product of the matrices $U_k$ remains bounded with overwhelming probability. 

\medskip

Let us observe that by Lemma~\ref{lem:psistructure} and \eqref{eq:psi1step}, we have  for any $1\le  n<m \le N$, 
\begin{equation} \label{psidecomp}
  \begin{aligned}
    &\psi_{m,n}^{(>1)} \left( \begin{smallmatrix} 0 & 0 \\ 0 &1\end{smallmatrix} \right) 
    = \sum_{k= n}^{m} U_{m} \dots U_{k+1} \left( \begin{smallmatrix} 0 & 0 \\ 0 &1\end{smallmatrix} \right)  \eta_{k,21}  \psi^{(1)}_{k-1,n,12}, \\
    &\psi_{m,n}^{(>1)} \left( \begin{smallmatrix} 1 & 0 \\ 0 &0\end{smallmatrix} \right) 
    = \sum_{k= n}^{m} \psi^{(>0)}_{m,k+1} \left( \begin{smallmatrix} 0 & 0 \\ 1 &0\end{smallmatrix} \right)  \eta_{k,21}  \psi^{(0)}_{k-1,n,11}  . 
 \end{aligned}
\end{equation}
Each term will be control separately, as they behave substantially differently.  

The first term is easy to handle since it depends only on $\psi^{(1)}_{k,n,12}$ which converges fast to 0 as $(n-k) \to \infty$ (see the proof of  Proposition~\ref{prop:8} below). 
The second  term is more difficult to control and it tends to be larger.  However, the process $\psi^{(1)}_{m,k,21},$ which is the leading contribution to the second term, tends to regenerate as $m$ is held fixed and $k$ decreases.  
This leads to better concentration properties, which salvages the estimate.  

\medskip

Recall that  $\psi^{(0)}_{k,n,11} = 1,$   $n = N- \hat{n}$ and let $\hat \filt_{\hat n}$ denotes  the $\sigma$--algebra generated by $\left\{ \eta_{k,ij} : \hat{k}\le \hat{n} \right\}$. 
By definitions, $\psi_{m,\hat n}^{(>1)}$ is adapted to $\hat \filt_{\hat n}$ as a process in $\hat n$ with $m$ held fixed.
Thus we can  perform a Doob decomposition $\psi_{m,n}^{(>1)} \left( \begin{smallmatrix} 1 & 0 \\ 0 &0\end{smallmatrix} \right)  =  Z_{\hat{n},\hat{m}} + \Upsilon_{\hat{n}, \hat{m}},$  
where for any $1\le n<m \le N$, 
\begin{equation}\label{eq:Zinc}
  \begin{aligned}
    Z_{\hat{n},\hat{m}} &=  \sum_{k=n}^{m} \psi^{(>0)}_{m,k+1} \left( \begin{smallmatrix} 0 & 0 \\ 1 &0\end{smallmatrix} \right)  (\eta_{k,21} - \Exp \eta_{k,21}), \\
    \Upsilon_{\hat{n},\hat{m}} &=  \sum_{k=n}^{m} \psi^{(>0)}_{m,k+1} \left( \begin{smallmatrix} 0 & 0 \\ 1 &0\end{smallmatrix} \right)  \Exp \eta_{k,21}. \\
  \end{aligned}
\end{equation}
Recall that 
$\psi^{(>0)}_{m,n} = \psi^{(1)}_{m,n} + \psi^{(>1)}_{m,n} $, 
so that by combining Proposition~\ref{prop:2} with Corollary~\ref{cor:3}, we obtain the estimate valid for all $1\le  n<m \le N$ such that $(m-n) \le \frac{\sqrt{N/\omega_N}}{C R}$, 
\[ \begin{aligned}
\left\VERT \psi^{(>0)}_{m,n}  \1_{ \A^4_{m,n}} \right\VERT_2
& \le \left\VERT \psi^{(1)}_{m,n}  \1_{\A^2_{m,n}} \right\VERT_2 +\left\VERT \psi^{(>1)}_{m,n}  \1_{\A^4_{m,n}} \right\VERT_2 \\
& \le \frac{C_RN^{1/4}}{(\omega_N + \hat{n})^{5/8} \omega_N^{1/8}}  
\left(1+ \frac{N^{1/4}\sqrt{R\Omega\log N}}{ \omega_N^{3/4}}  \right) . 
\end{aligned}\]
Note that we used that $\A^4_{m,n} \subset \A^2_{m,n}$. 
As $\frac{N^{1/4}\sqrt{R\Omega\log N}}{ \omega_N^{3/4}}= \sqrt{R}$, by adjusting the constant $C$, this shows that on short blocks,
\vspace{-.5cm}
\begin{equation} \label{est:psi>0}
\left\VERT \psi^{(>0)}_{m,n}  \1_{ \A^4_{m,n}} \right\VERT_2 \le \frac{C_R N^{1/4}}{(\omega_N + \hat{n})^{5/8} \omega_N^{1/8}}   . 
\end{equation}
%In particular, by \eqref{eq:Zinc} and \eqref{recall:eta}, this yields the following bound $\big($valid if $(m-n) \le \frac{\sqrt{N/\omega_N}}{CR\Omega}\big)$ for the increments of the martingale $Z_{\hat{n},\hat{m}}$, 
%\begin{equation} \label{est:Zinc}
%\left\VERT \E\left[  \| Z_{\widehat{n+1},\hat{m}}- Z_{\hat{n},\hat{m}} \|^2 \middle| \hat \filt_{\hat n} \right] \1_{ \A^4_{m,n}} \right\VERT_2    \lesssim \frac{ \var \big[\eta_{n,21}\big] N^{1/4}}{(\omega_N + \hat{n})^{5/8} \omega_N^{1/8}}   
%\lesssim  \frac{N^{1/4}}{(\omega_N + \hat{n})^{13/8} \omega_N^{1/8}}   . 
%\end{equation}
%
%\medskip

In the following, we wish to apply the Freedman--Tropp's inequality from Theorem~\ref{thm:M} to obtain a tail bound for the martingale $Z_{\hat{n},\hat{m}}$ when $\hat{m}=0$.  To this end, we need an a priori estimate for its quadratic variation. 

\begin{proposition} \label{prop:5}
Let $Q$ be the (total) quadratic variation of the martingale $(Z_{\hat{n},0})_{\hat{n}=0}^{N} :$
\begin{equation*} 
Q =  \sum_{\hat{n} = 0}^{N-1}\E\left[  \| Z_{\widehat{n+1},0}- Z_{\hat{n},0} \|^2 \middle| \hat \filt_{\hat n} \right]  .
\end{equation*}
For any $R>0$,  there exists a constant $c_R>0$ and an event $\mathscr{G}_1$ with $\P\left[\mathscr{G}_1^c \right] \lesssim_R N^{5-R\Omega}  $ so that 
\[
\big\VERT \big( Q - \tfrac{c_R}{\Omega \log N} \big)_+  \1_{\mathscr{G}_1 } \big\VERT_2^2
 \le  \frac{c_R}{(\Omega\log N)^{3}} . 
\]
\end{proposition}

Unfortunately,  the estimates \eqref{est:psi>0} are only valid on short blocks and they are not precise enough for our application to Proposition~\ref{prop:5}. 
To get our bound for quadratic variation  $Q$, we have to exploit the independence of the matrix $U_k$ by using a blocking argument. The details of the proof are given in the next section.
By combining this bound  with the Freedman--Tropp's inequality from Theorem~\ref{thm:M}, we also deduce in Section~\ref{sect:extraproofs}  a tail bound for the martingale $Z_{\hat{n},0}$. 

\begin{proposition} \label{prop:6}
For any $R>0$,   there exists a constant $c_R>0$ and an event $\mathscr{G}_2$ with $\P\left[\mathscr{G}_2^c \right] \lesssim_R  N^{5-R\Omega}$ so that 
\[
  \big\VERT \max_{\hat{n}=0,\dots, N-1}\big\|Z_{\hat{n},0}\big\| \1_{\mathscr{G}_2} \big\VERT_2 \le   \frac{c_R}{ \sqrt{\Omega\log N}} . 
\]
\end{proposition}

Returning to our original considerations, we are now ready to provide a tail bound for  the random variable
$\psi_{N,n}^{(>1)} \left( \begin{smallmatrix} 1 & 0 \\ 0 &0\end{smallmatrix} \right)$. 
The proof relies on the notation from Section~\ref{sect:extraproofs}, in particular on the decomposition  \eqref{eq:psialphadecomposition} below.

\begin{proposition} \label{prop:7}
For any $R>0$,  there exists a constant $c_R>0$ and an event $\mathscr{G}_3$ with $\P\left[\mathscr{G}_3 \right] \lesssim_R N^{5-R\Omega}$ so that 
\[
  \big\VERT \max_{n=1,\dots, N-1}\big\|\psi_{N,n}^{(>1)} \left( \begin{smallmatrix} 1 & 0 \\ 0 &0\end{smallmatrix} \right)\big\| \1_{\mathscr{G}_3} \big\VERT_2 \le   \frac{c_R}{\sqrt{\Omega\log N}} . 
\]
\end{proposition}

\begin{proof}
  Let us recall that $\psi_{N,n}^{(>1)} \left( \begin{smallmatrix} 1 & 0 \\ 0 &0\end{smallmatrix} \right)  =  Z_{\hat{n},0} + \Upsilon_{\hat{n}, 0} $ for any $n=1,\dots, N$ where $Z$ and $\Upsilon$ are defined in \eqref{eq:Zinc}. 
By Proposition~\ref{prop:6}, we already have good control of the martingale part $ Z_{\hat{n},0}$, so it suffices to establish that  there also exists an event $\mathscr{C}$ with 
$\P[\mathscr{C}^c] \lesssim N^{5-R\Omega}$ such that 
\begin{equation}\label{Upsilon:est} 
  \big\VERT \max_{\hat{n}=0,\dots, N-1}\big\|\Upsilon_{\hat{n},0}\big\| \1_{\mathscr{C}} \big\VERT_2 \le   \frac{c_R}{\sqrt{\Omega\log N}} . 
\end{equation}
 
 We rely on the blocking argument and the notation from the proof of Proposition~\ref{prop:5}. 
 Recall from \eqref{eq:Zinc} that
\[ \begin{aligned}
 \Upsilon_{\hat{n}, 0}  
 & =  \sum_{m=n}^{N} \E[\eta_{m,21}]   \psi^{(>0)}_{N, m+1} \left( \begin{smallmatrix}  0 & 0 \\ 1 &0 \end{smallmatrix} \right).
\end{aligned}\]
Then, using the decomposition  \eqref{eq:psialphadecomposition} below with $\alpha=1$ in this case and the triangle inequality, we obtain that for all $1\le m <N$, 
\[
\left\| \psi^{(>0)}_{N,m}  \right\| \1_{\mathscr{C}_m}
\lesssim_R  \sum_{\ell= 0}^{\chi_{m}} e^{-\ell/ 4CR}  \big\| Y^{(\ell)}_{m}  \big\| \1_{\mathscr{B}^{(\ell)}_{m}} .
\]
The vents $\mathscr{C}_m$ are as in \eqref{P:C} and theexponential factor comes from the estimate \eqref{est9}.
Hence, exactly like \eqref{Q:decomposition}, since $\big| \E \eta_{m,21}  \big| \lesssim \frac{1}{\omega_N+\hat{m}} $ for $m=1,\dots, N$,  this shows that  with $\mathscr{C}= \cap_{m = 1}^{N} \mathscr{C}_m $,
\begin{equation} \label{est10}
\max_{\hat{n}=0,\dots, N-1}\|   \Upsilon_{\hat{n}, 0} \| \1_{\mathscr{C}}  \lesssim_R \sum_{\ell= 0}^{\chi} e^{-\ell/ 4CR} \bigg(\sum_{m = 1}^{N} \frac{\big\| Y^{(\ell)}_{m}  \big\| \1_{\mathscr{B}^{(\ell)}_{m}}  }{\omega_N+\hat{m}}  \1_{\chi_m \ge \ell} \bigg). 
\end{equation}
Moreover, using the estimate \eqref{Y:est} and the triangle inequality, we obtain uniformly  for all $\ell= 0, \dots, \chi$, 
\vspace{-.5cm}
\[
\Bigg\VERT\sum_{m = 1}^{N} \frac{\big\| Y^{(\ell)}_{m}  \big\| \1_{\mathscr{B}^{(\ell)}_{m}}  }{\omega_N+\hat{m}}  \1_{\chi_m \ge \ell}  \Bigg\VERT_2
\lesssim_R  \sum_{m = 1}^{N}  \frac{N^{1/4}}{(\omega_N + \hat{m})^{13/8} \omega_N^{1/8}}  
\lesssim_R \frac{N^{1/4}}{\omega_N^{3/4}}  = \frac{1}{\sqrt{\Omega\log N}} . 
\]
By \eqref{est10} and since the probability of the event $\mathscr{C} =  \mathscr{C}_{1}$ is given by \eqref{P:C}, this completes the proof of bound \eqref{Upsilon:est}. 
\end{proof}

This concludes the control of the difficult part of $\psi^{(>1)}_{N,1},$ and we turn to the easier part.

\begin{proposition} \label{prop:8}
  For any $R>0$,  there exists a constant $c_R>0$ and an event $\mathscr{G}_4$ with $\P\left[\mathscr{G}_4^c \right] \lesssim_R  N^{5-R\Omega}$ so that
    \[
    \big\VERT \psi_{N,1}^{(>1)} \left( \begin{smallmatrix} 0 & 0 \\ 0 &1\end{smallmatrix}  \right) \one_{\mathscr{G}_4} \big\VERT_2 \le c_R \Omega (\log N)   N^{-3/8} .
  \]
\end{proposition}
\begin{proof}
  We begin with recalling that by \eqref{psidecomp}, 
  \[
    \psi_{N,1}^{(>1)} \left( \begin{smallmatrix} 0 & 0 \\ 0 &1\end{smallmatrix} \right) 
    = \sum_{k=3}^{N} U_{N} \dots U_{k+1} \left( \begin{smallmatrix} 0 & 0 \\ 0 &1\end{smallmatrix} \right)  \eta_{k,21}  \psi^{(1)}_{k-1,1,12}. 
  \]
  We let $\mathscr{A} =  \left\{ \max_{1\le k <n \le N}\| U_n \cdots U_k\|  \le C_R'\right\}  \bigcap_{1\le p< n\le N} \A^3_{n,p}  $ as in Proposition \ref{prop:3} 
and by Proposition \ref{prop:4}, we have  $\P[\mathscr{A}^c] \lesssim_R   N^{5-R\Omega}$.
By Proposition \ref{prop:1} (using that $\hat{1} =N-1$), it holds conditionally on $\mathscr{A}$ for all $1\le p\le N$, 
  \[
    |\psi_{p,1,22}^{(0)}| \lesssim_R e^{-p/4} . 
  \]

  We also recall that for any $1 \leq p < k$, 
  \[
    \psi^{(1)}_{k,1,12}
    =
    \psi^{(1)}_{k,p+1,12}
    \psi^{(0)}_{p,1,22}
    +
    \psi^{(0)}_{k,p+1,11}
    \psi^{(1)}_{p,1,12}
    =
    \psi^{(1)}_{k,p+1,12}
    \psi^{(0)}_{p,1,22}
    +
    \psi^{(1)}_{p,1,12}
    ,
  \]
    by decomposing according to the location of the index of the single perturbing term (c.f.\ \eqref{eq:psiexpansion}).  
  Then, choosing $p = \lceil 4M \log N \rceil$ for a constant $M\ge 1$, we obtain
  \begin{equation}\label{eq:tinytiny}
    \big\VERT \max_{k=p+1,\dots, N}
    |\psi^{(1)}_{k,1,12}-\psi^{(1)}_{p,1,12}|
    \one_{\mathscr{A}}
    \big\VERT_2 
    =
    \big\VERT  |\psi^{(0)}_{p,1,22}|
    \max_{k=p+1,\dots, N}
    |\psi^{(1)}_{k,p+1,12}|
    \one_{\mathscr{A}}
    \big\VERT_2 
    \lesssim_R  N^{-3/8-M}
  \end{equation}
  where we have used that according to Proposition \ref{prop:2}, we have  for any $\ell \le p+1$,  
    \begin{equation}\label{eq:bonus}
    \big\VERT
     \max_{k=\ell+1,\dots, N} |\psi^{(1)}_{k,\ell,12}|
     \one_{\mathscr{A}}
    \big\VERT_2 \le C_R N^{-3/8} . 
  \end{equation}
  
This leads us to decompose  
  \begin{equation}\label{eq:tripod}
    %\sum_{k=2}^{N} U_{N} \dots U_{k+2} \left( \begin{smallmatrix} 0 & 0 \\ 0 &1\end{smallmatrix} \right)  \eta_{k+1,21}  \psi^{(1)}_{k,1,12}
     \psi_{N,1}^{(>1)} \left( \begin{smallmatrix} 0 & 0 \\ 0 &1\end{smallmatrix} \right) 
    =
    \left\{
      {\begin{aligned}\textstyle
	&\sum_{k=p+1}^{N} U_{N} \dots U_{k+1} \left( \begin{smallmatrix} 0 & 0 \\ 0 &1\end{smallmatrix} \right)  \eta_{k,21}  \psi^{(1)}_{p,1,12} \\
	&+\sum_{k=p+1}^{N} U_{N} \dots U_{k+1} \left( \begin{smallmatrix} 0 & 0 \\ 0 &1\end{smallmatrix} \right)  \eta_{k,21} 
        (\psi^{(1)}_{k-1,1,12}
	-\psi^{(1)}_{p,1,12}) \\
	&+\sum_{k=1}^{p} U_{N} \dots U_{k+1} \left( \begin{smallmatrix} 0 & 0 \\ 0 &1\end{smallmatrix} \right)  \eta_{k,21}  \psi^{(1)}_{k-1,1,12} \\
      \end{aligned}}
    \right\}.
  \end{equation}
  Let $(i), (ii), (iii)$ be the three lines in the brackets respectively.  First,  by \eqref{eq:tinytiny} and \eqref{recall:eta},  the second line is controlled by submultiplicativity, 
  \[
    \VERT (ii) \VERT_2=
    \big\VERT {\textstyle
      \sum_{k=p}^{N} U_{N} \dots U_{k+1} \left( \begin{smallmatrix} 0 & 0 \\ 0 &1\end{smallmatrix} \right)  \eta_{k,21} 
      (\psi^{(1)}_{k-1,1,12}
      -\psi^{(1)}_{p,1,12}) \one_{\mathscr{A}}
    }\big\VERT_2 \lesssim_R  N^{5/8-M} .
  \]
  Likewise, using \eqref{eq:bonus} for the third line
  \[
    \VERT (iii) \VERT_2=
    \big\VERT{\textstyle
    \sum_{k=1}^{p} U_{N} \dots U_{k+1} \left( \begin{smallmatrix} 0 & 0 \\ 0 &1\end{smallmatrix} \right)  \eta_{k,21}  \psi^{(1)}_{k-1,1,12}
  }\one_{\mathscr{A}}\big\VERT_2
    \lesssim_R  (\log N)  N^{-3/8}.
  \]
%  Both of these will be negligibly small in comparison to the first term.
  Notice that we have the Doob's decomposition $(i) =\big( M_{\hat{p}}+ A_{\hat{p}} \big) \psi^{(1)}_{p,1,12} $ where 
  \[
  M_{\hat{n}} = {\textstyle
      \sum_{k=n}^{N}} U_{N} \dots U_{k+1} \left( \begin{smallmatrix} 0 & 0 \\ 0 &1\end{smallmatrix} \right)  \big(\eta_{k,21}  -\E\eta_{k,21} \big)
 \]
  is a $\hat \filt_{\hat n}$--martingale with bounded increments (depending on $R$ and conditionally on the event $\A$) and  $A_{\hat{n}} =   \sum_{k=n}^{N} U_{N} \dots U_{k+1} \left( \begin{smallmatrix} 0 & 0 \\ 0 &1\end{smallmatrix} \right) \E\eta_{k,21}$  is a predictable process. 
 Then, we easily check that conditionally on $\mathscr{A}$, the quadratic variation of $M_{\hat{n}}$ is controlled uniformly by $ \sum_{k=1}^{N} \var[\eta_{k,21} ] \lesssim \log N$ and the predictable part is
 uniformly bounded by $\| A_{\hat{n}}\| \lesssim \sum_{k=1}^{N} |\E \eta_{k,21} |  \lesssim \log N$ (c.f. \eqref{recall:eta}). 
  Hence, by Theorem~\ref{thm:M} with $\Sigma^2 = C_R\Omega \log N$ for a sufficiently large constant $C_R>0$, 
  there exists an event $\mathscr{B} \subset \mathscr{A}$ such that 
  \[
  \VERT  M_{\hat{N}} \1_{\mathscr{B}}+ A_{\hat{N}} \VERT_2 \le \Sigma 
  \qquad\text{and}\qquad 
  \P[\mathscr{A}\setminus \mathscr{B}] \lesssim N^{-R\Omega} .
  \]
 By \eqref{eq:bonus}, this shows that
    \[
    \big\VERT (i)  \1_{\mathscr{B} \cap \big\{ |\psi^{(1)}_{p,1,12}|  \le C_R \Sigma   N^{-3/8}  \big\} } \big\VERT_2 \lesssim_R \Sigma^2  N^{-3/8} 
      \qquad\text{and}\qquad 
      \P\left[ |\psi^{(1)}_{p,1,12}|  \ge C_R \Sigma   N^{-3/8} \right]\lesssim N^{-R\Omega} .
    \]
   In all, if we set $\mathscr{G}_4= \mathscr{B} \cap \big\{ |\psi^{(1)}_{p,1,12}|  \le C_R \Sigma   N^{-3/8}  \big\}$, we conclude that
   $\P[\mathscr{G}_4^c] \lesssim \P[\mathscr{A}^c] \lesssim_R  N^{5-R\Omega}$ and
    \[
      \VERT ((i) + (ii) + (iii))\one_{\mathscr{G}_4} \VERT_2 \lesssim_R \Omega (\log N)  N^{-3/8} . 
    \]
    According to \eqref{eq:tripod}, this completes the proof.
\end{proof}

We are now ready to conclude the proof of our main result.  

\begin{proof}[Proof of Theorem \ref{thm:main}]
Combining Proposition~\ref{prop:7} and \ref{prop:8}, we have shown that for any $R>0$,   there exists a constant $c_R>0$ and an event $\mathscr{G}=\mathscr{G}_3 \cap \mathscr{G}_4$ with $\P\left[\mathscr{G}^c \right] \lesssim_R  N^{5-R\Omega}$ so that if $N$ is sufficiently large (depending on $\Omega$),
\[
  \big\VERT \psi_{N,1}^{(>1)}\1_{\mathscr{G}} \big\VERT_2 \le   \frac{c_R}{\sqrt{\Omega\log N}} . 
\]
Moreover as $\mathscr{G}\supset\A^2_{N,1}$,  by Proposition~\ref{prop:2}, we also have 
  \[
  \big\VERT \psi^{(1)}_{N,1} \1_{\mathscr{G}} \big\VERT_2  \le C_R   \frac{N^{1/4}}{\omega_N^{3/4}} = \frac{C_R}{\sqrt{\Omega\log N}} . 
\]
From the  perturbative expansion \eqref{eq:psiTaylor}, this implies that (after adjusting the constant $c_R$) if $N$ is sufficiently large, 
\begin{equation*}
\P\left[ \Big\{ \big\| {\textstyle \prod_{k=1}^ N U_k } - \psi^{(0)}_{N,1} \big\| \ge  \varepsilon \Big\}\cap \mathscr{G}
\right] \lesssim  N^{-c_R\varepsilon^2\Omega}  . 
\end{equation*}
Finally as $\mathscr{G}\supset\A^1_{N,1}$, according to Proposition~\ref{prop:1}, we also have that  on the event $\mathscr{G}$, 
\[ 
\big\|  \psi^{(0)}_{N,1} -  \left( \begin{smallmatrix} 1 & 0 \\ 0 &0\end{smallmatrix} \right)  \big\|
\le | \psi^{(0)}_{N,1, 22}| \le  C_R e^{-N/4}. 
\]
Altogether, this shows that if $ \varepsilon \ge N^{-R}$ and $N$ is sufficiently large, 
\begin{equation*}
\P\left[ \big\| {\textstyle \prod_{k=1}^ N U_k } -  \left( \begin{smallmatrix} 1 & 0 \\ 0 &0\end{smallmatrix} \right)\big\| \ge  \varepsilon
\right] \lesssim  N^{-c_R\varepsilon^2\Omega} + \P[\mathscr{G}^c] ,
\end{equation*} 
with $\P\left[\mathscr{G}^c \right] \lesssim_R N^{5-R\Omega}$. 
\end{proof}

\subsection{Proofs of Propositions~\ref{prop:5} and~\ref{prop:6}} 
\label{sect:extraproofs}

In this section, we give proofs of the two propositions that we used to control the moderate deviations of the martingale part in the decomposition of $\psi_{N,n}^{(>1)} \left( \begin{smallmatrix} 1 & 0 \\ 0 &0\end{smallmatrix} \right)$. 

\begin{proof}[Proof of Proposition~\ref{prop:5}]
Fix $R>0$ and let $\mathscr{C}_{n} = \left\{ \max_{n\le m<k \le N}\| U_m \cdots U_k\|  \le C_R' \right\} \bigcap_{n \le m<k\le N} \A^4_{k,m} $ where $C_R'$ is as in Proposition~\ref{prop:4} and  the events $\A^4_{k,m}$ are as in Corollary~\ref{cor:3}.
Then, the events $\mathscr{C}_{n}$ are increasing and  we have for all $1\le n  < N$, 
\begin{equation} \label{P:C}
\P[\mathscr{C}_n^c] \lesssim_R N^{5-R\Omega} . 
\end{equation}

 Let $\gamma_N = \lfloor \frac{\sqrt{N/\omega_N}}{CR} \rfloor$ with $C$ at least as  in Proposition~\ref{prop:3}  and let $e^{-1/4CR} <\alpha<1$.
 We may assume that $N$ is sufficiently large so that $\gamma_N \ge 2$.
  Since
$\psi^{(0)}_{k-1,n} = \psi^{(0)}_{k-1,n+j+1} \psi^{(0)}_{n+j,n}$ for any suitable integer $j\ge 0$, we have for all $1\le n < N$,
\[ \begin{aligned}
\psi^{(>0)}_{N,n} 
& = \sum_{k=n}^{N} U_N \dots U_{k+1} \left(\begin{smallmatrix}
    0 & \eta_{k,12} \\
    \eta_{k,21} & 0 
  \end{smallmatrix}\right)  \psi^{(0)}_{k-1,n} \\
 & =   \sum_{\ell= 0}^{\chi_n-1}   U_{N} \cdots U_{n+(\ell+1)\gamma_N+1} \psi^{(>0)}_{n+ (\ell+1) \gamma_N,  n+ \ell \gamma_N} \psi^{(0)}_{n+\ell \gamma_N,n} + \psi^{(>0)}_{N,n + \chi_n\gamma_N} \psi^{(0)}_{n+\chi_n \gamma_N,n}
\end{aligned}\]
where $\chi_n = \lfloor\frac{N-n}{\gamma_N} \rfloor$. 
Then,  if denote for any $0\le \ell<\chi_n$,
\begin{equation} \label{def:Y}
Y^{(\ell)}_{n} =\psi^{(>0)}_{n+(\ell+1) \gamma_N,  n +\ell  \gamma_N} \left( \begin{smallmatrix}  0 & 0 \\ 1 &0 \end{smallmatrix} \right)
\qquad\text{and} \qquad Y^{(\chi_n)}_{n} =  \psi^{(>0)}_{N,n + \chi_n\gamma_N} \left( \begin{smallmatrix}  0 & 0 \\ 1 &0 \end{smallmatrix} \right), 
\end{equation}
this shows  that 
\begin{equation} \label{eq:psialphadecomposition} \begin{aligned}
\psi^{(>0)}_{N,n}\left( \begin{smallmatrix}  0 & 0 \\ 1 &0 \end{smallmatrix} \right) 
& =  \sum_{\ell= 0}^{\chi_n-1} \alpha^\ell \left( \alpha^{-\ell}   U_{N} \cdots U_{n+(\ell+1)\gamma_N+1} Y^{(\ell)}_{n}  \psi^{(0)}_{n+\ell \gamma_N,n,22} \right) \\
&\qquad + \alpha^{\chi_n} \alpha^{-\chi_n}  Y^{(\chi_n)}_{n}\psi^{(0)}_{n+\chi_n \gamma_N,n,22} . 
\end{aligned}
\end{equation}

Using this decomposition, by applying Jensen's inequality, we obtain
\[ \begin{aligned}
\left\| \psi^{(>0)}_{N,n}\left( \begin{smallmatrix}  0 & 0 \\ 1 &0 \end{smallmatrix} \right)   \right\|^2     \le \frac{1}{1-\alpha} \Bigg( &\sum_{\ell= 0}^{\chi_n-1} \alpha^{-\ell}   \left\| U_{N} \cdots U_{n+(\ell+1)\gamma_N+1} Y^{(\ell)}_{n}  \right\|^2 \big| \psi^{(0)}_{n+\ell \gamma_N,n,22}\big|^2   \\
&\quad+  \alpha^{-\chi_n} \big\| Y^{(\chi_n)}_{n}  \big\|^2 \big| \psi^{(0)}_{n+\chi_n \gamma_N,n} \big|^2\Bigg) . 
\end{aligned}\]
Moreover, let us recall that by Proposition~\ref{prop:1}, 
it holds conditionally on the event $\mathscr{C}_n$, 
\begin{equation} \label{est9}
 \big| \psi^{(0)}_{n+\ell \gamma_N,n,22}\big|^2  \lesssim_R  \exp\left(- \frac{\ell \gamma_N}{2} \sqrt{\frac{\omega_N}{N}} \right) \le  C_R^2 e^{-\ell/ 4CR} .
\end{equation}
This implies that for all $1\le n < N$, 
 \begin{equation} \label{est6}
\left\| \psi^{(>0)}_{N,n} \left( \begin{smallmatrix}  0 & 0 \\ 1 &0 \end{smallmatrix} \right)   \right\|^2\1_{\mathscr{C}_n}
\lesssim_R \frac{1}{1-\alpha}  \sum_{\ell= 0}^{\chi_n} \widetilde{\alpha}^{\ell} \big\| Y^{(\ell)}_{n}  \big\|^2 \1_{\mathscr{B}^{(\ell)}_{n}} 
\end{equation}
where $\widetilde{\alpha} = e^{-1/4CR}/\alpha$ and we set
$\mathscr{B}^{(\ell)}_{n} = \A^4_{n+(\ell+1) \gamma_N,  n +\ell  \gamma_N}$ for $0\le \ell <\chi_n$
and $\mathscr{B}^{(\chi_n)}_{n} = \A^4_{N,n + \chi_n\gamma_N} $. 

\medskip

Now, observe that by formula \eqref{eq:Zinc}, we have
$\displaystyle Q =  \sum_{n = 1}^{N}  \var \big[\eta_{n,21}\big]\big\| \psi^{(>0)}_{N,n+1}  \left( \begin{smallmatrix}  0 & 0 \\ 1 &0 \end{smallmatrix} \right) \big\|^2$.
Then, by  \eqref{recall:eta} and \eqref{est6}, this implies that with  $\mathscr{C} =  \mathscr{C}_{1}$ and  adjusting the constant $C$,
\begin{equation} \label{Q:decomposition}
Q\1_{\mathscr{C}}\le C_R' \sum_{\ell= 0}^{\chi} \widetilde{\alpha}^{\ell}  Q_\ell
\qquad\text{with}\qquad
 Q_\ell = \sum_{n = 1}^{N} \frac{\big\| Y^{(\ell)}_{n}  \big\|^2  \1_{\mathscr{B}^{(\ell)}_{n}}  }{\omega_N+\hat{n}}  \1_{\chi_n \ge \ell} , \qquad \chi=\chi_{0} =  \lfloor\tfrac{N}{\gamma_N} \rfloor ,
\end{equation}
 and by assumption $0< \widetilde{\alpha} < 1$. 

By definition, $ Q_\ell $ are non--negative random variables. So, let us denote for $\ell=0,\dots, \chi$, 
\[
\widetilde{Q}_\ell =  Q_\ell -\E Q_\ell  
\qquad \text{and}\qquad X_n^{(\ell)} =  \frac{\big\| Y^{(\ell)}_{n}  \big\|^2  \1_{\mathscr{B}^{(\ell)}_{n}}  - \E\big[ \big\| Y^{(\ell)}_{n}  \big\|^2  \1_{\mathscr{B}^{(\ell)}_{n}}  \big] }{\omega_N+\hat{n}} \1_{\chi_n \ge \ell} 
\qquad\text{for}\quad n=1, \dots, N . 
\]

By construction, for any fixed integer $\ell \ge 0$, the dependency graph of the collection of random variables $\big(X^{(\ell)}_{n}\big)_{n=1}^{N}$ have degree bounded by $\gamma_N$ (c.f. \eqref{def:Y}). Hence by applying Theorem~\ref{thm:Svante} with $\mathcal{J}  =\{1, \dots , N \}$ and  $\mathcal{J}_k = \{ k+ j\gamma_N : j\ge 0 \} \cap \mathcal{J}$ for $k=1, \dots, \gamma_N$,    there exist some events $\mathscr{E}_\ell$ such that the following hold
\[
 \max_{\ell=0,\dots, \chi}\VERT  \widetilde{Q}_\ell  \1_{\mathscr{E}_\ell}\VERT_2 \lesssim  \gamma_N \sigma 
\qquad\text{and}\qquad
 \max_{\ell=0,\dots, \chi}\P[\mathscr{E}_\ell^c] \leq 2\gamma_N e^{-(\sigma/b)^2/C}
\]
where 
\vspace{-.5cm}
\[
b = \max_{\ell=0,\dots, \chi} \max_{n=1,\dots,N } \left\{ \frac{\big\VERT Y^{(\ell)}_{n} \1_{\mathscr{B}^{(\ell)}_{n}} \big\VERT_2^2}{\omega_N+\hat{n} } \right\}
\qquad\text{and}\qquad
\sigma^2 =  \max_{\ell=0,\dots, \chi} \max_{k=1,\dots, \gamma_N}\left\{ \sum_{n \in \mathcal{J}_k }
\frac{\big\VERT Y^{(\ell)}_{n}   \1_{\mathscr{B}^{(\ell)}_{n}}\big\VERT_2^4}{(\omega_N+\hat{n})^2} \right\} .
\]
Note that we used \eqref{eq:Xcentering} as well as the fact that for any random matrix $X$, we have $\VERT \|X\|^2 \VERT_1 \le \VERT X\VERT_2^2$. 
Using the estimate \eqref{est:psi>0}, we obtain for all $n=1,\dots, N$ and uniformly  for all $\ell= 0, \dots, \chi$, 
\begin{equation} \label{Y:est}
\frac{\big\VERT Y^{(\ell)}_{n} \1_{\mathscr{B}^{(\ell)}_{n}} \big\VERT_2^2}{\omega_N+\hat{n} }
\lesssim_R \frac{\sqrt{N}}{(\omega_N + \hat{n})^{9/4} \omega_N^{1/4}}   . 
\end{equation}
We have used above that the block size $\gamma_N \le \frac{\sqrt{N/\omega_N}}{CR}$. 
%$\frac{\gamma_N}{\omega_N} \lesssim \frac{\sqrt{N}}{\omega_N^{3/2}}  = \frac{1}{\Omega\log N}$. 
This shows that we can choose
\[
b^2 = C_R \frac{N}{\omega_N^{5}}
\qquad\text{and}\qquad
\sigma^2  =  C_R \frac{\gamma_N^{-1} N}{\omega_N^4} \ge  \gamma_N^{-1} \sum_{n=1}^{N} \frac{N}{(\omega_N+\hat{n})^{9/2}\sqrt{\omega_N}}  
\]
for a suitably large constant $C_R>0$.  This implies that
\begin{equation} \label{Q:est}
 \max_{\ell=0,\dots, \chi}\VERT   \widetilde{Q}_\ell  \1_{\mathscr{E}_\ell}\VERT_2^2 \lesssim_R  \frac{\gamma_N N}{\omega_N^4}  \lesssim_R \frac{N^{3/2}}{\omega_N^{9/2}} = \frac{1}{(\Omega\log N)^{3}}
 \end{equation}
and $\displaystyle \max_{\ell=0,\dots, \chi}\P[\mathscr{E}_\ell^c] \leq 2\gamma_N e^{- \frac{\omega_N}{\gamma_N}} \le  2\gamma_N  e^{- R\frac{\omega_N^{3/2}}{\sqrt{N}}}$. 

The previous estimate combined with \eqref{P:C} implies that  if $\mathscr{G}_1 := \mathscr{C}_1  \bigcap_{\ell=0}^\chi\mathscr{E}_\ell$, then 
\begin{equation} \label{P:E}
\P\left[\mathscr{G}_1^c \right] \lesssim_R \gamma_N \chi  e^{-R\Omega \log N}  + N^{5-R\Omega} \lesssim_R N^{5-R\Omega}  , 
\end{equation}
where we used that $\gamma_N \chi  \le N+1$. 
Recall that by \eqref{Q:decomposition}, we have
\[
Q \1_{\mathscr{G}_1 } \lesssim_R  \sum_{\ell= 0}^{\chi} \widetilde{\alpha}^{\ell}  \widetilde{Q}_\ell \1_{\mathscr{E}_\ell} +  \sum_{\ell= 0}^{\chi}  \widetilde{\alpha}^{\ell} \E[Q_\ell ] . 
\]
Moreover, using \eqref{eq:Xmoment} and  the estimate \eqref{Y:est},   we obtain
\[
 \E[Q_\ell ]  
 \lesssim  \sum_{n = 1}^{N} \frac{\big\VERT\big\| Y^{(\ell)}_{n}  \big\|^2  \1_{\mathscr{B}^{(\ell)}_{n}} \big\VERT_1}{\omega_N+\hat{n}} \1_{\chi_n \ge \ell}
  \lesssim_R  \sum_{n = 1}^{N} \frac{\sqrt{N}}{(\omega_N + \hat{n})^{9/4} \omega_N^{1/4}}  
  \lesssim \frac{\sqrt{N}}{\omega_N^{3/2}} . 
\] 
Since $\widetilde{\alpha} <1$, this shows that   $\sum_{\ell= 0}^{\chi}  \widetilde{\alpha}^{\ell}\E[Q_\ell ]   \lesssim_R \frac{(1-\widetilde{\alpha} )^{-1}}{\Omega \log N}$.  Hence, by \eqref{Q:est}, we conclude that for a suitable constant $c_R>0$, 
\[
\big\VERT \big( Q - \tfrac{c_R}{\Omega \log N} \big)_+  \1_{\mathscr{G}_1 } \big\VERT_2^2
\lesssim  \frac{1}{1-\widetilde{\alpha}} \sum_{\ell= 0}^{\chi} \widetilde{\alpha}^{\ell} \big\VERT \widetilde{Q}_\ell \1_{\mathscr{E}_\ell} \big\VERT_2^2 \le  \frac{c_R}{(\Omega\log N)^{3}}
\]
where we used Jensen's inequality (as $0<\widetilde{\alpha}<1$).
By \eqref{P:E}, this completes the proof.
\end{proof}

\begin{proof}[Proof of Proposition~\ref{prop:6}] We relie on the notation and estimates from Proposition~\ref{prop:5}. 
We also need to introduce the following stopping time: 
\[ \begin{aligned}
\widehat{T} : = \inf \Bigg\{ \hat{n} < N : &\left\{  \max_{n\le m<k \le N}\| U_m \cdots U_k\|  > C_R' \right\} \bigcup_{n \le m<k\le N}  \A^{4,c}_{k,m} \bigcup \Bigg\{ \max_{\ell=0,\dots, \chi_n}\big\| Y^{(\ell)}_{n} \big\|^2 \1_{\mathscr{B}^{(\ell)}_{n}} > \tfrac{ C_R \Omega \log N }{\omega_N^{3/2}/\sqrt{N} }  \Bigg\}  \Bigg\} 
\end{aligned}\]
with $ C_R'$ as in Proposition~\ref{prop:4} and  a constant $C_R$ to be chosen sufficiently large. 
By \eqref{def:Y},  the random variables $\big( Y^{(\ell)}_{n}\big)_{\ell=0}^{\chi_n}$ are $\hat \filt_{\hat n}$ measurable, so that $\widehat{T}$ is indeed a stopping time. 
We let $Z_{\hat{n},0}^{\widehat{T}}= Z_{\widehat{T}\wedge \hat{n},0}$ for $\hat{n} =0, \dots, N-1$. 
First let us observe that by \eqref{eq:Zinc} and \eqref{recall:eta}, it holds  
\[
\big\| Z_{\widehat{n+1},0}^{\widehat{T} }- Z_{\hat{n},0}^{\widehat{T} } \big\|^2  \lesssim  \frac{R\Omega\log N}{\omega_N}    \left\|\psi^{(>0)}_{N,n} \left( \begin{smallmatrix}  0 & 0 \\ 1 &0 \end{smallmatrix} \right) \right\|^2 \1_{\widehat{T}>\hat{n}} . 
\]
Since $\big\{\widehat{T}>\hat{n}\big\} \subset \mathscr{C}_n$, c.f. above \eqref{P:C}, using the decomposition \eqref{est6}, we obtain
\[
\big\| Z_{\widehat{n+1},0}^{\widehat{T} }- Z_{\hat{n},0}^{\widehat{T} } \big\|^2  \lesssim_R \frac{\Omega\log N}{\omega_N} 
\Bigg( \sum_{\ell= 0}^{\chi_n} \widetilde{\alpha}^{\ell} \big\| Y^{(\ell)}_{n}  \big\|^2 \1_{\mathscr{B}^{(\ell)}_{n}}  \Bigg)  \1_{\widehat{T}>\hat{n}}  .
\]
Then,  since $\big\{\widehat{T}>\hat{n}\big\} \subset \Bigg\{ {\displaystyle\max_{\ell=0,\dots, \chi_n}}\| Y^{(\ell)}_{n} \|^2 \1_{\mathscr{B}^{(\ell)}_{n}}  \le \frac{C_R \Omega \log N \sqrt{N} }{\omega_N^{3/2}}  \Bigg\}$ and $\widetilde{\alpha}<1$,  this implies that for all $\hat{n} =0, \dots, N-1$, 
\begin{equation*} 
\big\| Z_{\widehat{n+1},0}^{\widehat{T} }- Z_{\hat{n},0}^{\widehat{T} } \big\|^2   \lesssim_R 
 \frac{(\Omega\log N)^2\sqrt{N} }{\omega_N^{5/2}} . 
\end{equation*}

\medskip

Since the quadratic variation of the martingale  $\big(Z^{\widehat{T}}_{\hat{n},0}\big)_{\hat{n}=0}^{N}$ is bounded by $Q$, by Proposition~\ref{prop:5}, if we  apply Theorem~\ref{thm:M} with 
 $\Sigma^2 =  \frac{2c_R}{\Omega \log N}$ and $\alpha^2 =  c_R\frac{(\Omega\log N)^2\sqrt{N} }{\omega_N^{5/2}}$ (after increasing the constant $c_R>0$ from Proposition~\ref{prop:5} if necessary), we obtain that there exists an event $\A$ such that
\begin{equation} \label{Z:est}
  \big\VERT \sup_{\hat{n}=0,\dots, N-1}\big\|Z^{\widehat{T}}_{\hat{n},0}\big\| \1_{\A} \big\VERT_2 \lesssim   \sqrt{\tfrac{c_R}{\Omega\log N}}
\end{equation}
and by a union bound, $\P[\A^c]  \le 4 e^{-(\frac{\Sigma}{\alpha})^2} + \P\left[ Q > \Sigma^2 \right] $.
With our conventions, we have 
$\frac{\Sigma^2}{\alpha^2} =\frac{2\omega_N^{5/2}}{(\Omega\log N)^3\sqrt{N}}= \frac{2N^{1/3}}{(\Omega\log N)^{4/3}} \le R\Omega \log N $ if $N$ is sufficiently large (depending on $\Omega$).  
Moreover with the event $ \mathscr{G}_1$ from Proposition~\ref{prop:5} and if $c_R\ge R$, it holds  by \eqref{eq:Xtail},
\[
\P\left[ Q > \Sigma^2 , \mathscr{G}_1\right]   = \P\left[ (Q - \Sigma^2/2)\1_{\mathscr{G}_1}  > \Sigma^2/2 \right]  \le 2 \exp\left( - \tfrac{\Sigma^4 (\Omega\log N)^3}{4c_R}\right)
\le 2 \exp\left( - R \Omega \log N \right) . 
\]
Combining these estimates with the fact that   $\P\left[\mathscr{G}_1^c \right] \lesssim_R N^{5-R\Omega}  $, this implies that if $N$ is sufficiently large,
\begin{equation} \label{est8}
\P[\A^c] \lesssim_R N^{5-R\Omega}  . 
\end{equation}

By \eqref{Z:est}, in order to complete the proof, it suffices to show that $\widehat{T}  \ge N$ with overwhelming probability. By definition, we have
\[ \begin{aligned}
\P\left[ \widehat{T} \ge N \right] = \P\Bigg[ & \max_{1 \le m<k \le N}\| U_m \cdots U_k\|  \le C_R' , \bigcap_{n \le m<k\le N}  \A^{4}_{k,m} , \max_{n=1,\dots,N }\max_{\ell=0,\dots, \chi_n}\big\| Y^{(\ell)}_{n} \big\|^2 \1_{\mathscr{B}^{(\ell)}_{n}} \le \tfrac{ C_R \Omega \log N }{\omega_N^{3/2}/\sqrt{N} }  \ \bigg]
\end{aligned}\]
Since  $\big\VERT\big\| Y^{(\ell)}_{n} \big\|^2 \1_{\mathscr{B}^{(\ell)}_{n}} \big\VERT_1 \le \big\VERT Y^{(\ell)}_{n} \1_{\mathscr{B}^{(\ell)}_{n}} \big\VERT_2^2$, by using the estimate \eqref{Y:est} and  \eqref{eq:Xtail}, we obtain that if $C_R$ is sufficiently large, then 
\[
\P\left[ \big\| Y^{(\ell)}_{n} \big\|^2 \1_{\mathscr{B}^{(\ell)}_{n}} > \tfrac{ C_R \Omega \log N }{\omega_N^{3/2}/\sqrt{N} }  \  \right] \le  2 N^{-R \Omega} . 
\]
Hence, by Proposition~\ref{prop:4} and Corollary~\ref{cor:3}, we obtain 
\[
 \P\left[ \widehat{T}<  N \right]\lesssim_R N^{5-R\Omega} . 
\]
Thus, by \eqref{Z:est}, we  conclude that if we let $\mathscr{G}_2 = 
\A \cap \left\{ \widehat{T} \ge N \right\}$, then 
\[
  \big\VERT \sup_{\hat{n}=0,\dots, N-1}\big\|Z_{\hat{n},0}\big\| \1_{\mathscr{G}_2} \big\VERT_2 \lesssim_R  \tfrac{1}{\sqrt{\Omega\log N}}
\]
and by \eqref{est8}, $\P[\mathscr{G}_2^c] \lesssim_R N^{5-R\Omega}$. 
\end{proof}

\section{Coupling with a Gaussian log--correlated field} \label{sect:coupling}

\subsection{Proof of Proposition~\ref{prop:coupling1}} \label{sect:coupling1}

The argument is divided in two steps. First, we use again Corollary \ref{cor:transfer} to compute the \emph{Hermite contribution} coming from  $\biggl[ \prod_{k=1}^{N}  \lambda_+(\tfrac{k}N)\big(1- \delta_k \big) \biggr]$.   Then, we obtain an approximation between sums involving $\{\eta_{k,11}\}_{k=1}^N$ and integrals against the appropriate Brownian motions.
We let $\alpha=1/9$ and $\delta=1/45$ as in Corollary~\ref{cor:transfer}.

\medskip

\noindent\underline{1- Contribution from the mean:}\
Let us observe that if there is no noise (i.e. when $\left\{ (X_k,Y_k) \right\}_{k=1}^N= 0$), by Lemma~\ref{lem:diag} and formula \eqref{eq:hermite}, we have for any $n \ge 2$,
\begin{equation} \label{detrec}
  \begin{pmatrix}
    \pi_n(z) \\
    \pi_{n-1}(z) 
  \end{pmatrix}=
\prod_{k=2}^{n}  \lambda_+(\tfrac{k-1}N)\big(1- \delta_k \big)   V_{n+1}\prod_{k=2}^n \widetilde{U_k} V_2^{-1}   \begin{pmatrix}
  z \\ 1
  \end{pmatrix} ,
\end{equation}
where $\widetilde{U_k} = \E U_k$ for $k\in\{1,\dots, N\}$ and  $\{\pi_n\}$ are the monic Hermite polynomials scaled to be orthogonal with respect to the weight $e^{-2N x^2}$ on $\R$.
Moreover, we can also apply the estimate from Corollary~\ref{cor:transfer} in the case where there is no noise, this implies that uniformly for all $z\in\mathscr{P}$,
\begin{equation*} 
  \begin{pmatrix}
    \pi_N(z) \\
    \pi_{N-1}(z) 
  \end{pmatrix}=
\prod_{k=2}^{N}  \lambda_+(\tfrac{k-1}N)\big(1- \delta_k \big)   V_{N+1}
\left[\begin{pmatrix} 1 & 0 \\ 0 & 0 \end{pmatrix}  + \O\big(N^{-\frac{1}{15}}\big) \right] V_2^{-1} \begin{pmatrix}   z  \\ 1\end{pmatrix} . 
\end{equation*}
Now using the estimates \eqref{bc} with $b_1 = 0$ for $z \in \mathscr{P},$ 
\[
 V_2^{-1} \begin{pmatrix}   z  \\ 1\end{pmatrix} = 
 \begin{pmatrix} 1 \\ 0 \end{pmatrix} + \O(N^{2\alpha-1}).
\]
Thus we obtain that for any compact set $K\subset\C$, it holds uniformly for all $z\in K \cap\mathscr{P}$,  
\begin{equation} \label{charpoly2}
    \pi_N(z)  =  \prod_{k=1}^{N}  \lambda_+(\tfrac{k}N)\big(1- \delta_k \big)  \big( 1+ \O\big(N^{-\frac{1}{15}}\big) \big) . 
\end{equation}

\noindent\underline{2- Truncation and linearization:}\
Let us work under the probability measure $\P_\mathrm{S}$ from Definition~\ref{def:cond} with $\mathrm{S}=N^\delta$. 
Under this measure, the random variables $\{X_k\} $ and  $\{Y_k\}$ from \eqref{def:XY} are bounded by $N^{3\delta/2}.$ 
Recall that $\sup_{t\in[0,1]}|z^2-t|^{-1} \le N^{2\alpha}$ for $z\in\mathscr{P}$ (see \eqref{estimate1}).
By \eqref{eq:delta}, this implies that $|\delta_k| \le \tfrac14 N^{-1+2\alpha}$ for $k\in\{1,\dots, N\}$. 
As in  \eqref{def:eta1}, we let 
\begin{equation} \label{etatilde} 
{\eta}_{k,11} =    \sqrt{\frac{1/2\beta}{Nz^2-k}}\left( {X}_k +  {Y}_k J\big(z \sqrt{\tfrac{N}{k-1}}\big) \right). 
\end{equation}
Hence $|\eta_{k,11}| \lesssim N^{-1/2+\alpha+3\delta/2}.$
First, if $N$ is sufficiently large, we can assume that $|\eta_{k,11}| \le \tfrac14$ for $k\in\{1,\dots, N\}$, in which case
\[
 \sum_{k=1}^{N} \Bigg| \log\bigg(1- \frac{\eta_{k,11}}{1-\delta_k} \bigg)  - \log\big(1-  {\eta}_{k,11} \big) \Bigg|
  \lesssim   \sum_{k=1}^{N} \big|\delta_k\eta_{k,11}\big| 
  = \O\left(N^{-\frac 12+3\alpha+ \frac{3\delta}{2}} \right). 
\]
Second, we check that by expanding $\log(\cdot)$ that
\[
\bigg| \sum_{k=1}^{N} \log\big(1- {\eta}_{k,11} \big) + \sum_{k=1}^{N} \bigg( {\eta}_{k,11} + \frac12  {\eta}_{k,11}^2 \bigg) \bigg|  \lesssim
\sum_{k = 1}^N \left|{\eta}_{k,11} \right|^3  \lesssim N^{-\frac 12+3\alpha+\frac{9\delta}{2}} . 
\]
With $\alpha=1/9$ and $\delta=1/45$, we obtain
\begin{equation*} 
 \prod_{k=1}^{N}  \lambda_+(\tfrac{k}N)\big(1- \delta_k - \eta_{k,11} \big)= \prod_{k=1}^{N}  \lambda_+(\tfrac{k}N)\big(1- \delta_k \big)  \exp\left( - \sum_{k=1}^{N} \bigg( {\eta}_{k,11} + \frac12  {\eta}_{k,11}^2 \bigg)+ \O\big(N^{-\frac{1}{15}}\big) \right) . 
\end{equation*}
Hence, by combining these asymptotics with \eqref{charpoly2}, we obtain $\P_{N^\delta}$--almost surely,\begin{equation} \label{charpoly3}
 \prod_{k=1}^{N}  \lambda_+(\tfrac{k}N)\big(1- \delta_k - \eta_{k,11} \big)=    \pi_N(z)  \exp\left( - \sum_{k=1}^{N} \bigg( {\eta}_{k,11} + \frac12  {\eta}_{k,11}^2 \bigg) \right)   \left( 1+ \O\big(N^{-\frac{1}{15}}\big) \right) , 
\end{equation}
 uniformly for all $z \in K\cap  \mathscr{P}$. 

\medskip

\noindent\underline{3- Comparing to stochastic integrals:}\
Recall that we have used the coupling from
Theorem \ref{thm:shaoKMT}.  Hence we have
\begin{equation} \label{coupling} 
    \Big\VERT \max_{1 \leq n \leq N} \biggl| {\textstyle \frac{1}{\sqrt{N}}\sum_{j=1}^n {X}_j} - \mathbf{X}_{\frac nN} \biggr| \Big\VERT_1 \vee 
    \Big\VERT \max_{1 \leq n \leq N} \biggl| {\textstyle \frac{1}{\sqrt{N}}\sum_{j=1}^n {Y}_j} - \mathbf{Y}_{\frac nN} \biggr| \Big\VERT_1 \lesssim \frac{\log N}{\sqrt{N}}.
\end{equation}

By Proposition \ref{prop:integralapprox}, \eqref{coupling} and \eqref{etatilde},  we verify that for any $\delta>0$, it holds with probability at least $1-e^{-cN^{\delta}}$, 
\begin{equation} \label{fluctu1}
  \left| \sum_{k=1}^{N} {\eta}_{k,11} - \frac{1}{\sqrt{2\beta}} \int_0^1 \frac{\d \mathbf{X}_u+ J(z/\sqrt{u}) \d\mathbf{Y}_u}{\sqrt{z^2-u}} \right| \ \lesssim N^{-1/2} \max_{f=f_1, f_2} \left\{ \|f\|_{\operatorname{TV},1} N^{2\delta}  + Y_N(f) \right\}
\end{equation}
where $f_1(u) = \frac{1}{\sqrt{z^2-u}}$, $f_2(u)= \frac{J(z/\sqrt{u})}{\sqrt{z^2-u}}$,  for $u\in(0,1]$.
We already used that  $\sup_{u\in[0,1]}\|f_j\|_{\infty} \lesssim N^{\alpha}$  for $j=1,2$ and we verify by Lemma  \ref{lem:J} and \eqref{J:logderivative} that for any $u\in (0,1]$,
\[
|f_1'(u)|+|f_2'(u)| \lesssim \frac{1}{|z^2-u|^{3/2}}+  \frac{|zJ(z/\sqrt{u})|}{u|z^2-u|} \lesssim \frac{1}{|z^2-u|^{3/2}}+ \frac{1/\sqrt{u}}{|z^2-u|} . 
\] 
This gives the bounds: $\|f_j\|_{\operatorname{TV},1} \lesssim N^{\alpha}$ and $\VERT Y_N(f_j)\VERT_1 \lesssim  N^{\alpha}$  for $j=1,2$. 
By \eqref{fluctu1}, these estimates show that with probability at least $1-e^{-cN^{\delta}}$, it holds  uniformly for all $z \in\mathscr{P}$,
\begin{equation} \label{fluctu2}
  \left| \sum_{k=1}^{N} {\eta}_{k,11} - \frac{1}{\sqrt{2\beta}} \int_0^1 \frac{\d \mathbf{X}_u+ J(z/\sqrt{u}) \d\mathbf{Y}_u}{\sqrt{z^2-u}} \right| \ \lesssim N^{-1/3} . 
\end{equation}

Moreover, we have $\P_{S}$--almost surely for $k\ge 1$,
\[
\VERT  {\eta}_{k,11}^2 \VERT_{1} \le \VERT  {\eta}_{k,11} \VERT_{2}^2 \lesssim N^{-1+2\alpha} \big( \VERT {X}_k \VERT_{2}^2  + \VERT {Y}_k \VERT_{2}^2  \big) \lesssim N^{-1+2\alpha} , 
\] 
 so that $ \sum_{k = 1}^N\big\VERT  {\eta}_{k,11}^2 \big\VERT_{1}^2  \lesssim N^{-1+4\alpha}$. 
 By Bernstein's inequality \eqref{eq:actualbernstein} with the measure $\Pr_S$ this shows that with probability at least $1-e^{-cN^{\delta}}$, 
\begin{equation} \label{fluctu3}
  \left|  \sum_{k =  1}^N \big( {\eta}_{k,11}^2 - \E_S {\eta}_{k,11}^2 \big)  \right| \leq N^{-1/2+\alpha+\delta}  . 
%\prec   \frac{\log N}{\sqrt{\omega_N}} 
\end{equation}

As for the means, we have for $k\ge 1$ $\E_S {\eta}_{k,11}^2 = \E {\eta}_{k,11}^2 + \O(e^{-cN^{-\delta}})$ and
\begin{equation} \label{meaneta}
\E {\eta}_{k,11}^2 
=  \frac{1/2\beta}{Nz^2-k}\left(  \E {X}_k^2 +  J\big(z \sqrt{\tfrac{N}{k-1}}\big)^2 \E {Y}_k^2 \right)
=   \frac{1}{2\beta N} \frac{ 1 +  J\big(z \sqrt{\tfrac{N}{k-1}}\big)^2 }{z^2-k/N} ,
\end{equation}
so  that by a Riemann sum approximation
\begin{equation} \label{fluctu4}
 \begin{aligned}
\sum_{k=1}^N \E {\eta}_{k,11}^2  & = \frac{1}{2\beta} \int_0^1 \frac{1+ J(z/\sqrt{u})^2}{z^2-u} du  
+\O\left( \sum_{k=1}^N  \frac{\sqrt{N/k}}{|Nz^2-k|^2} \right) \\
& = \E\left[ \left( \frac{1}{\sqrt{2\beta}}  \int_0^1 \frac{\d \mathbf{X}_u+ J(z/\sqrt{u}) \d\mathbf{Y}_u}{\sqrt{z^2-u}}\right)^2 \right] + \O\big(N^{-1+4\alpha}\big) . 
\end{aligned}
\end{equation} 
Here we have used that the estimate \eqref{estimate1} and that if $f(u) = \frac{1+ J(z/\sqrt{u})^2}{z^2-u}$, then by Lemma  \ref{lem:J} and \eqref{J:logderivative}, $|f'(u)| \lesssim  \frac{1/\sqrt{u}}{|z^2-u|^2} $  for $u\in[0,1]$ uniformly for all $z \in K\cap  \mathscr{P}$. 

\medskip

By combining the estimates \eqref{fluctu2}, \eqref{fluctu3} and \eqref{fluctu4}, we conclude that there exists an event $\A \subset \mathscr{T}_{N^\delta}$, such that $\P[\A^{c}] \le 2e^{-N^\delta}$ such that on  $\A$, it holds uniformly for all $z \in K\cap  \mathscr{P}$, 
\[
\sum_{k=1}^{N} \left(  {\eta}_{k,11} + \frac 12 {\eta}_{k,11}^2  \right)
 = \sqrt{\frac{2}{{\beta}}}  \mathrm{W}(z) + \frac{1}{\beta} \E\left[ \mathrm{W}(z)^2\right] + \O\big( N^{-1/3} \big) ,
\]
where we recall from \eqref{eq:Wtz}
\[
  \mathrm{W}(z) = \frac{1}{{2}} \int_0^1 \frac{\d \mathbf{X}_u+ J(z/\sqrt{u}) \d\mathbf{Y}_u}{\sqrt{z^2-u}}   .
\]
By combining these asymptotics with \eqref{charpoly3}, this completes the proof.
\hfill$\square$

\subsection{Proof of Proposition \ref{prop:coupling2}}  \label{sect:coupling2}

The argument is very similar to that of Section~\ref{sect:coupling1}.
We work under the probability measure $\Pr_S$ with $S={N^\epsilon}$ from Definition~\ref{def:cond} for a small $\epsilon>0$ to be chosen later.

\medskip

We fix a point  $z\in \mathscr{D}_H$ (see Definition~\ref{def:hyper}) and we let 
$\chi= |\Re z|^2- \Omega N^{-2/3}|\Re z|^{2/3}$ (in particular, we note that for all $t\in[0,\chi]$, we  have $\lfloor Nt  \rfloor \le N_H$). 
Like in Section~\ref{sect:coupling1}-3, we claim that if $\epsilon>0$ is sufficiently small compared to $\delta>0$,  then it holds $\P_{N^\epsilon}$--almost surely,  for all $n \in \{1,\dots, N_H\}$, 
\begin{equation} \label{truncation2}
\bigg| \sum_{k=1}^{n} \log\bigg(1- \frac{\eta_{k,11}}{1-\delta_k} \bigg) + \sum_{k=1}^{n} \bigg(  {\eta}_{k,11} + \frac12  {\eta}_{k,11}^2 \bigg) \bigg|  \lesssim N^{-\epsilon} . 
\end{equation}
This estimate follows from the fact that  $|Nz^2 - k|  \ge  \frac{\omega_N + \hat{k}}{\sqrt{2}}$  for all $k\in\{1,\dots, N_H\}$ (see Proposition~\ref{prop:Hyper}). Then, according to \eqref{eq:delta} and \eqref{def:eta1},  it holds $\P_{N^\epsilon}$--almost surely,  
 \begin{equation*}
%  \sum_{k\le \lfloor \beta^{-1} N^\epsilon \rfloor} |\eta_{k,11}|    \lesssim   \sum_{k\le \lceil \beta^{-1} N^\epsilon \rceil} \frac{N^{3\epsilon/2}}{|N z^2-k|^{1/2}}   \lesssim_\beta\, N^{5\epsilon/2} \omega_N^{-1/2} \, ,
%  \qquad
    \sum_{k=1}^{N_H} \big|\delta_k\eta_{k,11}\big|  
    %\lesssim \sum_{k = 1}^N  \frac{N^{\epsilon/2}}{|N z^2-k|^{3/2}} 
    \lesssim N^{\epsilon/2} \omega_N^{-1/2}
    \qquad\text{ and }\qquad
    \sum_{k = 1}^{N_H} \left|{\eta}_{k,11} \right|^3  \lesssim   \sum_{k = 1}^{N_H}  \frac{N^{3\epsilon/2}}{|N z^2-k|^{3/2}}   \lesssim_\beta\, N^{3\epsilon/2} \omega_N^{-1/2}. 
 \end{equation*}
Since $\omega_N \ge  \Omega  N^{2\delta/3}$, the estimate \eqref{truncation2} implies that  for all $n \in \{1,\dots, N_H\}$, 
\begin{equation*} 
 \prod_{k=1}^{n}  \lambda_+(\tfrac{k}N)\big(1- \delta_k - \eta_{k,11} \big)= \prod_{k=1}^{n}  \lambda_+(\tfrac{k}N)\big(1- \delta_k \big)  \exp\left( -  \sum_{k=1}^{n} \bigg(  {\eta}_{k,11}  +  \frac 12  {\eta}_{k,11}^2 \bigg)   + \O\big( N^{-\epsilon}\big) \right) . 
\end{equation*}
Hence by \eqref{charpoly5}, this shows that uniformly for all $z\in \mathscr{D}_H$ and $n \in [N^\delta, N_H]$,  
\begin{equation} \label{charpol6}
\biggl[ \prod_{k=2}^{n}  \lambda_+(\tfrac{k-1}N)\big(1- \delta_k - \eta_{k,11} \big) \biggr]  \begin{pmatrix}  \lambda_+(\frac{n}{N}) \\ 1   \end{pmatrix} 
=  \exp\left( -  \sum_{k=1}^{n} \bigg(  {\eta}_{k,11}  +  \frac 12  {\eta}_{k,11}^2 \bigg)\right)
\begin{pmatrix} \pi_{n} \\  \pi_{n-1}  \end{pmatrix}   \left(1+  \O\left(N^{-\epsilon} \right) \right) .
\end{equation}

Now, by \eqref{coupling}, \eqref{etatilde} and  Proposition \ref{prop:integralapprox}, for any $n\le N_H$,  there exists an event $\A_n$ which is independent of $\F_{>n}$, such that $\P[\A_n^c] \le e^{-N^{\epsilon}}$ and  
\begin{equation} \label{fluc0}
  \left| \sum_{k=1}^{n} {\eta}_{k,11} - \frac{1}{\sqrt{2\beta}} \int_0^t \frac{\d \mathrm{X}_u+ J(z/\sqrt{u}) \d\mathrm{Y}_u}{\sqrt{z^2-u}} \right| \lesssim
 N^{-1/2} \max_{f=f_1, f_2} \left\{ \|f\|_{\operatorname{TV},\chi} N^{3\epsilon/2}  + Y_n(f) \right\}
\end{equation}
where $t=n/N$ and $f_1(u) = \frac{1}{\sqrt{z^2-u}}$, $f_2(u)= \frac{J(z/\sqrt{u})}{\sqrt{z^2-u}}$,  for $u\in[0,\chi]$.
To control the errors, observe that for $j=1,2,$
\[
\|f_j\|_{\infty} \lesssim \sup_{u\in[0,\chi]}\left( \tfrac{1}{\sqrt{|\Re z|^2-u}} \right)  \lesssim_\Omega\, N^{1/3} |\Re z|^{-1/3}
\le N^{1/2 - \delta/3} ,
\]
where we have used that for any $z\in  \mathscr{D}_H$,
$|z^2-u| \ge \frac{|\Re z|^2-u}{\sqrt 2}$ for $u\in[0,\chi]$ (see the proof of Proposition~\ref{prop:Hyper}) and $|\Re z| \ge N^{\delta - \frac{1}{2}}$.  
Similarly, we have 
\[
\int_0^\chi |f_1'(u)| du  \le  \int_0^\chi \frac{d u}{(|\Re z|^2-u)^{3/2}}  \lesssim_\Omega\, N^{1/3} |\Re z|^{-1/3}
\le N^{1/2 - \delta/3} 
\]
and by \eqref{J:logderivative}, 
\[
\int_0^\chi |f_2'(u)| du  \le \int_0^\chi |f_1'(u)| du  + |z|  \int_0^\chi \frac{|J(z/\sqrt{u})|}{|z^2-u|} \frac{d u}{u}
\] 
where by Lemma~\ref{lem:J} and a change of variable $s = u/|\Re z|^2$, we obtain
\[
\int_0^\chi \frac{|J(z/\sqrt{u})|}{|z^2-u|} \frac{d u}{u} \lesssim_\Omega |\Re z|^{-1}  \int_0^{1- N^{-\delta/3}} \frac{|J(1/\sqrt{s})|}{1-s} \frac{ds}{s} \lesssim_\Omega N^{1/2-\delta}   \log N . 
\]
The previous estimates imply that $  \VERT Y_n(f_j) \VERT_1 \lesssim_\Omega N^{1/2 - \delta/3}$ for $j=1,2$ so that  if $\epsilon>0$ is sufficiently small compared to $\delta>0$, we have $|Y_n(f_j) | \le  N^{1/2 -\epsilon}$ on the event $\A_n$. 
By \eqref{fluc0}, since we also have $\|f_j\|_{\operatorname{TV},\chi}  \lesssim_\Omega N^{1/2 - \delta/3}$ for $j=1,2$, this shows that on $\A_n$, 
\begin{equation} \label{fluc1}
  \left| \sum_{k=1}^{n} \eta_{k,11} - \frac{1}{\sqrt{2\beta}} \int_0^t \frac{\d \mathbf{X}_u+ J(z/\sqrt{u}) \d\mathbf{Y}_u}{\sqrt{z^2-u}} \right|  \lesssim 
 N^{-\epsilon} . 
\end{equation}

Moreover, we also verify that $\VERT  {\eta}_{k,11}^2 \VERT_{1} \lesssim  \frac{1}{\omega_N + \hat{k}}$ for all $k\in\{1, \dots, N_H\}$, 
 so that $ \sum_{k =1}^{N_H}\big\VERT  {\eta}_{k,11}^2 \big\VERT_{1}^2  \lesssim  \omega_N^{-1} \le  \Omega^{-1}  N^{-2\delta/3} $ for $z\in\mathscr{D}_H$. By Bernstein's inequality \eqref{eq:actualbernstein} and \eqref{Rtrunc},  this shows that with probability at least $1-e^{-c N^{\epsilon}}$, 
 \vspace*{-.5cm}
\begin{equation} \label{fluc2}
\sup_{n\le N_H}  \left|  \sum_{k =1}^{n} \big( {\eta}_{k,11}^2 - \E {\eta}_{k,11}^2 \big)  \right| \leq N^{-\epsilon}  . 
\end{equation}
As for the mean, by \eqref{meaneta}--\eqref{fluctu4}, we obtain  with $t=n/N \le \chi$, 
\begin{equation} \label{fluc3}
 \begin{aligned}
\sum_{k=1}^{n} \E {\eta}_{k,11}^2 % & = \frac{1}{2\beta} \int_0^t \frac{1+ J(z/\sqrt{u})^2}{z^2-u} du  +\O\big(\omega_N^{-1}\big) \\
& = \E\left[ \left( \frac{1}{\sqrt{2\beta}}  \int_0^t \frac{\d \mathbf{X}_u+ J(z/\sqrt{u}) \d\mathbf{Y}_u}{\sqrt{z^2-u}}\right)^2 \right] + \O\big(\omega_N^{-1/2}\big) . 
\end{aligned}
\end{equation} 
where we have used that  by \eqref{J:logderivative} and Lemma~\ref{lem:J}, $\left| \frac{d}{du}\frac{1+ J(z/\sqrt{u})^2}{z^2-u}\right| \lesssim \frac{1}{(|\Re z|^2-u)^2} + \frac{1}{\sqrt{u}(|\Re z|^2-u)^{3/2}}   $
for $u\in (0,\chi]$. 
Hence by combining the estimates \eqref{fluc1},  \eqref{fluc2} and \eqref{fluc3},  this implies that there exists an event  $\mathscr{G}_n$ which is independent of $\F_{>n}$, such that $\P[\mathscr{G}_n^c] \le e^{- cN^{\epsilon}}$ and
\[
\sup_{n\le N_H , t=n/N}  \left\{ \left| \sum_{k=1}^{n} \bigg(  \eta_{k,11}  +  \frac12  \eta_{k,11}^2 \bigg) - 
\sqrt{\frac{2}{{\beta}}}  \mathfrak{g}_t(z) - \frac12  \E\left[ \left(\sqrt{\frac{2}{{\beta}}}  \mathfrak{g}_t(z)\right)^2 \right]    \right| \1_{\mathscr{G}_n} \right\}  \lesssim  4N^{-\epsilon} . 
\]
where we recall \eqref{eq:Wtz}, which states $\displaystyle   \mathfrak{g}_t(z) = \frac{1}{2}\int_0^t \frac{\d \mathbf{X}_u+ J(z/\sqrt{u}) \d\mathbf{Y}_u}{\sqrt{z^2-u}}$. 
%Note that for any $t>0$,  $  \mathfrak{g}_t(z)$ defines a GAF in $\C\setminus [-\sqrt{t}, \sqrt{t}]$. 
%Moreover, by \eqref{fluc1} and \eqref{fluc4}, this shows that there exists an event  $\mathscr{G}_n$ which is independent of $\F_{>n}$, such that $\P[\mathscr{G}_n^c] \le e^{- cN^{\epsilon}}$ and 
%\begin{equation} \label{fluc4}
%\sup_{n\le N_H , t=n/N}  \left|  \sum_{k =1}^{n} \eta_{k,11}^2 - \E\left[ \left( \frac{    \mathfrak{g}_t(z)}{\sqrt{2\beta}}  \right)^2 \right]    \right| \leq 3N^{-\epsilon} , \qquad 
%\end{equation}
%
%Combining the previous estimate with \eqref{charpol6}, this completes the proof of formula \eqref{charpoly7}\footnote{If $Z$ is a centered complex Gaussian random variable then $\E[e^Z] = e^{\E Z^2/2}$.}
%with the required uniformity and independence (see Section~\ref{sec:embeddings}).
%Finally, by Lemma~\ref{lem:magic}, we check that for any $t, s\ge 0$, 
%\begin{equation}  \label{Wcov}
%\begin{aligned}
%\E\left[  \mathfrak{g}_t(z)   \mathfrak{g}_s(q)  \right]
%\E\left[ \int_0^1 \frac{\d \mathrm{X}_u+ \frac 12J(z/\sqrt{u}) \d\mathrm{Y}_u}{\sqrt{z^2-u}} \int_0^1 \frac{\d \mathrm{X}_u+ \frac 12J(w/\sqrt{u}) \d\mathrm{Y}_u}{\sqrt{w^2-u}} \right] 
%&=   \int_0^{t\wedge s} \frac{1 +  J(z/\sqrt{u}) J(q/\sqrt{u}) }{2\sqrt{q^2-u} \sqrt{z^2-u}} du \\
%&=  - \log\big(1- J\big(z/\sqrt{t\wedge s}\big)J\big(q/\sqrt{t\wedge s}\big)\big) , 
%\end{aligned}
%\end{equation}
%where we used that $J(q) \to0$ as $q\to+\infty$. 
\hfill$\square$

%\clearpage

\appendix

\section{Properties of the inverse Joukowsky transform and asymptotics of Hermite polynomials}
\label{sect:properties}

In this section, we record a few basic properties of the  the inverse Joukowsky transform $J$ from \eqref{eq:J} which we will need for the proofs of Theorems~\ref{thm:planar} and~\ref{thm:real}; we also explain the relationship between the asymptotics from Section~\ref{sect:main} and the  Plancherel--Rotach expansion for the  Hermite polynomials.  

The Joukowsky map is chosen to be conformal in $\C \setminus [-1,1],$ and it maps $\C \setminus [-1,1]$ bijectively to $\D \setminus \{0\}.$  It can alternatively be defined in terms of the principal branch of $\sqrt{\cdot}$ by 
\[
  J(q) = q-\sqrt{q-1}\sqrt{q+1}.
\]
Geometrically, the preimages under $J$ of concentric circles for $r \in (0,1)$, i.e. 
\[
 \mathcal{E}_r = \left\{ q\in \C  :  |J(q)| = r \right\}
\]
are ellipses with foci $\pm 1$ and major semi--axis $\frac{r+r^{-1}}{2}.$  This extends to $r=1$ by taking $\mathscr{E}_1 =[-1,1],$ which is a degenerate ellipse.  Moreover, the map $J$ has boundary values on $[-1,1]$ which are either the upper or lower half of the unit disk, respectively, depending on if the interval is approached from the upper or lower half plane.

%We also explain the relationship between the asymptotics from Section~\ref{sect:main} and the  Plancherel--Rotach expansion for the  Hermite polynomials. 

\begin{lemma} \label{lem:J}
We have for any $q \in \C$, $|J(q)| \le 1/|q|$ and $|J(q)| \le | J(\Re q) |$
Moreover, we also have for any $q \in [1,2]$,
 \[
0 \le  J(q)  \le \exp\left( - \frac 23 \sqrt{q^2-1} \right) . 
\] 
\end{lemma}

\begin{proof}
  Consider the map $z\in\D \mapsto J(1/z) $. By definition, this map is holomorphic  and $|J(1/z)|<1$ for any $z\in\D$. Hence, by the Schwartz Lemma, $|J(1/z)| < |z|$ for all $z\in \D$. 
Since $|J(q)| \le 1$ for any $q\in\C$, this proves the first claim. 

\medskip

Since the closed ellipses $\{\mathcal{E}_s\}_{s \in [r,1]} $ are nested, we have $|J(q)| \ge r$ for any $q\in \cup_{s \in [r,1]} \mathcal{E}_s$. So, if $q\in  \mathcal{E}_r$ for a  $r \in (0,1]$, as $\Re q \in \cup_{s \in [r,1]} \mathcal{E}_s$, this proves the second claim. 

\medskip

For any $q>1$, we have $J(q)^{-1} >1$ and 
\begin{equation}  \label{J:logderivative}
\frac{d}{dq} \log\big(J(q)^{-1} \big)  = - \frac{J'(q)}{J(q)}= \frac{1}{\sqrt{q^2-1}} . 
\end{equation}
Since $J(1) =1$, this shows that for $q \ge 1$,
\[
J(q) = \exp\left( - \int_1^q  \frac{dt}{\sqrt{t^2-1}} \right) . 
\]
In particular for $q\in[1,2]$, we have
\[
0 < J(q) \le \exp\left( -\frac{\sqrt{q+1}}{3} \int_1^q  \frac{dt}{\sqrt{t-1}} \right) =  \exp\left( - \frac 23 \sqrt{q^2-1}\right) . 
\]
This implies the third claim. 
\end{proof}

\begin{proposition} \label{prop:Hyper}
Recall the definition \eqref{Jlambda} of $\rho_k$ as well as the Definition~\ref{def:hyper}. 
For any $z\in \mathscr{D}_H$ and for all $k=1,\dots, N_H(z)$, 
\[
|Nz^2 - k|  \ge  \frac{\omega_N + \hat{k}}{\sqrt{2}}
\qquad\text{and}\qquad
|\rho_k(z)| \le  \exp\left( - \frac{4}{3} \sqrt{\frac{\omega_N + \hat{k}}{N_H}}\right) , 
\]
where we take the convention that  $\hat{k}= N_H(z)- k $. 
\end{proposition}

\begin{proof}

We have for any $z\in \mathscr{D}_H$ and  for any $k=1,\dots, N_H$, 
\begin{equation}  \label{est0}
|Nz^2 - k| 
\ge \frac{N\big( \Re(z^2) + \Im(z^2) \big) - k  }{\sqrt{2}}
= \frac{N\big( (\Re z)^2 - (\Im z)^2 + 2\Re z\Im z \big) - k  }{\sqrt{2}}
\ge  \frac{N( \Re z)^2-  N_H + \hat{k}}{\sqrt{2}} = \frac{\omega_N + \hat{k}}{\sqrt{2}} ,
\end{equation}
where we used that $\Re(z^2) + \Im(z^2) = (\Re z)^2 + \Im z\big( 2 \Re z-\Im z) \ge (\Re z)^2$ at the second step. 
Using the estimates from Lemma~\ref{lem:J}, we also have for any $z\in \mathscr{D}_H$ and  for any $k=1,\dots, N_H$, 
\[
\big| J( \tfrac{z}{\sqrt{k/N}})  \big| 
\le  J( \tfrac{\Re z}{\sqrt{k/N}}\wedge 2) 
\le \exp\left( - \frac{2}{3} \sqrt{\frac{N(\Re z)^2 \wedge 4k -k}{k}} \right) .
\]
These estimates shows that 
\[
\big| J( \tfrac{z}{\sqrt{k/N}})  \big| 
\le \exp\left( - \frac{2}{3} \sqrt{\frac{\omega_N + \hat{k}}{k}}\right) \vee e^{-2/\sqrt{3}} . 
\]
Upon replacing $k$ by $N_H$ on the RHS and using \eqref{Jlambda}, this proves the claim. 
\end{proof}

\begin{proposition} \label{prop:Hermite}
Recall the definitions \eqref{def:lambda} and \eqref{eq:delta}  of $\lambda_+$  and $\delta_k$. 
There exists a small $\epsilon>0$ (depending on $\delta>0$ in the definition~\ref{def:hyper}) so that for any $z\in \mathscr{D}_H$ and $n \in [N^\delta, N_H]$, 
\begin{equation} \label{charpoly5}
\biggl[ \prod_{k=2}^{n}  \lambda_+(\tfrac{k-1}N)\big(1- \delta_k \big) \biggr]  \begin{pmatrix}  \lambda_+(\frac{n}{N}) \\ 1   \end{pmatrix}
= \begin{pmatrix} \pi_{n}(z) \\  \pi_{n-1}(z)  \end{pmatrix}   \left(1+  \O\left(N^{-\epsilon} \right) \right)
\end{equation}
where $\{\pi_n\}$ are the monic Hermite polynomials, orthogonal with respect to  $e^{-2N x^2}$ on~$\R$.
\end{proposition}

\begin{proof}

\noindent\underline{1- Riemann sum approximations:}\
Recall that $\lambda_+(t) = \sqrt{t}J(z/\sqrt{t})^{-1}/2.$
Since $J : \C \setminus (-\infty,1] \mapsto \D \setminus[-1,0]$, we 
 can define an analytic version of it complex logarithm on this simply connected domain using the principal branch.  Hence also, we can define the complex logarithm of $\lambda_+(t)$ which is analytic in $z$ and infinitely differentiable in $(z,t)$ on the domain
\[
  \widehat{\mathscr{D}} = \left( (z,t) \in \C \times [0,\infty) : \neg (z \leq \sqrt{t}) \right) . 
\]
For $z\in \mathscr{D}_H, n\le N_H$ we have that $(z, \tfrac{n}{N})$ is in $ \widehat{\mathscr{D}}$ since $n/N \leq N_H/N < N_P/N < (\Re z)^2$ and $\Re z > 0.$
Computing the second derivative:
\[
\frac{d^2 \log(\lambda_+(t))}{dt^2}
=-\frac{2\lambda_+(t)+\sqrt{z^2-t}}{16\lambda_{+}(t)^2(z^2 - t)^{3/2}}.
\]
Then, by using the trapezoidal rule, we have for $t=n/N$, 
\[
 \sum_{k=1}^{n} \log\lambda_+(\tfrac{k}N)   = N \int_0^t \log\left( \lambda_+(u) \right) du
 + \frac{ \log\left(\lambda_+(t)\right) - \log(\lambda_+(0) )}{2}
 + \O\left( \frac{1}{\sqrt{N}}  \sum_{k=1}^{n} \bigg| \frac{\lambda_+(\tfrac{k}N) + \sqrt{z^2-\frac kN}}{\lambda_+(\tfrac{k}N)^2(Nz^2- k)^{3/2}} \bigg| \right) . 
\]
According to Lemma \ref{lem:J}, we have $|\lambda_+(t)| \ge |z| /2$ for any $t \in (0,1]$. 
Using the estimate \eqref{est0}, this implies that for $z\in \mathscr{D}_H, n\le N_H$, 
\[
\frac{1}{\sqrt{N}}  \sum_{k=1}^{n} \bigg| \frac{\lambda_+(\tfrac{k}N) + \sqrt{z^2-\frac kN}}{\lambda_+(\tfrac{k}N)^2(Nz^2- k)^{3/2}} \bigg| 
\  \lesssim \  \frac{N^{-1/2}}{|z| \sqrt{\omega_N}} +  \frac{\log N}{|z|^2N}
= \O\left(N^{-\delta}\right) ,
\]
where we have used that $|z| \ge N^{-1/2+\delta}$ for $z\in \mathscr{D}_H$.
Since $\lambda_+(0) = z$,  this shows that
\begin{equation} \label{mean1}
 \sum_{k=1}^{n} \log\lambda_+(\tfrac{k}N)   = N \int_0^t \log\left( \lambda_+(u) \right) du
 + \frac{ \log\left(\lambda_+(t)\right) - \log(z)}{2} + \O\left(N^{-\delta} \right) . 
\end{equation}

In addition, let us observe 
\[
\frac{1}{2} \int_{0}^{t} \frac{ du}{z^2-u} = \log\left(\frac{z}{\sqrt{z^2-t}} \right) 
\]
where the branch has been chosen so that the RHS is analytic on $\C \setminus[-\sqrt t, \sqrt t]$. 
 
Taylor expanding \eqref{eq:delta} and using the estimate \eqref{est0} we have
\[
  \delta_k 
  = -\frac{1}{4N(z^2 - \tfrac kN)} + \O(|Nz^2-k|^{-2})
  = -\frac{1}{4N(z^2 - \tfrac kN)} + \O(|\omega_N + \widehat{k}|^{-2}),
\]
so that using the estimate \eqref{est0}, we obtain for $n \leq N_H$
\begin{equation} \label{mean2} 
\begin{aligned}
\sum_{k=2}^n \delta_k 
&=-  \frac 14 \int_{0}^{t} \frac{ du}{z^2-u}  + \O\left( \sum_{\widehat{k}=1}^\infty    \frac{1}{(\omega_N + \widehat{k})^2}  \right) \\
& = -\frac 12 \log\left(\frac{z}{\sqrt{z^2-t}} \right)  +  \O\left(\omega_N^{-1} \right) . 
\end{aligned}
\end{equation}

Since $\omega_N \ge \Omega N^{2\delta/3}$, combining the estimates \eqref{mean1} and \eqref{mean2}, we obtain that for $n \le N_H$ with $t=n/N$,  
\begin{align*}
 \sum_{k=1}^{n}  \log\lambda_+(\tfrac{k}N)  -  \sum_{k=2}^{n}   \delta_k  & = N \int_0^t \log\left( \lambda_+(u) \right) du
 + \frac{ \log\left(\lambda_+(t)\right) - \log(\sqrt{z^2-t})}{2} +  \O\left(N^{-\frac{2\delta}{3}} \right) \\
 &
  = N \int_0^t \log\left( \lambda_+(u) \right) du
 + \log\left( \frac{\sqrt{z-\sqrt{t}} + \sqrt{z+\sqrt{t}} }{2(z^2-t)^{1/4}} \right)+  \O\left(N^{-\frac{2\delta}{3}} \right) , 
 \end{align*} 
where we used that $\left(\sqrt{z-\sqrt{t}} + \sqrt{z+\sqrt{t}} \right)^2 = 4 \lambda_+(t)$. 
Since $\sum_{k=2}^n \delta_k^2=  \O\left(\omega_N^{-1} \right) $ for  $n \le N_H$, this shows that   with $t=n/N$,  
\begin{equation} \label{mean3}
\prod_{k=1}^{n}  \lambda_+(\tfrac{k}N) \prod_{k=2}^{n} \big(1- \delta_k \big)
= \left(\frac{\gamma(z/\sqrt{t}) + \gamma(z/\sqrt{t})^{-1}}{2} \right)  \exp\left( N \int_0^t \log\left( \lambda_+(u) \right) du+  \O\left(N^{-\frac{2\delta}{3}} \right) \right) 
\end{equation}
where  $\gamma(z) =  \left(\frac{z+1}{z-1}\right)^{1/4}$ is analytic, non-zero on $\C\setminus[-1,1]$, and the error is uniform for  $z\in \mathscr{D}_H$.

\medskip

\noindent\underline{2- $g$ function:}\
Since $\lambda_+(u) = \frac{\sqrt{u} J(z/\sqrt{u})^{-1}}{2} $, we deduce from  \eqref{J:logderivative} that
$\frac{d \lambda_+(u)}{dz}   =  \frac{\lambda_+(u)}{\sqrt{z^2-u}}$. 
This implies that 
\begin{equation} \label{mean4} 
\frac{d}{dz} \left( \int_0^t \log\left( \lambda_+(u) \right) du \right) = \int_0^t \frac{du}{\sqrt{z^2-u}}
= 2 \left( z- \sqrt{z^2-t} \right) = 4 \lambda_-(t)
\end{equation}
where we used \eqref{def:root}. 
Let $\rho(x) = \frac{\sqrt{1-x^2}}{\pi/2}\1_{|x| \le 1} $ denotes the semicircle density on $[-1,1]$ and recall that its Stieltjes transform satisfies for all $t>0$, 
\[
\int \frac{\rho_t(x)}{z-x} dx =\frac{ 4 \lambda_-(t)}{t}  ,\qquad \text{where}\qquad  \rho_t(x) = \rho(x/\sqrt{t})/\sqrt{t} . 
\]
So if we define for  $(z,t) \in \widehat{\mathscr{D}},$
\begin{equation}\label{eq:gfunction}
g_t(z) = \int \log (z -x) \rho_t(x) dx,
\end{equation}
using the principal branch of the logarithm.
We deduce from \eqref{mean4} that
\[
 \int_0^t \log\left( \lambda_+(u) \right) du = t g_t(z)  . 
\]
Note that there is no constant of integration since $\lambda_+(t) = z + \O_t(z^{-1})$ as $z\to+\infty$ and  $g_t(z)= \log z + \O_t(z^{-2})$ as $z\to+\infty$.
Hence, by \eqref{mean3}, we obtain that  for  $n \le N_H$ and  $t=n/N$,  
\begin{equation} \label{mean5}
\prod_{k=1}^{n}  \lambda_+(\tfrac{k}N) \prod_{k=2}^{n} \big(1- \delta_k \big)
= \left(\frac{\gamma(z/\sqrt{t}) + \gamma(z/\sqrt{t})^{-1}}{2} \right)  \exp\left( n g_t(z)+  \O\left(N^{-\frac{2\delta}{3}} \right) \right) .
\end{equation}

\medskip

\noindent\underline{3- Hermite asymptotics:}\
Recall from \cite[Theorem 1.3]{Deiftuni} that the (monic) Hermite polynomials $\{\Pi_n\}$ defined with respect the weight $e^{-2n x^2}$ on $\R$ have the asymptotics as $n\to+\infty$ with $g=g_1$ as in \eqref{eq:gfunction} 
\begin{equation} \label{Piasymp}
\Pi_n(z) = \left(\frac{\gamma(z) + \gamma(z)^{-1} }{2} \right) e^{n g(z)}\big(1+ \O(n^{-\epsilon})\big) . 
\end{equation}
uniformly for $(\Re z)^2 \ge 1+ n^{-1/3+\delta}$, with $\epsilon>0$ sufficiently small depending on $\delta>0$.
Using the scaling property of the Hermite polynomials  
$\pi_n(z) = \Pi_n(z/\sqrt{t}) t^{n/2}$ and $g_t(z) = g(z/\sqrt{t}) + \log(\sqrt{t})$ for $t=n/N$, we obtain  that as $n\to+\infty$, 
\[
\pi_n(z) =  \left(\frac{\gamma(z/\sqrt{t}) + \gamma(z/\sqrt{t})^{-1}}{2} \right)  \exp\left( n g_t(z)+ \O(n^{-\epsilon})\right) , \qquad n \le N_H(z) . 
\]
By \eqref{mean5}, this shows that for $\delta>0$, there exists a small $\epsilon>0$ (depending on $\delta>0$) such that uniformly for $z\in\mathscr{D}_H$, for all $n\in [N^\delta, N_H]$, 
\[
\prod_{k=1}^{n}  \lambda_+(\tfrac{k}N) \prod_{k=2}^{n} \big(1- \delta_k \big)
= \pi_n(z)  \left(1+  \O\left(N^{-\epsilon} \right) \right) . 
\]
The same argument  with $t=\frac{n-1}{N}$ shows that  for all $n\in [N^\delta, N_H]$,
$
\prod_{k=1}^{n-1}  \lambda_+(\tfrac{k}N) \prod_{k=2}^{n} \big(1- \delta_k \big)
= \pi_{n-1}(z)  \left(1+  \O\left(N^{-\epsilon} \right) \right)$.
This completes the proof. 
\end{proof}

\begin{lemma} \label{lem:magic}
The function
$(q,z) \mapsto \log\big(1- J(q)J(z)\big)$ 
is biholomorphic in the domain $\C\setminus[-1,1]$ and we have for any  $q, z\in \C\setminus[-1,1]$ and $t\in(0,1]$, 
\[
\frac{d}{dt}  \log\big(1- J(q/\sqrt{t})J(z/\sqrt{t})\big) = - \frac{1 +  J(z/\sqrt{t}) J(q/\sqrt{t}) }{4\sqrt{q^2-t} \sqrt{z^2-t}} . 
\]
\end{lemma}

\begin{proof}
Since $J: \C\setminus[-1,1] \mapsto \C$ is holomorphic, it immediately follows that  $\log\big(1- J(q)J(z)\big)$ is  biholomorphic in $(\C\setminus[-1,1])^2.$
Differentiating in $t,$
%and $t\in(0,1]$, 
\[
\frac{d}{dt}  \log\big(1- J(q/\sqrt{t})J(z/\sqrt{t})\big) = \frac{1}{2t^{3/2}} \frac{zJ'(z/\sqrt{t}) J(q/\sqrt{t})+ qJ(z/\sqrt{t}) J'(q/\sqrt{t}) }{1- J(q/\sqrt{t})J(z/\sqrt{t})}. 
\]
Using the identity \eqref{J:logderivative}, this implies that 
\[
\frac{d}{dt}  \log\big(1- J(q/\sqrt{t})J(z/\sqrt{t})\big) = - \frac{J(z/\sqrt{t}) J(q/\sqrt{t})}{\sqrt{z^2/t-1} \sqrt{q^2/t-1}} \frac{z\sqrt{q^2/t-1}+ q\sqrt{z^2/t-1} }{ 2t^{3/2}\big(1- J(q/\sqrt{t})J(z/\sqrt{t})\big)}. 
\]
Now, let us observe that since $J(q)^{-1} = q+\sqrt{q^2-1}$, we have 
\[
  \frac{J(q/\sqrt{t})^{-1} J(z/\sqrt{t})^{-1} - J(q/\sqrt{t})J(z/\sqrt{t})}
  {2}
  = \frac{z\sqrt{q^2/t-1}+ q\sqrt{z^2/t-1}}{\sqrt{t}}  , 
\]
so that 
\[ \begin{aligned}
\frac{d}{dt}  \log\big(1- J(q/\sqrt{t})J(z/\sqrt{t})\big) & = - \frac{1/(4t)}{\sqrt{z^2/t-1} \sqrt{q^2/t-1}} \frac{1- J(q/\sqrt{t})^2J(z/\sqrt{t})^2}{1- J(q/\sqrt{t})J(z/\sqrt{t})} \\
& = - \frac{1+J(q/\sqrt{t})J(z/\sqrt{t})}{4\sqrt{q^2-t} \sqrt{z^2-t}} . 
\end{aligned}\]
\end{proof}

\section{Estimates for the noise} \label{sect:est}

In this section, we provide estimates for the random variables \eqref{def:XY} which are necessary to obtain control of the noise for the proofs of Theorems~\ref{thm:planar} and~\ref{thm:real}. 

\begin{lemma} \label{lem:XY}
The random variables $X_1, X_2, X_3, \dots $ and $Y_2, Y_3, \dots$ are all independent. 
We have $X_k \sim \mathcal{N}(0,1)$ for all $k\in \N$. Moreover,  we have for any $k\ge2$,
\[
\E Y_k = 0 , \qquad \E Y_k^2 = 1
\qquad \text{and}\qquad
\VERT Y_k \VERT_1 \lesssim 1. 
\]
Finally,  for any $k\ge 2$ and $ \xi \le  \sqrt{\beta(k-1)/2}$, 
\[
\P\left[ |Y_k | \ge \xi \right]  \le 2 e^{- \xi^2/4} . 
\]
\end{lemma}

\begin{proof}
We have seen at the beginning of Section~\ref{sect:TM} that $a_{k}^2 \sim \Gamma(\frac{\beta k}{2},2)$ for $k\ge 1$. Using the formulae for the mean and variance of a Gamma random variable, we find that 
$\E Y_k = 0$ and $\E Y_k^2 = 1$.
Moreover, we can also explicitly compute the Laplace transform of the random variable $a_{k}^2$; for $\xi \in [0, \tfrac12]$,
\[
  \E[e^{\xi a_{k}^2}] = (1- 2{\xi})^{- \beta k/2}  .
\]
%This implies that $\VERT a_{k}^2 \VERT_1 \lesssim k$ so that by a rescaling, we obtain $\VERT Y_{k+1} \VERT_1 \lesssim  \sqrt{k}$ for $k\ge 1$. 
%This estimates is accurate for the tail of the random variables $Y_{k}$ and we can be more precise in the moderate deviation regime. 

Using that $1-x  \ge e^{-x -  x^2}$ for all $x\in[0,1/2]$, we obtain for any $k\ge 1$ and $\xi \in [0,\sqrt{2 \beta k}/4]$, 
\[
\E[e^{\xi Y_{k+1}}] = \left(1- \frac{2\xi}{\sqrt{2\beta k}}\right)^{- k\beta /2} e^{- \xi \sqrt{k \beta/2}} \le e^{ \xi^2 } .
\]
Similarly, using that $1+x \ge  e^{x - x^2/2}$ for all $x\ge 0$, we have for any $\xi \ge 0$,
\[
\E[e^{ - \xi Y_{k+1}}] = \left(1 + \frac{2\xi}{\sqrt{2\beta k}}\right)^{-k \beta /2} e^{\xi \sqrt{k \beta/2 }} \le e^{ \xi^2/2} . 
\]
Both these estimates combined show that uniformly in $k,$ sufficiently small exponential moments exist, so there is an absolute constant $\varkappa > 0$ so that $\VERT Y_{k+1} \VERT_1 \leq \varkappa\beta^{-1/2}.$ 
Moreover by Markov's inequality, these bounds show that for any $k\ge 1$ and $\xi \in [0, \sqrt{\beta k/2}]$, 
\[
\P[ | Y_{k+1}| \ge \xi ] \le  \E[e^{\xi |Y_{k+1}|/2}] e^{- \xi^2/2}  \le 2 e^{-\xi^2/4} . 
\]
This completes the proof. 
\end{proof}

The important consequence is that we can truncate the random variables $X_1 ,Y_2, X_2 , \cdots,  Y_N, X_N$ by working on an event of overwhelming probability. 
For  $\mathrm{S}>0$, we define the event 
\begin{equation}\label{truncation} \begin{aligned}
\mathscr{T}_\mathrm{S}   =  \left\{ |Y_k| \le \sqrt{\mathrm{S}} \text{ for }k= \lceil \beta^{-1} \mathrm{S} \rceil, \dots N \right\} \cap  \left\{ |Y_k| \le  \mathrm{S}  \text{ for }k=1, \dots  \lfloor \beta^{-1} \mathrm{S}  \rfloor \right\} \\
 \cap \left\{ |X_k| \le  \sqrt{\mathrm{S}} \text{ for }k=1, \dots N \right\} . 
\end{aligned}
\end{equation}
Then, it follows from Lemma~\ref{lem:XY} that there exists absolute constants $C, c>0$ such that for any $\mathrm{S}>0$, 
\begin{equation} \label{Rtrunc}
\P[\mathscr{T}_\mathrm{S}] \ge 1-  C \left( 2+  \tfrac{\mathrm{S}}{\beta N} \right)  N  e^{-c \mathrm{S}} . 
\end{equation}
Note that it is possible to choose $\mathrm{S}$ growing with the dimension $N$ in the estimate \eqref{Rtrunc}, so that the event $\mathscr{T}_\mathrm{S}$ holds with overwhelming probability.

\medskip

Since, we would like to work with truncated random variables instead of $\{ X_k, Y_k \}_{k=1}^N$, this motivates the following notation. 

\begin{definition} \label{def:cond}
Let $\P_\mathrm{S} = \frac{\P[\cdot \1_{\mathscr{T}_\mathrm{S}}]}{\P[\mathscr{T}_\mathrm{S}]}$. 
This probability measure is absolutely continuous with respect to $\P$ and the random variables $X_1, X_2, X_3, \dots $ and $Y_2, Y_3, \dots$ remain independent under  $\P_\mathrm{S}$.
\end{definition}

Let us also record that by Lemma~\ref{lem:XY} and \eqref{Rtrunc}, we have for any integers $k\ge 1$ and $q\ge 1$, 
\begin{equation} \label{XYmoments}
\E_\mathrm{S}[X_k^q ] =  \E[X_k^q]  + \O_q(N e^{-c \mathrm{S}})
\qquad\text{and}\qquad
\E_\mathrm{S}[Y_{k+1}^q ] =  \E[Y_{k+1}^q]  + \O_q(N e^{-c \mathrm{S}}) . 
\end{equation}

\medskip

We now turn to the applications  for the proofs of Theorems~\ref{thm:planar} and~\ref{thm:real} (see Section~\ref{sec:rmp} for further details).

\begin{lemma} \label{lem:planarest}
Let $\mathrm{S} = N^{\epsilon}$ and $\P_\mathrm{S}$ be as in definition~\ref{def:cond}, then the conditions \eqref{cond:eta} hold under $\P_\mathrm{S}$  uniformly for all $z \in \mathscr{P}$, \eqref{def:P}. 
\end{lemma}

\begin{proof}
  First, we claim that if $\Im z \ge N^{-\alpha}$ and $k\in\{0,1,\dots, N\}$, 
\begin{equation} \label{estimate1}
|N z^2 - k|^{-1} 
\leq N^{2\alpha - 1} .
  \end{equation}
Indeed,  the closest point the parabola $z \mapsto Nz^2$ restricted to a horizontal line with $\Im z = \eta >0$ makes to the positive real axis is for $z=\i \eta$.  For $\eta = N^{-\alpha}$ and $k=0$, we obtain the bound claimed in the previous display.   As for $J$, we have that for  $N^\alpha \ge 2$, $t \geq 1$ and $\Im z \ge N^{-\alpha},$
  \begin{align} 
    |J(zt)|  &\notag \leq |J(z)| \leq |J(\i N^{-\alpha})| = \sqrt{1+ N^{-2\alpha}} - N^{-\alpha}\\
    & \label{estimate2}  \leq  1- 3 N^{-\alpha}/4 .
  \end{align}
  Here we used that $J$ maps the $\C\setminus \mathscr{E}$ conformally onto the disk $\big\{|q|\le  |J(\i N^{-\alpha})| \big\}$ where $\mathscr{E}$  is the ellipse with foci $\pm 1$ tangent to the line $\Im z = N^{-\alpha}$.

  \medskip
  
 By similar considerations, we verify that both estimates \eqref{estimate1} and \eqref{estimate2}  hold if $\Re z \ge 1+ N^{-2\alpha}/2$. Hence by symmetry, these estimates also hold uniformly for all $z \in \mathscr{P}$ and $k\in\{1,\dots, N\}$.
%   \[
%    \frac{1}{|N z^2 - k|} \leq  N^{2\alpha-1} 
%    \qquad\text{and}\qquad 
%     |J(zt)| \leq 1- N^{-\alpha} . 
%  \]
% 
 \smallskip
 
By \eqref{def:eta1},  this implies that for all $z\in \mathscr{P}$ and $k\in\{1,\dots, N\}$,
 \begin{equation} \label{boundeta1}
|  \eta_{k,11}  |  \le     \sqrt{\tfrac{N^{2\alpha - 1}}{2\beta}}\big(|X_k| + |Y_k| \big) . 
 \end{equation}
So that by  \eqref{XYmoments}, we obtain $\E_\mathrm{S}[  \eta_{k,11} ] = \O(Ne^{-cN^{\epsilon}})$, 
 $\E_\mathrm{S}[ |\eta_{k,11}|^2 ] \le \tfrac{N^{2\alpha - 1}}{\beta}\big(1+  \O(Ne^{-c N^{\epsilon}})\big)$, 
 and $ |  \eta_{k,11}  |  \le \sqrt{2\beta^{-1}} N^{\epsilon+ \alpha-1/2} $, $\P_{\mathrm{S}}$ almost surely.  
  Therefore, this shows that for an absolute constant $C_\beta$, the random variables $\{\eta_{k,11} \}_{k=1}^N $ satisfies the conditions \eqref{cond:eta} from Proposition~\ref{prop:rec} (uniformly for all $z \in \mathscr{P}$). 
  
  \medskip
  
   Now, by \eqref{eq:delta} and \eqref{Jlambda},   we also have  for all $z\in \mathscr{P}$ and $k\in\{1,\dots, N\}$, 
   \[
   | \delta_k | \le N^{2\alpha - 1}/4 
   \qquad\text{and}\qquad
   |\rho_k| \le  1- c N^{-\alpha}  
   \]
  with $c\in (1,3/2)$  if $N$ is sufficiently large. 
   By  \eqref{def:eta2},  this implies that 
  \[
 | \eta_{k,21} | \le 2  \big(   | \delta_k |  + |\eta_{k,11}| \big)  
\quad\text{and}\qquad
|\eta_{k,12}| \le  2  | \delta_k |+ \sqrt{\tfrac{2N^{2\alpha - 1}}{\beta}}\big(|X_k| + |Y_k| \big)  ,
  \]
  where we have used that $|\breve{Y}_k| \le |Y_k|$ and that $|\delta_k|+|\eta_{k,11}| \le 1/2 $  if $N$ is sufficiently large. Using all the previous estimates, this shows that  for all $z\in \mathscr{P}$ and $k\in\{1,\dots, N\}$,
  \[
  \E_\mathrm{S}[ |\eta_{k,21}|^2 ] ,   \E_\mathrm{S}[ |\eta_{k,12}|^2 ] \lesssim (1+\beta^{-1}) N^{2\alpha - 1} 
  \qquad \text{and} \qquad
   | \eta_{k,21} | ,  | \eta_{k,12} |  \lesssim  (1+\beta^{-1}) N^{\epsilon+\alpha-1/2} , 
  \]
  where the second holds $\P_{\mathrm{S}}$ almost surely if $N$ is sufficiently large. 
  Plainly, similar bounds holds for the random variable $\eta_{k,22}$ as well. 
 As for the means, by a Taylor expansion of \eqref{def:eta2},  we have   
\begin{equation}\label{eq:etak21expansion}
  \begin{aligned}
\eta_{k,21}
& =  \big( \delta_k + \eta_{k,11} \big)\big(1 + \delta_k + \eta_{k,11}+  \O( |\delta_k|^2+|\eta_{k,11}|^2)\big) \\
& =  \delta_k  +  \eta_{k,11}  + (\delta_k + \eta_{k,11})^2 + \O\left(  |\delta_k|^3+|\eta_{k,11}|^3\right) .
\end{aligned}  
\end{equation}
Using the estimates \eqref{XYmoments} and \eqref{boundeta1}, this shows that  for all $z\in \mathscr{P}$ and $k\in\{1,\dots, N\}$,
  \[
  \E_\mathrm{S}[  \eta_{k,21} ] =  \delta_k + \E  \eta_{k,11}^2  +  \O_\beta(N^{3\alpha -3/2}). 
  \]
  We conclude that if $N$ is sufficiently large, we have 
  $\big|  \E_\mathrm{S}[  \eta_{k,21} ] \big|  \lesssim (1+\beta^{-1}) N^{2\alpha - 1}  $.  By similar considerations, we obtain the same estimates for the means of $\eta_{k,12}$ and  $\eta_{k,22}$. 
  \end{proof}
 
Analogous estimates also holds for the noise if $z\in \mathcal{D}_H$ away from the turning point of the recurrence, that is for $k\in \{1, \dots, N_H\}$, see Definition~\ref{def:hyper}. 
However, in this case, we need to keep carefully track of the size of the noise as $k$ approaches the turning point. 

\begin{lemma} \label{lem:noise}
Let $\{\eta_{k,ij}\}$ be as in Lemma~\ref{lem:diag} for $i,j \in \{1,2\}$ and let $\P_\mathrm{S}$ be as in Definition~\ref{def:cond} with $\mathrm{S} \ge r_\beta \log N$ for a fixed large $ r_\beta >0$. 
It holds  $\P_\mathrm{S}$ almost surely, for all $z\in \mathcal{D}_H$ and all $k=1, \dots, N_H(z)$, 
\begin{equation} \label{est:eta2}
\big|\E_{\mathrm{S}} \eta_{k,ij} \big|  \lesssim  \frac{1+\beta^{-1}}{\omega_N +\hat{k}} \ ,
\quad
\E_R |\eta_{k,ij}|^2 \lesssim  \frac{1+\beta^{-1}}{(\omega_N +\hat{k})}
\quad
\text{and}
\quad
|\eta_{k,ij}|  \lesssim  \sqrt{ \frac{\mathrm{S}(1+\beta^{-1})}{(\omega_N +\hat{k})} }  , 
\end{equation}
 where $\hat{k}= N_H(z) - k$ and the implied constant depends only on $r_\beta$. 
\end{lemma}

\begin{proof}
First recall that for all  $z\in \mathcal{D}_H$,   $\omega_N  \ge   \Omega N^{2\delta/3}$ is a large parameter.
With the notation from the proof of Lemma~\ref{lem:diag}, we have $\eta_{k,11} =  \sqrt{\frac{1/2\beta}{Nz^2-k}}\big( X_k + \breve{Y}_k \big)$ where $\breve{Y}_k =Y_k J\big(z \sqrt{\tfrac{N}{k-1}}\big)$ and $X_k, Y_k$ have mean 0 and variance 1 under $\P$. 
Then, according to \eqref{XYmoments}, this shows that  
\[
\E_{\mathrm{S}} \eta_{k,11} =  \sqrt{\frac{1/2\beta}{Nz^2-k}}   \O(N^{1-c r_\beta})  
 \quad
\text{and}
\quad
\E_{\mathrm{S}} |\eta_{k,11}|^2 \le \frac{1/\beta}{|Nz^2-k|} \big( 1+  \O(N^{1-c r_\beta})  \big)
\]
By Proposition \ref{prop:Hyper}, this shows that $|\E_{\mathrm{S}}\eta_{k,11}|  =  \O( \beta^{-1/2}N^{1-c r_\beta})$ 
and $\E_{\mathrm{S}} |\eta_{k,11}|^2 \lesssim \frac{\beta^{-1}}{\omega_N +\hat{k}}$. 
Moreover, by Lemma~\ref{lem:J},  we have 
\[
|\breve{Y}_k  |  \le |Y_k|  \sqrt{\tfrac{k}{N (\Re z)^2}}  
\qquad\text{and}\qquad
| \rho_k | \le \frac{k}{N (\Re z)^2} . 
\]
This implies that conditionally on the event $\mathscr{T}_{\mathrm{S}}$, \eqref{truncation}, 
$|\eta_{k,11}| \le  4  \sqrt{\frac{\mathrm{S}/\beta}{\omega_N +\hat{k}}}$ if $N$ is sufficiently large (depending on $\delta, \beta$).
By formula \eqref{eq:delta}, we also verify that for $k=1, \dots, N_H$,  $|\delta_k | \lesssim \frac{1}{\omega_N +\hat{k}}$. 
Then by \eqref{def:eta2}, we can obtain similar estimates for the other random variables $\eta_{k,ij}$. 
For instance, conditionally on the event $\mathscr{T}_{\mathrm{S}}$, 
\[
|\eta_{k,21}| \le 2 |\delta_k| + 2| \eta_{k,11}| \lesssim \sqrt{\frac{(1+ \beta^{-1})\mathrm{S}}{\omega_N +\hat{k}}}
\]
and 
\[
\E_{\mathrm{S}} |\eta_{k,21}|^2  \le 8 \big( |\delta_k|^2  + \E_{\mathrm{S}} | \eta_{k,11}|^2  \big) \lesssim \frac{1+ \beta^{-1}}{\omega_N +\hat{k}} . 
\]
As for the mean, using \eqref{eq:etak21expansion},
  \[
 \E_{\mathrm{S}}[  \eta_{k,21} ] =  \delta_k - \E  \eta_{k,11}^2  +  \O\left( \Big( \tfrac{1/\beta}{\omega_N+ \hat{k}}\Big)^{-3/2} \right). 
  \]
This shows that we also have $|\E_{\mathrm{S}} \eta_{k,21}|  \lesssim \frac{1+ \beta^{-1}}{\omega_N +\hat{k}}$.  By similar considerations, we obtain the same estimates for the means of $\eta_{k,12}$ and  $\eta_{k,22}$. 
\end{proof}

\section{Strong embeddings}
\label{sec:embeddings}

We discuss some of the literature on so-called \emph{strong embeddings}, which embed random walks into Brownian motions with essentially optimal supremum error bounds.  The classical paper in this subject is the Koml\'os--Tusn\'ady--Major \cite{KMT}, but many other authors have contributed, e.g.\ \cite{Chatterjee, Shao,  Goldstein, Zaitsev, Sakhanenko}.  For our purposes, the formulation of \cite{Shao} suffices (the theorem is due to \cite{Sakhanenko}), which we formulate in the following way.
\begin{theorem}  \label{thm:shaoKMT}
  Suppose that $\left\{ X_k \right\}_1^\infty$ is a sequence of independent, real, centered, variance $1$ random variables having $M = \sup_{k\in\N} \VERT X_k \VERT_1 < \infty.$  Then there is a constant $C_M>0$ and an extension of the probability space supporting a standard Brownian motion ${\mathbf{X}}$ so that for any $n\in\N$,
  \[
    \left\VERT \max_{1 \leq k \leq n} \biggl| {\textstyle \sum_{j=1}^k X_j} - \mathbf{X}_n \biggr| \right\VERT_1 \leq C_M\log n .
  \]
\end{theorem}
\begin{proof}
  The condition in \cite[Theorem A]{Shao} is that there is a $\lambda > 0$ so that
  \(
  \lambda \Exp[ e^{\lambda |X_k|} |X_k|^3] \leq \Exp X_k^2 = 1.
  \)
  By H\"older's inequality 
  \(
  \Exp[ e^{\lambda |X_k|} |X_k|^3] \leq  (\Exp[ e^{2\lambda |X_k|}] \Exp |X_k|^6)^{1/2} \leq C M^3,
  \)
  for some absolute constant $C>0$ as soon as $2\lambda \leq M.$  Hence the condition of \cite{Shao} is satisfied taking $\lambda = 1/(CM^3).$
  The conclusion of \cite{Shao} shows that with $\mathcal{M}_n =  \max_{1 \leq k \leq n} \biggl| {\textstyle \sum_{j=1}^k X_j} - \mathbf{X}_n \biggr|,$ for all $n \geq 1,$
  \[
    \Exp e^{\tfrac{\mathcal{M}_n}{(CM^3)}} \leq 1 + CnM^3
  \]
  for an appropriately large absolute constant $C > 0.$  Thus from Jensen's inequality,
  \[
    \Exp e^{\tfrac{\mathcal{M}_n}{(CM^3\log n)}}
    \leq
    \biggl(\Exp e^{\tfrac{\mathcal{M}_n}{(CM^3)}}\biggr)^{1/\log n}
    \leq 
    (1 + CnM^3)^{1/\log n},
  \]
  which is uniformly bounded in $n \in \N$ for each $M.$
  We conclude that $\VERT \mathcal{M}_n \VERT_1 \lesssim \log n.$
\end{proof}

This type of strong embedding can be used to control the errors in integration, using integration by parts.  For a continuously differentiable function $f : [0,1] \to \C,$ define
\[
  \|f\|_{\operatorname{TV},t}
  =  |f(t)| + \int_{0}^t |f'(s)|\,ds.
\]
\begin{proposition} \label{prop:integralapprox}
  Let $\left\{ X_n \right\}_1^\infty$ be random variables and let $(B_t : t \geq 0)$ be a standard Brownian motion defined on the same probability space.
  Let $W_t=\frac{1}{\sqrt{N}}B_{N t}$ be another Brownian motion.
  For any continuously differentiable $f : [0,1] \to \C,$ it holds for any $n\in\{1,\dots, N\}$ with $t=n/N$,  
  \[
    \biggl|\frac{1}{\sqrt{N}} \sum_{k=1}^{n} f(\tfrac{k}{N}) X_k - \int_0^t f(t) dW_t\biggr|
    \leq \frac{1}{\sqrt{N}} \left( \|f\|_{\operatorname{TV},t}\max_{1 \leq k \leq n} \biggl| {\textstyle \sum_{j=1}^k X_j} - B_k \biggr|    +Y_n(f) \right) 
  \]
  where $Y_n(f)$ is an (increasing) random variable measurable with respect to $\sigma\{ B_u : u\le n \}$
  which satisfies 
  \[
  \VERT Y_n(f) \VERT_1 \lesssim  \int_{0}^t |f'(s)|\,ds.
  %,
  %\qquad\text{and}\qquad
  %  \|f\|_{\operatorname{TV},t}
  %  =  |f(t)| + \int_{0}^t |f'(s)|\,ds .
  \]
\end{proposition}
\begin{proof}
 Define $X(t) = \sum_{0 \leq k \leq t N} X_k$ (where we set $X_0 = 0$).
Then we apply Abel summation to the partial sum, 
Let $W_t = \frac{1}{\sqrt N} B_{tN},$ and apply stochastic integration by parts: 
\[
  \int_0^t f(t) dW_t
  = W_{t} f(t) - \int_0^t W_s f'(s)\,ds
  = \frac{B_{tN}}{\sqrt{N}} f(t) - \frac{1}{\sqrt{N}}\int_0^t B_{sN} f'(s)\,ds.
\]
Hence we arrive at the bound with  $t=n/N$, 
\[
  \begin{aligned}
    \biggl|\frac{1}{\sqrt{N}}\sum_{k=1}^{n} f(\tfrac{k}{N}) X_k 
    -\int_0^t f(t) dW_t
    \biggr| 
  &  \leq \frac{1}{\sqrt{N}} \biggl(|f(t)|+ \int_{0}^t |f'(s)|\,ds\biggr)\cdot \max_{1 \leq k \leq n} \biggl| {\textstyle \sum_{j=1}^k X_j} - B_k\biggr| \\
   & \quad + \frac{1}{\sqrt{N}} \biggl| \int_0^t (B_{sN}-B_{\lfloor s N \rfloor}) f'(s)\,ds  \biggr|  .
     \end{aligned}
\]
%Note that we used that $X(s) =0$ for $s\in[0,1/N]$.
For this second term, we can define $\xi_k = \max_{s \in [0,1]} |B_{s+k} - B_k|$ and bound
\[
 \biggl| \int_0^t (B_{sN}-B_{\lfloor s N \rfloor})f'(s)\,ds \biggr|
  \leq 
  \sum_{k=0}^{n-1} \xi_k  \int_{\frac kN}^{\frac{k+1}N} |f'(s)|\,ds 
  \eqqcolon Y_n(f).
\]
Since $\VERT\xi_k \VERT_1  \lesssim 1$ uniformly for all $k\ge 0$,  we conclude that 
\[
\VERT Y_n(f) \VERT_1 \lesssim \sum_{k=0}^{n-1} \int_{\frac kN}^{\frac{k+1}N} |f'(s)|\,ds  =   \int_0^1 |f'(s)| .
\]
This completes the proof.
\end{proof}

\section{Central limit theorem for smooth linear statistics}
\label{sec:CLT}

Our goal is to show how Theorem~\ref{thm:planar} relates to  numerous results on eigenvalues linear statistics of the the Gaussian $\beta$-ensemble.
Let $\rho(x) = \frac{\sqrt{1-x^2}}{\pi/2}\1_{|x| \le 1} $  be the semicircle density on $[-1,1]$. 
Let us recall that the eigenvalues of the Gaussian $\beta$-ensemble $\{\lambda_j\}_{j=1}^N$ have law \eqref{eq:GbE} and that linear statistics satisfy a central limit theorem.

\begin{theorem} \textnormal{(\cite[Theorem 2.4]{Johansson})} \label{thm:clt}
If $f:\R\to\R$ is a sufficiently smooth with $|f(x)| \le C (1+x^2)$, then as $N\to+\infty$
\begin{equation} \label{Johansson0}
\bigg( \sum_{j=1}^N f(x_j) -  N \int f(x) \rho(x) dx \bigg) \Rightarrow \left(\frac{2}{\beta}-1\right)\mathbf{m}(f) +\sqrt{\tfrac{2}{\beta}\Sigma(f)} \mathcal{N}
\end{equation}
where $\mathcal{N}$ is a standard Gaussian variable, $\Sigma(f) =\sum_{k=1}^{+\infty} k f_k^2$, $f_k$ are the Fourier--Chebyshev\footnote{The Chebyshev polynomials are defined by $T_k(\cos\theta) = \cos(k\theta)$ for any $k\ge 0$ and $\theta\in\R$. They form an orthogonal basis of $L^2([-1,1], \frac{dx}{\pi\sqrt{1-x^2}})$.} coefficients of $f$: for $k\ge 0$, 
\begin{equation} \label{def:m}
f_k =    \int_{-1}^1 \frac{f(x) T_k(x)}{\pi\sqrt{1-x^2}}dx   
\qquad\text{and}\qquad
\mathbf{m}(f)  =   \frac{f(1)+f(-1)}{4}  - \frac{f_0}{2}
\end{equation}
\end{theorem}

Theorem \ref{thm:clt} first appeared in the seminal work of \cite{Johansson} (see \cite{BPS95} for the law of large numbers) and they have been several improvements, e.g. \cite{BG13, Shcherbina13, BLS19, LLW19}.

First, let us explain where the mean in Theorem~\ref{thm:clt} comes from.
Recall that  $\E\varphi_N = \pi_N$ where $\pi_N$ is a Hermite polynomial of degree $N$ orthogonal with respect to the weight $e^{-2N x^2}$ on $\R$ (c.f. \eqref{eq:hermite}).
Let $z_1,\dots, z_N$ be the zeroes of $\pi_N$, and let $\log \pi_N(z)$ be the complex logarithm that is analytic in $\C \setminus (-\infty, 1]$ (the zeros $z_j\in [-1,1]$) and real valued on $\R.$ 

Let $\boldsymbol\gamma \subset \C$ be a fixed (simple oriented) contour around the cut $[-1,1]$ and suppose it crosses $(-\infty,-1)$ at a point $x^*.$  Parameterize this contour so that $\boldsymbol\gamma: [0,1] \to \C$ and so that $\boldsymbol\gamma(0)=\boldsymbol\gamma(1)=x^*$.  Then by an integration by parts,
\[
\sum_{j=1}^N f(z_j) 
= \frac{1}{2\pi\i}
\oint_{\boldsymbol\gamma}  f(z) \frac{\pi_N'(z)}{\pi_N(z)} dz
= \frac{1}{2\pi\i}
\int_{0}^1  f(\boldsymbol\gamma(t)) \frac{\pi_N'(\boldsymbol\gamma(t))}{\pi_N(\boldsymbol\gamma(t))} \boldsymbol\gamma'(t) dt
= 
Nf(x^*)
-
\frac{1}{2\pi\i}
\oint_{\boldsymbol\gamma}  f'(z) \log \pi_N(z) dz,
\]
where we have used argument principle to evaluate the boundary term 
\[
  \lim_{\varepsilon \to 0}
  \bigl(
  f(\gamma(1-\varepsilon))
  \log \pi_N(\gamma(1-\varepsilon))
  -f(\gamma(\varepsilon))
  \log \pi_N(\gamma(\varepsilon))
  \bigr)
  =2\pi\i N f(x^*).
\]

Recall the Hermite polynomial asymptotics \eqref{Piasymp} 
where $\gamma(z) =  \left(\frac{z+1}{z-1}\right)^{1/4}$ is analytic, non-zero on $\C\setminus[-1,1]$ and $g(z)=g_1(z)$ defined in \eqref{eq:gfunction} has its branch cut on $(-\infty,1]$. Hence, taking logarithm, we obtain for $z\in \boldsymbol\gamma$,
\[
 \log \pi_N(z)  = N g(z) + \log\left(\frac{\gamma(z) + \gamma(z)^{-1} }{2} \right)  + \O(N^{-1}) . 
\]
The error term of order $N^{-1}$ comes from the fact that we can choose the contour $\boldsymbol\gamma$ macroscopically separated from $[-1,1]$.

Moreover, we also have 
\[
\frac{1}{2\pi\i}
\oint_{\boldsymbol\gamma}  f(z) g'(z) dz
= 
f(x^*)
-\frac{1}{2\pi\i}
\oint_{\boldsymbol\gamma}  f'(z) g(z) dz.
\]

Thus for a function $f$ which is analytic in a simply connected neighborhood of  $[-1,1]$, 
\[ \begin{aligned}
\sum_{j=1}^N f(z_j)  &= Nf(x^*)-\frac{1}{2\pi\i} \oint_{\boldsymbol\gamma}  f'(z) \log \pi_N(z)  dz \\
& =  \frac{N}{2\pi\i} \oint_{\boldsymbol\gamma}  f(z) g'(z) dz
- \frac{1}{2\pi\i} \oint_{\boldsymbol\gamma}  f'(z) \log\left(\frac{\gamma(z) + \gamma(z)^{-1} }{2} \right) dz +\O(N^{-1})
\end{aligned}\]

Moreover from Step 2 in the proof of Proposition~\ref{prop:Hermite}, we have $ g'(z) = 4\lambda_-(1)$, so that by \eqref{def:root}, 
\[
\frac{1}{2\pi\i} \oint_{\boldsymbol\gamma}  f(z) g'(z) dz = - \frac{1}{\pi\i} \oint_{\boldsymbol\gamma}  f(z) \sqrt{z^2-1} dz =  \int f(x) \rho(x) dx
\]
where the last step follows from collapsing the contour $\boldsymbol\gamma$ on $[-1,1]$. 
Moreover, we also check that 
$-\frac{d}{dz} \log\left(\frac{\gamma(z) + \gamma(z)^{-1} }{2} \right)= \frac{z/2}{z^2-1} -\frac{1/2}{\sqrt{z^2-1}} $ and by a similar argument, we obtain
\[
\frac{1}{2\pi\i} \oint_{\boldsymbol\gamma}  f'(z) \log\left(\frac{\gamma(z) + \gamma(z)^{-1} }{2} \right) dz 
=  \frac{f(1)+f(-1)}{4} - \int_{-1}^1 \frac{f(x)}{2\pi\sqrt{1-x^2}} dx
=  \mathbf{m}(f)
\]
according to \eqref{def:m}.
This shows that 

\begin{equation}  \label{zero_hermite}
\sum_{j=1}^N f(z_j)  
=Nf(x^*) - \frac{1}{2\pi\i} \oint_{\boldsymbol\gamma}  f'(z) \log \pi_N(z)  dz 
= N \int f(x) \rho(x) dx - \mathbf{m}(f) , 
\end{equation}
which corresponds to the asymptotics \eqref{Johansson0} as $\beta\to+\infty$.
%Recall that we have $\frac{d}{dz} \log \pi_N(z) = \frac{\pi_N'(z)}{\pi_N(z)} \propto \frac{\pi_{N-1}(z)}{\pi_N(z)} $ and using the uniform asymptotics for the Hermite polynomials from the RH, it is possible to make the asymptotics \eqref{zero_hermite} rigorous... 

\medskip

We now proceed in a similar way to recover Theorem~\ref{thm:clt} for analytic test functions from our Theorem~\ref{thm:planar}. 
Namely, on an event of probability $1-e^{-c_\beta N^{\delta}}$, we have  for any function $f$ which is analytic in a neighborhood of  $[-1,1]$, 
\[ \begin{aligned}
&\sum_{j=1}^N f(\lambda_j)  = 
Nf(x^*)- \frac{1}{2\pi\i} \oint_{\boldsymbol\gamma}  f'(z) \log \varphi_N(z)  dz \\
& = 
Nf(x^*)- \frac{1}{2\pi\i} \oint_{\boldsymbol\gamma}  f'(z) \log \pi_N(z)  dz 
- \frac{\sqrt{2/\beta}}{2\pi\i} \oint_{\boldsymbol\gamma}  f'(z) \mathrm{W}(z) dz
+ \frac{1}{2\pi\i} \oint_{\boldsymbol\gamma}  f'(z) 
%\log\E\big[\exp\tfrac{\mathrm{W}(z)}{\sqrt{\beta}}\big] 
\E\big[\tfrac{\mathrm{W}(z)^2}{\beta}\big]
dz 
+\O\Big(\tfrac{\|f'\|_\infty}{N^{1/15}} \Big) .
\end{aligned}\]
Note that to evaluate the integral, it is crucial that the asymptotics from Theorem~\ref{thm:planar} are uniform and they imply that with probability at least $1-e^{-c_{\beta} N}$, all $\lambda_j \in [-1-\epsilon, 1+\epsilon]$ with $\epsilon\le  N^{-2\alpha}/2$.
%This follows from the large deviation principle for the empirical measure of the  Gaussian $\beta$-ensemble, see e.g. \cite[Section 2.6]{AGZ}. 
%Recall that $\mathrm{W}$ is a (centered) Gaussian analytic function in $\C\setminus [-1,1]$ with covariance kernel \eqref{eq:Wdef}. 
%In particular, we have $\log\E\big[\exp\frac{\mathrm{W}(z)}{\sqrt{\beta}}\big] = \frac{1}{2\beta} \E\big[ \mathrm{W}(z)^2\big]$ so that 
By integration by parts and using the asymptotics \eqref{zero_hermite}, 
\begin{equation} \label{Johansson1}
\sum_{j=1}^N f(\lambda_j)  = N \int f(x) \rho(x) dx - \mathbf{m}(f)
+  \frac{\sqrt{2/\beta}}{2\pi\i} \oint_{\boldsymbol\gamma}  f(z) \mathrm{W}'(z) dz
-   \frac{2/\beta}{2\pi\i} \oint_{\boldsymbol\gamma}  f(z) \E\big[ \mathrm{W}'(z)\mathrm{W}(z)\big] dz +\O\Big(\tfrac{\|f'\|_\infty}{N^{1/15}} \Big) .
\end{equation}
According to our Remark~\ref{rk:GAF}, we can represent $ \mathrm{W}(z) = \sum_{k=1}^{+\infty} \frac{\xi_k}{\sqrt{k}} J(z)^k$ and  $ \mathrm{W}'(z) = -  \sum_{k=1}^{+\infty}\frac{ \sqrt{k} \xi_k J(z)^k}{\sqrt{z^2-1}}$ 
for $z\in \C\setminus[-1,1]$ --  we used formula \eqref{J:logderivative} to compute the derivative. 
Then, since $\xi_1,\xi_2,\dots$ are i.i.d. standard Gaussians,  we have
 \[
  \E\big[\mathrm{W}(z)\mathrm{W}'(z)\big] = - \frac{1}{\sqrt{z^2-1}} \sum_{k=1}^{+\infty}  J(z)^{2k} .
 \]
For any $k\in\N$ and $x\in[-1,1]$, we have the boundary values: $J(x_\pm)^k = e^{\mp\i k \theta}$ if $x= \cos(\theta)$ with $\theta\ge 0$. This shows that $ \frac{J(x_+)^{k} + J(x_-)^{k}}{2} = T_k(x)$ where $T_k$ are the Chebyshev's polynomials of the first kind.  
 By using analyticity, to deform the contour $\boldsymbol\gamma$ to $[-1,1]$, this implies that 
 \[
 \frac{-1}{2\pi\i} \oint_{\boldsymbol\gamma}  f(z) \E\big[ \mathrm{W}'(z)\mathrm{W}(z)\big] dz 
 = \int_{-1}^1  f(x) \sum_{k=1}^{+\infty}  \frac{J(x_+)^{2k} + J(x_-)^{2k}}{2}   \frac{dx}{\pi\sqrt{1-x^2}}
 =\sum_{k=1}^{+\infty}  \int_{-1}^1  \frac{  f(x) T_{2k}(x) }{\pi\sqrt{1-x^2}} dx.
 \] 
If we expand $f = f_0 + 2  \sum_{k=1}^{+\infty}  f_k T_k$ in the Chebyshev's polynomial basis\footnote{The Chebyshev polynomials orthogonal with respect to the arcsine law:  
$\displaystyle \int_{-1}^1  T_k(x) T_{j}(x)   \frac{dx}{\pi\sqrt{1-x^2}} =  \frac{1+\1_{k=0}}{2} \delta_{k,j} $ for all $k,j\ge0$.},
%we obtain $\displaystyle \int_{-1}^1  f(x) T_{2k}(x)   \frac{dx}{\pi\sqrt{1-x^2}} = f_{2k} \frac{1+\1_{k=0}}{2}$
 according to \eqref{def:m} 
 and since $T_k(\pm 1) = (\pm 1)^k$ for all $k\ge 0$, we also have 
 \[
 \sum_{k=1}^{+\infty}  \int_{-1}^1   \frac{ f(x) T_{2k}(x) }{\pi\sqrt{1-x^2}} dx = \sum_{k=1}^{+\infty} f_{2k} = \frac{f(1)+f(-1)}{4} -  \frac{f_0}{2} =  \mathbf{m}(f). 
 \]
 This shows that 
 \begin{equation} \label{Johansson2}
   \frac{2/\beta}{2\pi\i} \oint_{\boldsymbol\gamma}  f(z)  \E\big[ \mathrm{W}'(z)\mathrm{W}(z)\big] dz =   -\frac{2}{\beta}\mathbf{m}(f).
 \end{equation}

By a similar argument, we also verify that
 \begin{align}
  \frac{1}{2\pi\i} \oint_{\boldsymbol\gamma}  f(z)  \mathrm{W}'(z) dz
&\notag  =  -  \sum_{k=1}^{+\infty} \sqrt{k}  \xi_k   \frac{1}{2\pi\i} \oint_{\boldsymbol\gamma}  f(z)   J(z)^k \frac{dz}{\sqrt{z^2-1}} \\
&\notag =  \sum_{k=1}^{+\infty}  \sqrt{k}  \xi_k   \int_{-1}^1  f(x) T_{k}(x)   \frac{dx}{\pi\sqrt{1-x^2}} \\
&\notag =   \sum_{k=1}^{+\infty}  \sqrt{k}  f_k \xi_k   \\
& \label{Johansson3}
 \overset{\rm law}{=} \sqrt{\Sigma(f)} \mathcal{N}
\end{align}
 where $\mathcal{N}$ is a standard Gaussian variable. 
 By combining \eqref{Johansson2}, \eqref{Johansson3} with the asymptotics \eqref{Johansson1}, we conclude that
 with  probability (at least) $1-e^{-c_\beta N^{\delta}}$, it holds  for any function $f$ which is analytic in a neighborhood of $[-1,1]$,
 \[
 \sum_{j=1}^N f(\lambda_j) =  
 N \int f(x) \rho(x) dx + \Big(\tfrac{2}{\beta}-1\Big) \mathbf{m}(f) +  \sqrt{\tfrac{2}{\beta}}\sum_{k=1}^{+\infty}  \sqrt{k}  f_k \xi_k  +\O\Big(\tfrac{\|f'\|_\infty}{N^{1/15}} \Big) .
  \]
 This completes our proof of Theorem~\ref{thm:clt}. In principle, we could also extend the results to functions $f\in \mathcal{C}^2$  in a neighborhood of $[-1,1]$ using the  Helffer-Sj\H{o}strand formula to construct an \emph{almost--analytic extension} of $f$. 
\printbibliography[heading=bibliography]%,title={References}]

%\bibliographystyle{alpha}
%\bibliography{hyperbolic}

\end{document}